\documentclass[11pt,a4paper]{article}
\usepackage[utf8]{inputenc}
\usepackage{amsmath}
\usepackage{amsfonts}
\usepackage{color}
\usepackage{amssymb}
\usepackage{mathrsfs}
\usepackage{esint}
\usepackage[colorinlistoftodos]{todonotes}
\usepackage[colorlinks=true]{hyperref}
\hypersetup{
	colorlinks=true,%
	citecolor=red,%
	filecolor=black,%
	linkcolor=blue,%
	urlcolor=blue
}
\usepackage[margin=1.1in]{geometry}

\usepackage[amsmath,thmmarks,hyperref]{ntheorem}
{
	\theoremstyle{nonumberplain}
	\theoremheaderfont{\bfseries}
	\theorembodyfont{\normalfont}
	\theoremsymbol{\mbox{$\Box$}}
	\newtheorem{pf}{Proof.}
}

\numberwithin{equation}{section}
\def\R{\mathbb{R}}
\def\B{\mathbb{B}}
\def\S{\mathbb{S}}

\def\e{\epsilon}
\newcommand{\ud}{\mathrm{d}}

\newtheorem{thm}{Theorem}[section]

\newtheorem{lem}{Lemma}[section]
\newtheorem{rem}{Remark}[section]

\newtheorem{pro}{\indent Proposition}[section]

\newtheorem*{thm A}{Theorem A}
\newtheorem*{thm B}{Theorem B}
\newtheorem*{thm C}{Theorem C}

\usepackage{upgreek}
\usepackage{graphicx}

\makeatletter

\newdimen\bibspace
\renewenvironment{thebibliography}[1]{%
	\section*{\refname 
		\@mkboth{\MakeUppercase\refname}{\MakeUppercase\refname}}%
	\list{\@biblabel{\@arabic\c@enumiv}}%
	{\settowidth\labelwidth{\@biblabel{#1}}%
		\leftmargin\labelwidth
		\advance\leftmargin\labelsep
		\itemsep\bibspace
		\parsep\z@skip     %
		\@openbib@code
		\usecounter{enumiv}%
		\let\p@enumiv\@empty
		\renewcommand\theenumiv{\@arabic\c@enumiv}}%
	\sloppy\clubpenalty4000\widowpenalty4000%
	\sfcode`\.\@m}
{\def\@noitemerr
	{\@latex@warning{Empty `thebibliography' environment}}%
	\endlist}

\makeatother

\makeatletter

\allowdisplaybreaks

\allowdisplaybreaks

\allowdisplaybreaks
\title{article}
\begin{document}
	\title{\bf Sharp quantitative  integral  inequalities for general conformally invariant extensions} 
	\date{\today}
	\author{\medskip Qiaohua  Yang\thanks{Q. Yang, School of Mathematics and Statistics, Wuhan University, Wuhan 430072, P. R. China. Email: qhyang.math@whu.edu.cn.},  Shihong Zhang\thanks{ S. Zhang, Yau Mathematical Sciences Center, Tsinghua university, Beijing 100084,  P. R. China. Email: shihong-zhang@tsinghua.edu.cn.}	}
	
	\renewcommand{\thefootnote}{\fnsymbol{footnote}}
	\maketitle

	{\noindent\small{\bf Abstract:  In this paper, we develop a refined analysis of hypergeometric functions to establish sharp quantitative integral inequalities for a general family of conformally invariant extension operators and their adjoints. Our results extend the recent work of Frank, Peteranderl, and Read \cite{Frank&Peteranderl&Read} to the full admissible parameter range under the natural index constraints. }

		\medskip 
		
		{{\bf $\mathbf{2020}$ MSC:} Primary 49J40,  33C05;  Secondary 47J20  47G10}
		
		\medskip 
		{\small{\bf Keywords:}
			Stability, conformally invariant extensions, hypergeometric functions}
		
		\tableofcontents	
		
		\section{Introduction}		
		Motivated by isoperimetric inequalities in Riemannian geometry,  sharp conformal isoperimetric-type inequalities have been extensively investigated. A prototypical two-dimensional example is the Carleman inequality \cite{Carleman}: for any $u \in C^{\infty}(\B^2)$  with $\Delta u \ge 0$,
		\begin{align*}
			\int_{\B^2} e^{2u}
			\le \frac{1}{4\pi}
			\left( \int_{\S^1} e^{u} \right)^2,
		\end{align*}
		with equality if and only if $u \equiv c$ or $u(x) = -2\log |x-x_0| + c$ for some $x_0 \in \mathbb{B}^2$.
		In 2008, Hang-Wang-Yan \cite{Hang&Wang&Yan} successfully generalized this
		isoperimetric inequality to higher dimensions. 
		We describe their result as follows. 
		Let $y=(y',0)\in \partial \R^n_{+}$,  and  for   given $f \in L^{\frac{2(n-1)}{n-2}}(\R^{n-1})$ and
		$g \in L^{\frac{2n}{n+2}}(\R^n_{+})$, introduce the convolution operators
		\begin{align*}
			P(f)(x)
			&=
			\int_{\R^{n-1}}
			\frac{x_n f(y')}{(x_n^2+|x'-y'|^2)^{\frac{n}{2}}}
			\,\ud y',\\
			T(g)(y')
			&=
			\int_{\R^n_{+}}
			\frac{x_n g(x)}{(x_n^2+|x'-y'|^2)^{\frac{n}{2}}}
			\,\ud x .
		\end{align*}
		Clearly, $P(f)$ is a harmonic extension of $f$; that is
		\begin{align*}
			\begin{cases}
				\displaystyle \quad~~~\Delta P(f)=0 \qquad&\mathrm{in}\qquad \R^{n}_{+},\\
				\displaystyle  P(f)(x',0)=|\S^{n-1}|f(x')/2\qquad&\mathrm{on}\qquad \partial\R^{n}_{+},         \end{cases}
		\end{align*}
		where $|\mathbb{S}^{n-1}|=\frac{2\pi^{n/2}}{\Gamma\left(\frac{n}{2}\right)}$
		denotes the surface area of the unit \((n-1)\)-sphere in \(\mathbb{R}^n\).
		Hang-Wang-Yan \cite{Hang&Wang&Yan} proved the following sharp  isoperimetric inequality and its dual inequality:
		\begin{align*}
			\|P(f)\|_{L^{\frac{2n}{n-2}}(\R^n_{+})}
			&\le c_{0,1}\,
			\|f\|_{L^{\frac{2(n-1)}{n-2}}(\R^{n-1})},\\
			\|T(g)\|_{L^{\frac{2(n-1)}{n}}(\R^{n-1})}
			&\le c_{0,1}\,
			\|g\|_{L^{\frac{2n}{n+2}}(\R^n_{+})},
		\end{align*}
		where  the optimal constant is
		\[
		c_{0,1}
		= 2^{-1}n^{-\frac{n-2}{2n}}|\S^{n-1}|^{1-\frac{n-2}{2n(n-1)}}.
		\]
		Recently, Frank, Peteranderl and Read \cite{Frank&Peteranderl&Read} demonstrated that the Hang--Wang--Yan inequalities and their dual are conformally invariant (see Appendix~A of \cite{Frank&Peteranderl&Read}).
		For the isoperimetric inequality for the polyharmonic extension operator with $T$-curvature on the boundary of the unit ball, see \cite{Chen Shibing,Chen&Zhang,Zhang}.
		
		A simple duality argument shows that  Hang--Wang--Yan inequalities are equivalent to 
		\begin{align*}
			\left|
			\int_{\R^n_{+}}
			\int_{\partial \R^n_{+}}
			\frac{x_n f(y')g(x)}{(x_n^2+|x'-y'|^2)^{\frac{n}{2}}}
			\ud x \ud y'
			\right|
			\le c_{0,1}\,
			\|f\|_{L^{\frac{2(n-1)}{n-2}}(\R^{n-1})}
			\|g\|_{L^{\frac{2n}{n+2}}(\R^n_{+})},       \end{align*}
		where $c_{0,1}$ is the same optimal constant as before. 	
		From a broader perspective, these inequalities can be viewed as special cases of a more general
		Hardy--Littlewood--Sobolev (HLS) inequality on the half-space.
		More precisely, for $f \in L^{p}(\R^{n-1})$ and $g \in L^{q'}(\R^n_{+})$, one has
		\begin{align}\label{Intro HLS}
			\left|
			\int_{\R^n_{+}}
			\int_{\partial \R^n_{+}}
			K_{\alpha,\beta}(x'-y',x_n)\,
			f(y')\, g(x)\,
			\ud x \ud y'
			\right|
			\le c_{\alpha,\beta}\,
			\|f\|_{L^{p}(\R^{n-1})}
			\|g\|_{L^{q'}(\R^n_{+})},
		\end{align}
		where the kernel is
		\begin{align*}
			K_{\alpha,\beta}(x)
			= \frac{x_n^{\beta}}{(|x'|^2+x_n^2)^{\frac{n-\alpha}{2}}}
		\end{align*}
		and  the exponents satisfy
		\begin{align*}
			\frac{1}{q'}+\frac{n-1}{np}
			= 1+\frac{\alpha+\beta-1}{n}.
		\end{align*}
		The sharp constant  $c_{\alpha,\beta}$ is defined by
		\begin{align}\label{sharp constant}
			c_{\alpha,\beta}
			= \sup \Biggl\{
			\left|
			\int_{\R^n_{+}}
			\int_{\partial \R^n_{+}}
			K_{\alpha,\beta}(x'-y',x_n)\,
			f(y')\, g(x)\,
			\ud x \ud y'
			\right| :
			\|f\|_{L^{p}(\R^{n-1})}
			= \|g\|_{L^{q'}(\R^n_{+})}
			=1
			\Biggr\}.
		\end{align}
		
		Taking $\alpha = 0$ and $\beta = 1$, the HLS inequality \eqref{Intro HLS} reduces to the inequality in the isoperimetric result of Hang--Wang--Yan. The borderline case $\alpha+\beta=1$ with $\alpha\in(2-n,1)$ was proved by  S. Chen \cite{Chen Shibing}. Afterwards, Dou and Zhu \cite{Dou&Zhu} established this HLS-type inequality for $\alpha\in(1,n)$ and $\beta=0$. More recently, Dou, Guo, and Zhu \cite{Dou&Guo&Zhu} developed a subcritical approach to prove \eqref{Intro HLS} for $\beta=1$ and $\alpha\in[2,n)$. All of these results  belong to  the regime $\alpha+\beta\geq 1$; the complementary case $\alpha+\beta<1$ turns out to be more delicate (see \cite{Gluck} or Lemma \ref{Bound d Lem}). Finally, Gluck \cite{Gluck} extended the subcritical approach  to establish the inequality for the full range of admissible parameters, including the delicate case $\alpha+\beta<1$.

		To present Gluck's integral inequalities, we first need to introduce the generalized conformally invariant extension operators. For $f\in L^{\frac{2(n-1)}{n+\alpha-2}}(\R^{n-1})$  and $g\in L^{\frac{2n}{n+\alpha+2\beta}}(\R^n_{+})$,  define 
		\begin{align*}
			E_{\alpha,\beta}(f)(x',x_n)=&\int_{\R^{n-1}}\frac{x_n^{\beta}f(y')}{(x_n^2+|x'-y'|^2)^{\frac{n-\alpha}{2}}}\ud y',\\
			R_{\alpha,\beta}(g)(y')=&\int_{\R^{n}_{+}}\frac{x_n^{\beta}g(x)}{(x_n^2+|x'-y'|^2)^{\frac{n-\alpha}{2}}}\ud x,            \end{align*}
		where $( \alpha,\beta)$ satisfies the constraints:
		\begin{align}\label{Index Condi}
			\displaystyle\beta\geq 0, \qquad  0<\alpha+\beta<n-\beta\qquad\mathrm{and} \qquad
			\displaystyle	\frac{n-\alpha-2\beta}{2n}+\frac{n-\alpha}{2(n-1)}<1.
		\end{align}
		In \cite{Gluck}, Gluck proved that the extension operator $E_{\alpha,\beta}$ exhibits the following boundary behavior (see \cite[Theorem 1.3]{Gluck}):
		\begin{itemize}
			\item $\alpha<1$. 
			\begin{align*}
				\lim_{x_n\to 0} x_{n}^{1-\alpha-\beta}E_{\alpha,\beta}f(x',x_n)=\pi^{\frac{n-1}{2}}
				\frac{\Gamma\!\left(\frac{1-\alpha}{2}\right)}
				{\Gamma\!\left(\frac{n-\alpha}{2}\right)} f(x')
			\end{align*}
			\item  $\alpha=1$.
			\begin{align*}
				\lim_{x_n\to 0} -x_{n}^{-\beta}(\log x_n)^{-1}E_{\alpha,\beta}f(x',x_n)=|\S^{n-1}|f(x')
			\end{align*}
			\item$\alpha>1$.
			\begin{align*}
				\lim_{x_n\to 0} x_{n}^{-\beta}E_{\alpha,\beta}f(x',x_n)=\frac{\Gamma\!\left(\frac{n-\alpha}{2}\right)}
				{2^{\alpha-1}\pi^{\frac{n-1}{2}}
					\Gamma\!\left(\frac{\alpha-1}{2}\right)}
				\,(-\Delta_{x'})^{\frac{1-\alpha}{2}} f(x').
			\end{align*}   
		\end{itemize}
		Clearly, the harmonic extension \(P(f)\) is a special case of \(E_{\alpha,\beta}\). Moreover, the  operator \(E_{\alpha,\beta}\) itself is also conformally invariant (see Appendix \ref{Appendix}). 
		
		In the sequel, we set 
		\begin{align*}
			p=p(\alpha,\beta)=\frac{2(n-1)}{n+\alpha-2}, \qquad q=q(\alpha,\beta)=\frac{2n}{n-\alpha-2\beta}
		\end{align*}
		and denote their dual exponents by
		\begin{align*}
			p'=\frac{p}{p-1}=\frac{2(n-1)}{n-\alpha},\qquad q'=\frac{q}{q-1}=\frac{2n}{n+\alpha+2\beta}.
		\end{align*}
		The last inequality in condition \eqref{Index Condi} is equivalent to $q>p$, or equivalently $q'<p'$. 
		We are now ready to state Gluck's theorem, which is equivalent to
		the HLS inequality \eqref{Intro HLS} on the upper
		half-space.
		\begin{thm A*}[\protect{Gluck \cite[Theorem 1.1 and Theorem 1.3]{Gluck}}]\label{thm A}
			Let $n \ge 3$ and suppose that $(\alpha,\beta)$ satisfies \eqref{Index Condi}. 
			Then there exists an optimal constant $c_{\alpha,\beta}$, given by \eqref{Intro const}, such that
			\begin{align}\label{Intro Gluck-1}
				\|E_{\alpha,\beta}(f)\|_{L^{q}(\R^{n}_{+})}
				\leq c_{\alpha,\beta}\,\|f\|_{L^{p}(\R^{n-1})},
			\end{align}
			and
			\begin{align}\label{Intro Gluck-2}				\|R_{\alpha,\beta}(g)\|_{L^{p'}(\R^{n-1})}
				\leq c_{\alpha,\beta}\,\|g\|_{L^{q'}(\R^{n}_{+})}.
			\end{align}
			Moreover, equality holds if and only if, up to a multiplicative constant,
			\[
			f = \bigl(\det \ud\Phi\big|_{\partial \R^{n}_{+}} \bigr)^{\frac{1}{p}},
			\qquad
			g = E^{q-1}_{\alpha,\beta}(f),
			\]
			where $\Phi$ is a conformal transformation from $\R^{n}_{+}$ onto itself and $\det \ud\Phi$ denotes the Jacobian of  $\Phi$.

		\end{thm A*}	
		\begin{rem}
			
			By Hölder’s inequality, if $f$ is a minimizer of \eqref{Intro Gluck-1}, then $E^{q-1}_{\alpha,\beta}(f)$ is, up to a multiplicative constant, a minimizer of \eqref{Intro Gluck-2}. Likewise, if $g$ is a minimizer of \eqref{Intro Gluck-2}, then $R^{p'-1}_{\alpha,\beta}(g)$ is, up to a multiplicative constant, a minimizer of \eqref{Intro Gluck-1}.
			
		\end{rem}
		To obtain an equivalent formulation on the unit ball, recall the conformal map
		\[
		I:(\mathbb{R}^n_{+},|\ud x|^2)\longrightarrow (\mathbb{B}^n,|\ud\xi|^2),\qquad
		\xi=I(x):=-e_n+\frac{2(x+e_n)}{|x+e_n|^2},
		\]
		where \(e_n=(0,\dots,0,1)\) is the unit vector in the \(x_n\)-direction. 		For $x=(x',x_n)\in\R^n_{+}$ and $y=(y',0)\in\partial\R^n_{+}$, 
		set $\xi=I(x)\in \B^n$ and $\eta=I(y)\in \S^{n-1}$. 
		Define the transformed boundary function
		\begin{align}\label{Introdu u}
			u(\eta)
			=\left(\frac{2}{|\eta+e_n|^2}\right)^{\frac{n+\alpha-2}{2}}f\circ I^{-1}(\eta)
		\end{align}
		and  the corresponding operator on $\B^n$ by
		\begin{align}\label{Introdu Q}
			Q_{\alpha,\beta}(u)(\xi)
			=\left(\frac{2}{|\xi+e_n|^2}\right)^{\frac{n-\alpha-2\beta}{2}}
			E_{\alpha,\beta}(f)\circ I^{-1}(\xi).
		\end{align}
		Then $Q_{\alpha,\beta}$ admits the integral representation (see Appendix \ref{Appendix})
		\[
		Q_{\alpha,\beta}(u)(\xi)
		=\int_{\S^{n-1}} H(\xi,\eta)\,u(\eta)\,
		\ud V_{\S^{n-1}}(\eta),
		\]
		where the kernel $H(\xi,\eta)$ is given by 
		\begin{align*}
			H(\xi,\eta)=\left(\frac{1-|\xi|^2}{2}\right)^{\beta}|\xi-\eta|^{\alpha-n}.
		\end{align*}
		Similarly,  define the dual operator
		\[
		S_{\alpha,\beta}(v)(\eta)
		=\int_{\B^n} H(\xi,\eta)\,v(\xi)\,\ud \xi.
		\]
		With these notations, the sharp constant $c_{\alpha,\beta}$ given in (\ref{sharp constant}) takes the form 
		\begin{align}\label{Intro const}
			c_{\alpha,\beta}
			=|\S^{n-1}|^{-\frac{1}{p}}
			\left\|
			\int_{\S^{n-1}}H(\cdot,\eta)\,
			\ud V_{\S^{n-1}}(\eta)
			\right\|_{L^q(\B^n)}.
		\end{align}
		
		Given a conformal transformation $\Psi:\B^{n}\to \B^{n}$ and two functions $u\in L^{p}(\S^{n-1})$ and $v\in L^{q'}(\B^n)$, we define their conformal transformation action by 
		\begin{align*}
			u_{\Psi}
			:=\left(\det \ud\Psi\big|_{\partial\B^n}\right)^{\frac{1}{p}}\,u\circ \Psi,
			\qquad
			v_{\Psi}
			:=\left(\det \ud\Psi\right)^{\frac{1}{q'}}\,v\circ \Psi,
		\end{align*}
		where $\det \ud\Psi$ stands for  the Jacobian of  $\Psi$ and $\det \ud\Psi\big|_{\partial\B^n}$ denote the Jacobian of $\Psi\big|_{\partial\B^n}$. Their precise expressions  can be found in Appendix \ref{Appendix}.  		
		
		Consequently, Theorem A can be restated on the unit ball as follows: 
		
		\begin{thm B*}\label{thm B}
			Let $n\geq 3$ and  suppose $(\alpha,\beta)$ satisfies  \eqref{Index Condi}. Then there exists a constant $c_{\alpha,\beta}$ such that 
			\begin{align*}
				\|Q_{\alpha,\beta}(u)\|_{L^{q}(\B^{n})}\leq c_{\alpha,\beta}\|u\|_{L^{p}(\S^{n-1}) }
			\end{align*}
			and 
			\begin{align*}
				\|S_{\alpha,\beta}(v)\|_{L^{p'}(\S^{n-1})}\leq c_{\alpha,\beta} \|v\|_{L^{q'}(\B^{n}) }.        \end{align*}
			Moreover, equality holds if and only if, up to a multiplicative constant,
			$
			u = \bigl(\det \ud\Psi\big|_{\partial \B^{n}} \bigr)^{\frac{1}{p}}$, $
			v =Q^{q-1}_{\alpha,\beta}(u),
			$
			where $\Psi$ is a conformal transformation from $\B^n$ to $\B^n$.

		\end{thm B*}
		\begin{rem}
			Using the conformal covariance of $Q_{\alpha,\beta}$ (see Appendix \ref{Appendix}), the second minimizer can be written as, up to a multiplicative constant, 
			$$
			Q_{\alpha,\beta}^{\,q-1}(1_{\Psi})
			=
			\left(\big(Q_{\alpha,\beta}(1)\big)_{\Psi}\right)^{q-1}
			=
			\left(\big(d_{\alpha,\beta}\big)_{\Psi}\right)^{q-1}=(d^{q-1}_{\alpha,\beta})_{\Psi},$$
			where $1_{\Psi}=\bigl(\det \ud\Psi\big|_{\partial \B^{n}} \bigr)^{\frac{1}{p}}$. Here, $d_{\alpha,\beta}\in L^q(\B^n), d_{\alpha,\beta}^{q-1}\in L^{q'}(\B^n)$ and the explicit formula for $d_{\alpha,\beta}$ is given in \eqref{Intro d-1}. 
			This shows that, even modulo the action of conformal transformations, the minimizer of the second inequality is, in general, non-constant; see the discussion below Theorem~\ref{Thm 1}. Such a phenomenon is in sharp contrast with the result of Hang-Wang-Yan \cite{Hang&Wang&Yan}.
		\end{rem}
		
		Very recently, Frank, Peteranderl and Read \cite{Frank&Peteranderl&Read} established a sharp quantitative   version of Hang--Wang--Yan inequalities:
		\begin{thm C}[\protect{Frank, Peteranderl and Read \cite[Theorem 1.3 and Theorem 1.4]{Frank&Peteranderl&Read}}]
			Let $n\geq 3$. There is a constant $c_0>0$ such that for all $u\not\equiv 0\in L^{\frac{2(n-1)}{n-2}}(\S^{n-1})$,  it holds
			\begin{align}\label{Intro FFR equ-1}
				&c^{\frac{2(n-1)}{n-2}}_{0,1}-\frac{\|Q_{0,1}u\|^{\frac{2(n-1)}{n-2}}_{L^{\frac{2n}{n-2}}(\B^n)}}{\|u\|^{\frac{2(n-1)}{n-2}}_{L^{\frac{2(n-1)}{n-2}}(\S^{n-1})}}
				\nonumber\\
				\geq c_0&\left(\inf_{\Psi,\lambda}\|\lambda|\S^{n-1}|^{\frac{n-2}{2(n-1)}}u_{\Psi}-1\|^{\frac{2(n-1)}{n-2}}_{L^{\frac{2(n-1)}{n-2}}(\S^{n-1})} +\|\lambda|\S^{n-1}|^{\frac{n-2}{2(n-1)}}u_{\Psi}-1\|^2_{L^2(\S^{n-1})}\right),
			\end{align}
			where $|\S^{n-1}|=\frac{2\pi^{\frac{n}{2}}}{\Gamma(\frac{n}{2})}$ is the surface area of the unit sphere. Moreover, for any $v\not\equiv 0\in L^{\frac{2n}{n+2}}(\B^n)$, we have
			\begin{align}\label{Intro FFR equ-2}				c^{\frac{2n}{n+2}}_{0,1}-\frac{\|S_{0,1}(v)\|^{\frac{2n}{n+2}}_{L^{\frac{2(n-1)}{n}}(\S^{n-1})}}{\|v\|^{\frac{2n}{n+2}}_{L^{\frac{2n}{n+2}}(\B^n)}}\geq c_0\inf_{\Psi,\lambda}\|\lambda|\S^{n-1}|^{\frac{n+2}{2n}}v_{\Psi}-1\|^{2}_{L^{\frac{2n}{n+2}}(\B^{n})}.
			\end{align}
			Here, $u_{\Psi}$ and $v_{\Psi}$ are given by 
			\begin{align*}
				u_{\Psi}
				:=\left(\det \ud\Psi\big|_{\partial\B^n}\right)^{\frac{1}{p(0,1)}}\,u\circ \Psi,
				\qquad
				v_{\Psi}
				:=\left(\det \ud\Psi\right)^{\frac{1}{q'(0,1)}}\,v\circ \Psi.
			\end{align*}
		\end{thm C}
		
		In the same work, the authors also showed that the exponents appearing in the distance functionals of (\ref{Intro FFR equ-1}) and (\ref{Intro FFR equ-2}) are optimal. 
		
		In recent years, stability analysis for geometric inequalities has attracted considerable attention. It can be traced back to a question posed  by Brezis and Lieb \cite{Brezis&Lieb}, which was later answered affirmatively by Bianchi and Egnell \cite{Bianchi&Egnell}. Since then, various new ideas have been introduced to establish stability results, including approaches based on concentration compactness, spectral analysis and rearrangement arguments, among others  \cite{Carlen,Carlen-1,Chen&Lu&Lu&Tang1,Chen&Lu&Lu&Tang2,Dolbeault&Esteban&Figalli&Frank&Loss,Frank&Peteranderl,Frank&Peteranderl-1,Figalli&Zhang,Konig&Peteranderl}. 
		For the existence of minimizers for the stability inequality and an upper bound for its sharp constant, we refer the reader to \cite{Tobias,Tobias-2}. Very recently, the authors  \cite{Gong&Yang&Zhang} established a stability type inequality for the reverse Sobolev inequality with a \textbf{explicit} sharp constant, albeit restricted to a specific range of parameters. Subsequently, T. König \cite{Tobias-3} confirmed the sharpness of this constant and  derived a complete stability inequality that holds for all admissible parameters.
		
		Among the works mentioned above, a particularly interesting result was established by Figalli and Zhang \cite{Figalli&Zhang}. In essence, they showed that the optimal exponent in the distance functional depends on $\max\{p,2\}$ for the $p$-Sobolev inequality with $p\in (1,n)$. This phenomenon is further supported by the work of Frank et al. \cite{Frank&Peteranderl&Read}, where—in contrast to earlier studies—the operator involved is nonlocal.
		
		In this paper, we extend the findings of \cite{Frank&Peteranderl&Read} from harmonic extensions to a broader class of conformally invariant extension operators. Our main theorem provides a family of such operators that are valid  for all $p\in (1,+\infty)$, and yields the sharp stability exponent for the distance functional ($L^p$ norm). This exponent reflects a degeneracy: it equals $2$ when $p\in(1,2)$ and becomes $p$ when $p\ge 2$, in accordance with the phenomenon observed by Figalli and Zhang~\cite{Figalli&Zhang}.
		A key distinction between our work and theirs lies in the range of $p$: 
		while their setting is restricted to $p\in(1,n)$, our exponent $p$ can be arbitrarily large. 
		For the Sobolev inequality on the specific cylinder, Frank~\cite{Frank} pointed out that the sharp exponent in the associated distance functional is $4$. 
		In contrast, our result shows that for the nonlocal operator under consideration, this degeneracy can in fact become arbitrarily large.

		\begin{thm}\label{Thm 1}
			For $n\geq 3$, assume that $( \alpha,\beta)$ satisfy \eqref{Index Condi}. Then the following sharp stability inequalities hold:
			\begin{itemize}
				\item[(1)]If $1\geq \alpha$, i.e., $p\geq2$, then there exists a constant $c_0>0$ such that for all $u\in L^{p}(\S^{n-1})$ and $u\not\equiv 0$, 
				\begin{align*}
					c^p_{\alpha,\beta}-\frac{\|Q_{\alpha,\beta}u\|^p_{L^q(\B^n)}}{\|u\|^p_{L^p(\S^{n-1})}}
					\geq c_0\left(\inf_{\Psi,\lambda}\|\lambda|\S^{n-1}|^{\frac{1}{p}}u_{\Psi}-1\|^p_{L^p(\S^{n-1})} +\|\lambda|\S^{n-1}|^{\frac{1}{p}}u_{\Psi}-1\|^2_{L^2(\S^{n-1})}\right)  . \end{align*} 
				\item[(2)]  If $n>\alpha>1$, i.e, $1<p<2$, then  there exists a constant $c_0>0$ such that for all $u\in L^{p}(\S^{n-1})$ and $u\not\equiv 0$, 
				\begin{align*}
					c^p_{\alpha,\beta}-\frac{\|Q_{\alpha,\beta}u\|^p_{L^q(\B^n)}}{\|u\|^p_{L^p(\S^{n-1})}}
					\geq c_0\inf_{\Psi,\lambda}\|\lambda|\S^{n-1}|^{\frac{1}{p}}u_{\Psi}-1\|^2_{L^p(\S^{n-1})}   . \end{align*}     \end{itemize}
			
		\end{thm}
		\begin{rem}
			The index constraints \eqref{Index Condi} are necessary and sharp, as shown in \cite{Gluck}; see also Remark~\ref{Rem 2}, Lemma~\ref{Compact lem0}. The main difficulty of this work lies in establishing a sharp quantitative integral inequality for \textbf{all} pairs $(\alpha,\beta)$ satisfying \eqref{Index Condi}, while simultaneously dealing with the essentially nonlocal nature of the operator. 
			
		\end{rem}

		Although Frank et al.~\cite{Frank&Peteranderl&Read} developed a systematic approach for the harmonic extension operator and its dual operator, new obstructions arise in our setting. 
		Let \begin{align}\label{define H}
			\mathcal{H}=\mathrm{span}\{1,\xi_1,\xi_2,\cdots,\xi_n\}\subset L^2(\mathbb{S}^{n-1}),
		\end{align}
		where $\xi_i$ denotes the $i$-th coordinate function in $\R^n$. 
		For the harmonic extension operator $Q_{0,1}$, we have 
		\begin{align*}
			Q_{0,1}(1)=\frac{|\S^{n-1}|}{2}
			\qquad\mathrm{and}\qquad 
			Q_{0,1}(\xi_i)=\frac{|\S^{n-1}|}{2}\,\xi_i,\quad \mathrm{for}\quad i=1,2,\cdots,n.
		\end{align*}
		However, for general $( \alpha,\beta)$ satisfying \eqref{Index Condi}, our operators are essentially nonlocal. Define 
		\[
		d_{\alpha,\beta}(\xi)=Q_{\alpha,\beta}(1)(\xi).
		\]
		Then (see (\ref{d Def}))
		\begin{align}\label{Intro d-1}
			d_{\alpha,\beta}(\xi)
			=\frac{\pi^{\frac{n}{2}}2^{1-\beta}}{\Gamma\!\left(\frac{n}{2}\right)}
			(1-|\xi|^2)^{\beta+\alpha-1}
			F\!\left(\frac{n+\alpha}{2}-1,\frac{\alpha}{2}; \frac{n}{2}; |\xi|^2\right),
		\end{align}
		where $F$ is the hypergeometric function (see Subsection~\ref{Sec 2.2}). 
		Note that $d_{\alpha,\beta}$ is constant if and only if $\alpha=0$ and $\beta=1$; moreover, it may even become  unbounded when $\alpha+\beta<1$. 
		This phenomenon forces us to consider the weighted operator $d_{\alpha,\beta}^{\frac{q-2}{2}}Q_{\alpha,\beta}$ rather than $Q_{\alpha,\beta}$ as in~\cite{Frank&Peteranderl&Read}. 
		To prove Theorem~\ref{Thm 1}, we need to overcome the following two difficulties: 
		\begin{itemize}
			\item[(1)] How to establish the spectral gap lemma for the operator $d_{\alpha,\beta}^{\frac{q-2}{2}}Q_{\alpha,\beta}$ and obtain an effective upper bound for its optimal constant?
			\item[(2)] How to establish the compactness of the operator $d_{\alpha,\beta}^{\frac{q-2}{2}}Q_{\alpha,\beta}$?
		\end{itemize}
		
		To address (1), we expand $\|d_{\alpha,\beta}^{\frac{q-2}{2}}Q_{\alpha,\beta}(\varphi)\|^2_{L^2(\B^n)}$ in spherical harmonics for $\varphi\in L^2(\S^{n-1})$, and use properties of hypergeometric functions to compare adjacent coefficients in three different parameter regimes. 
		The desired upper bound for the optimal constant in the spectral gap lemma is obtained from the second variation of Gluck-type inequalities. 
		
		To handle (2), we introduce a new approach to prove compactness by establishing decay of the coefficients in the expansion of $\|d_{\alpha,\beta}^{\frac{q-2}{2}}Q_{\alpha,\beta}(\varphi)\|^2_{L^2(\B^n)}$. 
		For more details, see Section~\ref{Sec 2}.

		Regarding the stability of the dual operator $S_{\alpha,\beta}$, Frank et al.~\cite{Frank&Peteranderl&Read} observed in the case $\alpha=0$ and $\beta=1$ that the dual quantitative inequality \eqref{Intro FFR equ-2} does not appear to follow from the abstract stability theory developed by Carlen~\cite{Carlen-1}; see \cite[p.~3, Remark (e)]{Frank&Peteranderl&Read}. For general $\alpha$ and $\beta$, our setting exhibits two noteworthy phenomena. First, stability for the dual operator $S_{\alpha,\beta}$ occurs only in a single parameter regime. This contrasts with the situation for $Q_{\alpha,\beta}$ in Theorem~\ref{Thm 1}, where the analysis naturally splits into two distinct regimes. Second, the minimizer is $\tilde{\mathbf{1}}$ (up to the action of conformal transformations), which is generally non-constant.
		These structural differences help explain why Carlen’s duality-based framework~\cite{Carlen-1} does not yield stability for $S_{\alpha,\beta}$ in the present context, and our results further support the observation of Frank et al.~\cite{Frank&Peteranderl&Read}. Indeed, the mapping
		$
		S_{\alpha,\beta}: L^{q'}(\B^n)\to L^{p'}(\S^{n-1})
		$
		falls into the range $q'<2$, so the associated distance functional must be measured with exponent $2$. 
		This behavior is consistent with the observation of Figalli and Zhang~\cite{Figalli&Zhang}.

		\begin{thm}\label{Thm 2}
			For \( n \geq 3 \), assume that $(\alpha,\beta)$ satisfy \eqref{Index Condi}. There exists a constant $c_0>0$ such that for all  $v\in L^{q'}(\mathbb{B}^{n})$ and $v\not\equiv 0$,
			\[
			c^{q'}_{\alpha,\beta}-\frac{\|S_{\alpha,\beta}(v)\|^{q'}_{L^{p'}(\mathbb{S}^{n-1})}}{\|v\|^{q'}_{L^{q'}(\mathbb{B}^n)}}
			\geq c_0\inf_{\Psi,\lambda}\bigl\|\lambda|\mathbb{S}^{n-1}|^{\frac{1}{q'}}v_{\Psi}-\tilde{\mathbf{1}}\bigr\|^{2}_{L^{q'}(\mathbb{B}^{n})},
			\]
			where $\tilde{\mathbf{1}}(\xi)=\left(\fint_{\B^n}d^q_{\alpha,\beta}\right)^{\frac{1-q}{q}}d^{q-1}_{\alpha,\beta}(\xi)$  and $\fint_{\B^n}$ denotes the normalized integral over $\B^n$.
		\end{thm}
		
		\begin{rem}
			The exponents appearing in the distance functionals of Theorems \ref{Thm 1} and \ref{Thm 2} are also optimal; see Subsection \ref{Sec 3.3}.
		\end{rem}
		

		For the dual operator, when $\alpha=0$ and $\beta=1$, one has $Q_{0,1}(\mathcal{H})=\mathcal{H}$, where $\mathcal{H}$ is defined in (\ref{define H}). 
		Based on this identity, the authors in \cite{Frank&Peteranderl&Read} employed a duality argument to derive a spectral gap lemma. In our setting, however, this approach is no longer feasible.
		There is, nevertheless, an essential difference between the dual case and the one considered earlier.  If $\varphi\in L^2(\S^{n-1})$, we can expand $\varphi=\sum_{l=0}^{+\infty}a_lY_l$, and then
		\begin{align*}
			\varphi \perp \mathcal{H} \quad \Leftrightarrow\quad a_0=a_1=0.
		\end{align*}
		This characterization, however, has no direct analogue for functions $\psi\in L^2(\B^n)$. In particular, the condition $\psi\perp\tilde{\mathcal{H}}$ (see Lemma \ref{Gap Lem 2}) does not force the corresponding coefficients in the spherical harmonic decomposition to vanish. Our main idea is therefore to adapt the technique developed previously, relying on refined estimates for the previous approach. The details are provided in Subsection \ref{Sec 3.1}. A useful simplification for dual operator is that $\tilde{d}_{\alpha,\beta}=S_{\alpha,\beta}(\tilde{\mathbf{1}})$ is constant, which substantially streamlines the argument.
		
		\textbf{Notation}: Throughout this paper, we write $A \lesssim B$ to indicate that there exists a positive constant $C$, depending only on the relevant parameters, such that $A \le C B$. We use $\eta$ to denote a point on $\S^{n-1}$ and $\varphi$ to denote a function on $\S^{n-1}$.        
		Similarly, we use $\xi$ to denote a point in $\B^n$ and $\psi$ to denote a function on $\B^n$.

		This paper is organized as follows. In Section~\ref{Sec 2}, we establish Theorem~\ref{Thm 1} through a three-step argument.  
		Subsection~\ref{Sec 2.1} collects the preliminary material required for the concentration–compactness framework.
		In Subsection~\ref{Sec 2.2}, we carry out a delicate analysis of hypergeometric functions, which yields both a spectral
		gap lemma and the compactness of a certain weighted operator. 
		These ingredients are then combined in Subsection~\ref{Sec 2.3} to complete the proof of Theorem~\ref{Thm 1}. Section~\ref{Sec 3} addresses the dual problem. In Subsection~\ref{Sec 3.1}, we gather several basic properties of the dual operator, including its variational characterization and a concentration--compactness principle obtained via a duality argument. 
		Subsection~\ref{Sec 3.2} adapts the hypergeometric techniques introduced in Section \ref{Sec 2} to derive an analogous spectral gap lemma.
		Finally, in Subsection~\ref{Sec 3.3}, we complete the proof of Theorem~\ref{Thm 2} and explain the optimality of the exponent in the distance functional.

		\section{Stability for general integral isoperimetric inequality}		\label{Sec 2}
		\subsection{Concentration compactness}\label{Sec 2.1}
		In this subsection we establish  the concentration-compactness for our problem. Since the argument is standard, we only present the key ingredients of the  proof.   The following two Lemmas correspond to the basic estimates in \cite[P. 60]{Hang&Wang&Yan}.
		\begin{lem}\label{Basic esti Lem}
			For any $1\leq s\leq t\leq \infty$, the following basic inequalities hold:
			\begin{itemize}
				\item
				If $1+\frac{1}{t}<\frac{n-\alpha}{n-1}+\frac{1}{s}$, then
				\begin{align*}       \|E_{\alpha,\beta}(f)(\cdot,x_n)\|_{L^t(\R^{n-1})}\lesssim x_n^{(n-1)\left(1+\frac{1}{t}-\frac{1}{s}\right)-(n-\alpha-\beta)}\|f\|_{L^s(\R^{n-1})};
				\end{align*}
				\item  Assume $\mathrm{supp}(f)\subset B_{R}(0)$. Then
				\begin{align*}
					|E_{\alpha,\beta}(f)(x)|\lesssim \frac{x_n^{\beta}}{(((|x'|-R)^{+})^2+x_n^2)^{\frac{n-\alpha}{2}}}\|f\|_{L^1(\R^{n-1})};
				\end{align*}
				\item  Assume $\mathrm{supp}(f)\subset B^{c}_{R}(0)$ and $\frac{1}{s}>\frac{\alpha-1}{n-1}$. Then
				\begin{align*}
					\|E_{\alpha,\beta}(f)(\cdot,x_n)\|_{L^{\infty}(B_{R/2})}\lesssim x_n^{\beta}R^{(\alpha-1)-\frac{n-1}{s}}\|f\|_{L^s(\R^{n-1})}.
				\end{align*}
				Moreover, if $\frac{1}{s}>\frac{\alpha+\beta-1}{n-1}$, then 
				\begin{align*}
					\|E_{\alpha,\beta}(f)\|_{L^{\infty}(B^{+}_{R/2})}\lesssim R^{\alpha+\beta-1-\frac{n-1}{s}}\|f\|_{L^s(\R^{n-1})},
				\end{align*}
				where $B_R=\{x'\in \R^{n-1}:|x'|<R\}$ and $B^+_R=\{(x',x_n)\in \R^n_{+}:\sqrt{|x'|^2+x_n^2}<R\}$.
			\end{itemize}
			
		\end{lem}
		
		\begin{pf}
			By Young's inequality, for $1+\frac{1}{t}=\frac{1}{r}+\frac{1}{s} $, $r>\frac{n-1}{n-\alpha}$, we have 
			\begin{align*}
				\|E_{\alpha,\beta}(f)(\cdot,x_n)\|_{L^t(\R^{n-1})}\lesssim& \|E_{\alpha,\beta}(\cdot,x_n)\|_{L^r(\R^{n-1})}\|f\|_{L^s(\R^{n-1})}\\
				\lesssim& x_n^{\frac{n-1}{r}-(n-\alpha-\beta)}\|f\|_{L^s(\R^{n-1})}\\
				=&x_n^{(n-1)\left(1+\frac{1}{t}-\frac{1}{s}\right)-(n-\alpha-\beta)}\|f\|_{L^s(\R^{n-1})}.
			\end{align*}
			The second inequality is trivial by the direct estimate.
			For the last inequality, since $\frac{1}{s}>\frac{\alpha-1}{n-1}$, we can   get 
			\begin{align*}
				\|E_{\alpha,\beta}(f)(\cdot,x_n)\|_{L^{\infty}(B_{R/2})}\lesssim& x_n^{\beta}\left(\int_{\R^{n-1}\backslash B_R}\frac{1}{(x_n^2+|x'-y'|^2)^{\frac{(n-\alpha)s}{2(s-1)}}}\ud y'\right)^{\frac{s-1}{s}}\|f\|_{L^s(\R^{n-1})}\\
				\lesssim& x_n^{\beta}R^{(\alpha-1)-\frac{n-1}{s}}\|f\|_{L^s(\R^{n-1})}.    \end{align*}
			If $\frac{1}{s}>\frac{\alpha+\beta-1}{n-1}$, we can also estimate			\begin{align*}
				\|E_{\alpha,\beta}(f)\|_{L^{\infty}(B^{+}_{R/2})}\lesssim& \left(\int_{\R^{n-1}\backslash B_R}\frac{1}{(x_n^2+|x'-y'|^2)^{\frac{(n-\alpha-\beta)s}{2(s-1)}}}\ud y'\right)^{\frac{s-1}{s}}\|f\|_{L^s(\R^{n-1})}\\
				\lesssim& \left(\int_{\R^{n-1}\backslash B_R}\frac{1}{|y'|^{\frac{(n-\alpha-\beta)s}{s-1}}}\ud y'\right)^{\frac{s-1}{s}}\|f\|_{L^s(\R^{n-1})}\\				\lesssim& R^{(\alpha+\beta-1)-\frac{n-1}{s}}\|f\|_{L^s(\R^{n-1})}.    \end{align*}
		\end{pf}
		\begin{rem}
			If $t=q=\frac{2n}{n-\alpha-2\beta}$ and $s=p=\frac{2(n-1)}{n+\alpha-2}$ with $( \alpha,\beta)$ satisfying \eqref{Index Condi}, then all the assumptions in the above lemma are fulfilled.
		\end{rem}
		
		The following lemma cannot be proved directly by the method of  \cite{Hang&Wang&Yan}; instead, some necessary modifications are required. For the reader’s convenience, we present a detailed proof below. 
		\begin{lem}\label{Basic esti Lem 2}
			Let $\varphi\in C^{1}_{c}(\mathbb{R}^{n-1})$ and let $( \alpha,\beta)$ satisfy \eqref{Index Condi}. For any fixed $0<\theta<\min\!\bigl\{1,\beta+\frac{1}{q}\bigr\}$, we have
			\begin{align}\label{Sec 2.1 Lem 2 euq-a}
				\|\varphi(x')E_{\alpha,\beta}(f)(x)-E_{\alpha,\beta}(\varphi(x')f)(x)\|_{L^q(\R^{n-1})}\lesssim x_n^{\theta-\frac{1}{q}}\|\varphi\|_{C^{1}(\R^{n-1})}  \|f\|_{L^p(\R^{n-1})}.
			\end{align}
		\end{lem}
		
		\begin{pf}
			Since $\varphi\in C^{1}_{c}(\R^{n-1})$, for any fixed $\theta\in (0,1]$ and any $x',y'\in \R^{n-1}$, we have
			\begin{align*}
				|\varphi(x')-\varphi(y')|\leq 2\|\varphi\|_{C^{1}(\R^{n-1})}|x'-y'|^{\theta}.
			\end{align*}
			Indeed, if $|x'-y'|\leq 1$,  the mean value theorem gives
			\begin{align*}
				|\varphi(x')-\varphi(y')|
				\leq \|\nabla\varphi\|_{L^{\infty}(\R^{n-1})}|x'-y'|
				\leq  \|\nabla\varphi\|_{L^{\infty}(\R^{n-1})}|x'-y'|^{\theta},
			\end{align*}
			while if $|x'-y'|\geq 1$, 
			\begin{align*}
				|\varphi(x')-\varphi(y')|
				\leq 2\|\varphi\|_{L^{\infty}(\R^{n-1})}
				\leq 2\|\varphi\|_{L^{\infty}(\R^{n-1})}|x'-y'|^{\theta}.
			\end{align*}
			Denote
			$
			L=(\varphi(x')E_{\alpha,\beta}(f)(x)-E_{\alpha,\beta}(\varphi(x')f)(x)).
			$
			Then
			\begin{align*}
				L=\int_{\R^{n-1}}\frac{x_n^{\beta}(\varphi(x')-\varphi(y'))f(y')}{(x_n^2+|x'-y'|^2)^{\frac{n-\alpha}{2}}}\ud y'.
			\end{align*}
			Consequently,
			\begin{align*}
				|L|
				\leq& 2\|\varphi\|_{C^{1}(\R^{n-1})}\int_{\R^{n-1}}\mathscr{Q}(x,y')|f(y')|\ud y'\\
				\leq&2\|\varphi\|_{C^{1}(\R^{n-1})}\left(\int_{\R^{n-1}}\mathscr{Q}^s(x,y')|f(y')|^p\ud y' \right)^{1/q}\left(\int_{\R^{n-1}} \mathscr{Q}^s\right)^{\frac{1}{s}-\frac{1}{q}}\left(\int_{\R^{n-1}} |f|^p\right)^{\frac{1}{p}-\frac{1}{q}},
			\end{align*}
			where $1+1/q=1/p+1/s$ and
			\begin{align*}
				\mathscr{Q}(x,y')=\frac{x_n^{\beta }|x'-y'|^{\theta}}{(x_n^2+|x'-y'|^2)^{\frac{n-\alpha}{2}}}.
			\end{align*}
			Since $\theta < \beta + \frac{1}{q}$, we have
			\begin{align*}
				s=\frac{1}{1+\frac{1}{q}-\frac{1}{p}}=\frac{n-1}{n-\alpha-\beta-\frac{1}{q}}>\frac{n-1}{n-\alpha-\theta}.
			\end{align*}
			Therefore, 			
			\begin{align*}
				\left(\int_{\R^{n-1}}\mathscr{Q}^s(x,y')\ud x'\right)^{1/s}=&x_n^{\theta+\alpha+\beta-n+\frac{n-1}{s}}\left(
				\int_{\R^{n-1}}\frac{|x'|^{s\theta}}{(1+|x'|^2)^{\frac{n-\alpha}{2}s}}\ud x'\right)^{1/s}\\
				\lesssim& x_n^{\theta-\frac{1}{q}}\left(\int_{0}^{1}r^{n-2+s\theta}\ud r+\int_{0}^{1}r^{n-2-s(n-\alpha-\theta)}\ud r\right)^{1/s}\\
				\lesssim& x_n^{\theta-\frac{1}{q}},
			\end{align*}
			where we used $n-\alpha-\beta=\frac{1}{q}+\frac{n-1}{s}$. 
			The desired estimate \eqref{Sec 2.1 Lem 2 euq-a} follows.
		\end{pf}

		Using the estimates in Lemma \ref{Basic esti Lem}, Lemma \ref{Basic esti Lem 2} and following the strategy of \cite[Theorem 3.1]{Hang&Wang&Yan}, we obtain the following concentration-compactness proposition. We omit the proof as it follows a standard argument that is well-known in the field.

		\begin{pro}\label{Prop 1}
			Let $\{u_i\}_{i=1}^{+\infty}$ be a sequence with the properties:
			\begin{align*}
				\|u_i\|_{L^p(\S^{n-1})}\to 1\qquad\mathrm{and}\qquad \|Q_{\alpha,\beta}(u_i)\|_{L^q(\B^n)}\to c_{\alpha,\beta}			\end{align*}
			Then, we have
			\begin{align*}
				\inf_{\Psi,\lambda\in\{\pm 1\}}\|\lambda|\S^{n-1}|^{\frac{1}{p}}(u_i)_{\Psi}-1\|_{L^p(\S^{n-1})}\to 0.
			\end{align*}
		\end{pro}

		\subsection{Spectral gaps}\label{Sec 2.2}
		\subsubsection{Spectral gaps with weight}
		
		We first recall some basic notation concerning the  hypergeometric function; for a complete treatment, see \cite{Hypergeometry,Gradshteyn&Ryzhik}.
		Given  real numbers $a,b,c$,  define
		\begin{align*}
			F\left(a,b;c; z \right)=\sum_{r=0}^{+\infty}\frac{(a)_{r}(b)_{r}}{(c)_{r}}\frac{z^r}{r!}\qquad\mathrm{for}\qquad |z|<1.
		\end{align*}
		where $ c \neq 0, -1, -2, \dots $ and $ (a)_k $ denotes the rising Pochhammer symbol:	
		$$
		(a)_{0}=1,\;(a)_{k}=a(a+1)\cdots(a+k-1), \;k\geq1.
		$$	Clearly, we have $F(a,b;c;z)=F(b,a;c;z)$. If $\mathrm{Re}~ c>\mathrm{Re}~ b>0$, it has the integral representation (cf. \cite[p.59]{Hypergeometry} ):
		\begin{align}\label{Pre 1.8}
			F(a,b;c;z)=\frac{\Gamma(c)}{\Gamma(b)\Gamma(c-b)}\int_{0}^{1}t^{b-1}(1-t)^{c-b-1}(1-tz)^{-a}\ud t.
		\end{align}
		Below we list several classical identities  that will be used in the subsequent analysis:
		\begin{itemize}	
			\item Transformation (cf. \cite[p.~1018: 9.131-1]{Gradshteyn&Ryzhik}):
			\begin{equation}\label{Pre 1.4}
				\begin{split}
					F(a,b;c;z)=(1-z)^{c-a-b} F(c-a,c-b;c;z).
				\end{split}
			\end{equation}
			
			\item     Gauss' recursion functions (cf. \cite{Gradshteyn&Ryzhik}, p.~1019, 9.137(11) and (17))
			\begin{align}\label{Pre 1.6}
				&F(a + 1, b; c; z)-F(a, b; c; z)=\frac{b}{c}z F\left( a + 1, b + 1; c + 1; z \right);\\
				\label{Pre 1.7}
				&F(a, b; c; z) = \frac{b}{c}\, F(a, b + 1; c + 1; z) +\frac{c - b}{c}\, F(a, b; c + 1; z).
			\end{align}		
		\end{itemize}
		
		We begin by recalling the integral representation lemma that expands $Q_{\alpha,\beta}$
		into a series involving harmonic polynomials and hypergeometric functions (see 
		\cite[Theorem 3.1]{yang} or  \cite[Theorem 4.1]{Gong&Yang&Zhang}).

		\begin{lem}\label{Sec 2 Spherical thm}  
			Let  $f\in C^{\infty}(\S^{n-1})$ and set $s=\frac{n-1}{2}+\gamma$ with $\gamma>0$. Define 
			\begin{align}\label{Sec 5 Integral formu solution}
				u(\xi)=\pi^{-\frac{n-1}{2}}\frac{\Gamma\left(\frac{n-1}{2}+\gamma\right)}{\Gamma(\gamma)}\int_{\S^{n-1}}  \left(\frac{1-|\xi|^2}{2|\xi-\eta|^2}\right)^sf(\eta) \ud V_{{\S^{n-1}}}(\eta).
			\end{align}
			Assume that $f$ has a spherical harmonic expansion $f = \sum\limits_{l=0}^{+\infty} Y_l$, where $Y_l\in \mathscr{H}_l$ (the space of  harmonic polynomials with degree $l$ in $\R^n$).
			Then $u$ can be expanded as
			\begin{align}\label{Sec 5 Series formu solution}
				u(\xi)=\left(\frac{1-r^2}{2}\right)^{n-1-s} \frac{\Gamma\left(\gamma+\frac{1}{2}\right)}{\Gamma(2\gamma)}\sum_{l=0}^{+\infty}\varphi_l(r^2)r^{l}Y_l,
			\end{align}
			where $r=|\xi|$ and
			\begin{align*}
				\varphi_l(t)
				=&\frac{\Gamma\left(l+\gamma+\frac{n-1}{2}\right)}{\Gamma(l+\frac{n}{2})}F\left(l+\frac{n-1}{2}-\gamma,\frac{1}{2}-\gamma; l+\frac{n}{2}; t\right).
			\end{align*}
		\end{lem}
		
		Recall that 
		\begin{align*}
			Q_{\alpha,\beta}(\varphi)(\xi)=\left(\frac{1-|\xi|^2}{2}\right)^{\beta}\int_{\S^{n-1}}  \frac{\varphi(\eta)}{|\xi-\eta|^{n-\alpha}} \ud V_{{\S^{n-1}}}(\eta).
		\end{align*}
		Suppose that $\varphi$ admits a spherical harmonic expansion $\varphi = \sum\limits_{l=0}^{+\infty} a_lY_l$; for simplicity we assume the spherical harmonics are normalized so that
		$$\int_{\mathbb{S}^{n-1}}|Y_l|^2=1,\quad l\in \mathbb{N}.$$ 
		Set
		\begin{align*}
			\gamma=\frac{n-\alpha}{2},\quad s=\frac{n}{2}+\gamma=n-\frac{\alpha}{2}.
		\end{align*}
		Applying Lemma \ref{Sec 2 Spherical thm} with this choice of parameters and using the duplication formula
		\begin{align}\label{Gamma 2z}
			\Gamma(2z)=\frac{2^{2z-1}}{\sqrt{\pi}}\Gamma(z)\Gamma\left(z+\tfrac{1}{2}\right),
		\end{align}
		we  rewrite $Q_{\alpha,\beta}$ as follows:
		\begin{align}\label{Gap lem equ a}
			Q_{\alpha,\beta}(\varphi)(\xi)
			=&\left(\frac{1-|\xi|^2}{2}\right)^{\frac{2\beta+\alpha-n}{2}}
			\int_{\S^{n-1}}  \left(\frac{1-|\xi|^2}{2|\xi-\eta|^2}\right)^{\frac{n-\alpha}{2}}
			\varphi(\eta) \ud V_{{\S^{n-1}}}(\eta)\nonumber\\
			=&\left(\frac{1-|\xi|^2}{2}\right)^{\beta+\alpha-1}\frac{\pi^{\frac{n}{2}}2^{\alpha}}{\Gamma(\frac{n-\alpha}{2})}
			\sum_{l=0}^{+\infty}\varphi_l(r^2)r^{l}a_lY_l,
		\end{align}
		where now
		\begin{align}\label{Sec 5 Series varphi}
			\varphi_l(t)
			=&\frac{\Gamma\left(l+\frac{n-\alpha}{2}\right)}{\Gamma(l+\frac{n}{2})}
			F\left(l+\frac{n+\alpha}{2}-1,\frac{\alpha}{2}; l+\frac{n}{2}; t\right).
		\end{align}
		We recall  the function $d_{\alpha,\beta}(\xi)=Q_{\alpha,\beta}(1)(\xi)$. By (\ref{Gap lem equ a}) and (\ref{Sec 5 Series varphi}), it can be written explicitly as
		\begin{align}\label{d Def}
			d_{\alpha,\beta}(\xi)=\frac{\pi^{\frac{n}{2}}2^{1-\beta}}{\Gamma(\frac{n}{2})}(1-|\xi|^2)^{\beta+\alpha-1}
			F\left(\frac{n+\alpha}{2}-1,\frac{\alpha}{2}; \frac{n}{2}; |\xi|^2\right).
		\end{align}	
		
		Now, we begin to present the following spectral gap lemma. It plays a central role in the proof of Theorem \ref{Thm 1}.

		\begin{lem}\label{Gap Lem}
			Assume $\varphi\in L^2(\S^{n-1})\cap \mathcal{H}^{\perp}$, where $\mathcal{H}$ is defined in (\ref{define H}). Then
			\begin{align*}
				\frac{1}{K_{n,\alpha,\beta}}
				=\sup_{\varphi\in L^2(\S^{n-1})\cap \mathcal{H}^{\perp},\ \varphi\not\equiv 0 }
				\frac{\int_{\B^n}d^{q-2}_{\alpha,\beta}Q^2_{\alpha,\beta}(\varphi)}{\|\varphi\|^2_{L^2(\S^{n-1})}},
			\end{align*}
			where the constant $K_{n,\alpha,\beta}$ is given by 
			\begin{align}\label{Gap lem K const}
				\frac{1}{K_{n,\alpha,\beta}}
				=&\frac{\pi^{n}2^{1-2\beta}}{\Gamma^2(\frac{n-\alpha}{2})} \frac{\Gamma\left(\frac{ n - \alpha+4}{2}\right)^{2}}{\Gamma\left( \frac{n+4}{2}\right)^{2}}\times
				\nonumber\\
				&\int_0^1 d^{q-2}_{\alpha,\beta}(\sqrt{t}) (1 - t)^{2(\beta + \alpha-1)} t^{n/2 +  1} F\left(\frac{n+\alpha}{2}+1,\frac{\alpha}{2}; 2+\frac{n}{2}; t\right)^{2} \mathrm{d}t.
			\end{align}
		\end{lem}
		
		The proof of Lemma \ref{Gap Lem} relies on the following monotonicity lemma. Its proof is highly nontrivial,  because different ranges of the parameter
		$\alpha$ require distinct strategies for estimating the  underlying hypergeometric functions.
		\begin{lem}\label{lem:hypergeo_monotonicity}
			Let \(t \in (0,1)\) and let \(l\) be a non-negative integer.
			\begin{enumerate}
				\item For \(0 < \alpha < n\), the quantity
				\[
				\frac{\Gamma\!\left(l + \frac{n - \alpha}{2}\right)^2}
				{\Gamma\!\left(l + \frac{n}{2}\right)^2}
				F\left(l+\frac{n+\alpha}{2}-1,\frac{\alpha}{2}; l+\frac{n}{2}; t\right)^2
				\]
				is strictly decreasing in \(l\).
				
				\item For \(\alpha \leq 0\),
				the quantity
				\[
				\frac{\Gamma\!\left(l + \frac{n - \alpha}{2}\right)^2}
				{\Gamma\!\left(l + \frac{n}{2}\right)^2} t^{l}
				F\left(l+\frac{n+\alpha}{2}-1,\frac{\alpha}{2}; l+\frac{n}{2}; t\right)^2
				\]
				is strictly decreasing in \(l\).
			\end{enumerate}
		\end{lem}

		\begin{pf}
			The proof  is divided into three cases according to the sign of \(\alpha\).
			
			\medskip
			\noindent\textbf{Case 1:} $\alpha\geq1$. For $t\in (0,1)$, it suffices to verify
			\begin{align}\label{a2.13}
				\frac{l+\frac{n-\alpha}{2}}{l+\frac{n}{2}}F\left(l+\frac{n+\alpha}{2},\frac{\alpha}{2}; l+1+\frac{n}{2}; t\right)
				<F\left(l-1+\frac{n+\alpha}{2},\frac{\alpha}{2}; l+\frac{n}{2}; t\right).
			\end{align}
			Using the definition of the hypergeometric function, we expand the left-hand side as
			\begin{align*}
				&\frac{l+\frac{n-\alpha}{2}}{l+\frac{n}{2}}F\left(l+\frac{n+\alpha}{2},\frac{\alpha}{2}; l+1+\frac{n}{2}; t\right)\\
				=&\frac{l+\frac{n-\alpha}{2}}{l+\frac{n}{2}}\sum_{j=0}^{\infty}\frac{(l+ \frac{n + \alpha}{2})_j (\frac{\alpha}{2})_j}{(l + 1+\frac{n}{2})_j j!}t^{j}\\
				=& \sum_{j=0}^{\infty}\frac{(l+ \frac{n + \alpha}{2}-1)_j (\frac{\alpha}{2})_j}{(l + \frac{n}{2})_j j!}t^{j}\times \frac{(l+\frac{n-\alpha}{2})(l-1+\frac{n+\alpha}{2}+j)}{(l-1+\frac{n+\alpha}{2})(l+\frac{n}{2}+j)}.
			\end{align*}
			Since for $\alpha\geq 1$,
			\begin{align*}
				&\frac{(l+\frac{n-\alpha}{2})(l-1+\frac{n+\alpha}{2}+j)}{(l-1+\frac{n+\alpha}{2})(l+\frac{n}{2}+j)}-1\\
				=&-\frac{1}{(l-1+\frac{n+\alpha}{2})(l+\frac{n}{2}+j)}\left\{\frac{\alpha}{2}(l-1+\frac{n+\alpha}{2})+j(\alpha-1)\right\}<0,
			\end{align*}
			we obtain
			\begin{align*}
				&\frac{l+\frac{n-\alpha}{2}}{l+\frac{n}{2}}F\left(l+\frac{n+\alpha}{2},\frac{\alpha}{2}; l+1+\frac{n}{2}; t\right)\\
				<& \sum_{j=0}^{\infty}\frac{(l+ \frac{n + \alpha}{2}-1)_j (\frac{\alpha}{2})_j}{(l + \frac{n}{2})_j j!}t^{j}
				=F\left(l-1+\frac{n+\alpha}{2},\frac{\alpha}{2}; l+\frac{n}{2}; t\right),
			\end{align*}
			which proves \eqref{a2.13}.
			
			\medskip
			\noindent\textbf{Case 2:} $0< \alpha<1$. Using the  transformation formula
			\begin{align}\label{a2.15}
				F\left( l-1+ \frac{n +\alpha}{2}, \frac{\alpha }{2}; l + \frac{n}{2}; t \right)
				=(1-t)^{1-\alpha}F\left(1-\frac{\alpha}{2}, l + \frac{n-\alpha}{2};  l + \frac{n}{2}; t \right),
			\end{align}
			it remains to check
			\begin{align}\label{a2.14}
				\frac{l+\frac{n-\alpha}{2}}{l+\frac{n}{2}}F\left(1-\frac{\alpha}{2}, l +1+ \frac{n-\alpha}{2};  l +1+ \frac{n}{2}; t \right)
				<F\left(1-\frac{\alpha}{2}, l + \frac{n-\alpha}{2};  l + \frac{n}{2}; t \right).
			\end{align}
			Noting that $1-\frac{\alpha}{2}>0$ and $\alpha>0$, we expand
			\begin{align*}
				&\frac{l+\frac{n-\alpha}{2}}{l+\frac{n}{2}}F\left(1-\frac{\alpha}{2}, l +1+ \frac{n-\alpha}{2};  l +1+ \frac{n}{2}; t \right)\\
				=&\frac{l+\frac{n-\alpha}{2}}{l+\frac{n}{2}}\sum_{j=0}^{\infty}\frac{(1-\frac{\alpha}{2})_j(l+1+ \frac{n - \alpha}{2})_j }{(l + 1+\frac{n}{2})_j j!}t^j\\
				=&\sum_{j=0}^{\infty}\frac{(1-\frac{\alpha}{2})_j(l+ \frac{n - \alpha}{2})_j }{(l + \frac{n}{2})_j j!}t^j\times \frac{l+\frac{n-\alpha}{2}+j}{l+\frac{n}{2}+j}\\
				<&\sum_{j=0}^{\infty}\frac{(1-\frac{\alpha}{2})_j(l+ \frac{n - \alpha}{2})_j }{(l + \frac{n}{2})_j j!}t^j=
				F\left(1-\frac{\alpha}{2}, l + \frac{n-\alpha}{2};  l + \frac{n}{2}; t \right),
			\end{align*}
			which establishes \eqref{a2.14}.
			
			\medskip
			\noindent\textbf{Case 3:} $\alpha\leq 0$.
			By  \eqref{a2.15}, the monotonicity for this range is equivalent to
			\begin{align}\label{a2.16}
				\frac{\left(l +\frac{ n - \alpha}{2}\right)^{2}}{\left(l  + \frac{n}{2}\right)^{2}}tF\left(1-\frac{\alpha}{2}, l + 1+\frac{n-\alpha}{2};  l + 1+\frac{n}{2}; t \right)^2
				<F\left(1-\frac{\alpha}{2}, l + \frac{n-\alpha}{2};  l + \frac{n}{2}; t \right)^{2}.
			\end{align}
			
			We first derive the identity
			\begin{align}\label{}
				\nonumber &F\left(1-\frac{\alpha}{2}, l + \frac{n-\alpha}{2};  l + \frac{n}{2}; t \right)^{2}-\frac{\left(l +\frac{ n - \alpha}{2}\right)^{2}}{\left(l  + \frac{n}{2}\right)^{2}}tF\left(1-\frac{\alpha}{2}, l + 1+\frac{n-\alpha}{2};  l + 1+\frac{n}{2}; t \right)^2\\
				\nonumber =
				&\frac{ \frac{\alpha}{2}}{l + \frac{n}{2}}
				\frac{l + \frac{n-\alpha}{2}}{l + \frac{n}{2}}tF\left(1-\frac{\alpha}{2}, l + \frac{n-\alpha}{2};  l+1 + \frac{n}{2}; t \right)F\left(1-\frac{\alpha}{2}, l +1+ \frac{n-\alpha}{2};  l+1 + \frac{n}{2}; t \right)\\
				&+F\left(-\frac{\alpha}{2}, l + \frac{n-\alpha}{2};  l + \frac{n}{2}; t \right)F\left(1-\frac{\alpha}{2}, l + \frac{n-\alpha}{2};  l + \frac{n}{2}; t \right).
				\label{2.19}
			\end{align}
			Indeed, applying \eqref{Pre 1.6} and \eqref{Pre 1.7} yields
			\begin{align*}
				&
				F\left(1-\frac{\alpha}{2}, l + \frac{n-\alpha}{2};  l + \frac{n}{2}; t \right)-
				F\left(-\frac{\alpha}{2}, l + \frac{n-\alpha}{2};  l + \frac{n}{2}; t \right)\\
				=&\frac{l +\frac{ n - \alpha}{2}}{l  + \frac{n}{2}}tF\left(1-\frac{\alpha}{2}, l + 1+\frac{n-\alpha}{2};  l + 1+\frac{n}{2}; t \right)
			\end{align*}
			and
			\begin{align*}
				&F\left(1-\frac{\alpha}{2}, l + \frac{n-\alpha}{2};  l + \frac{n}{2}; t \right)\\
				=&\frac{l +\frac{ n - \alpha}{2}}{l  + \frac{n}{2}}F\left(1-\frac{\alpha}{2}, l + 1+\frac{n-\alpha}{2};  l + 1+\frac{n}{2}; t \right)+\frac{\frac{ \alpha}{2}}{l  + \frac{n}{2}}
				F\left(1-\frac{\alpha}{2}, l + \frac{n-\alpha}{2};  l + 1+\frac{n}{2}; t \right).
			\end{align*}
			Multiplying the two expressions gives
			\begin{align*}
				&\left\{F\left(1-\frac{\alpha}{2}, l + \frac{n-\alpha}{2};  l + \frac{n}{2}; t \right)-
				F\left(-\frac{\alpha}{2}, l + \frac{n-\alpha}{2};  l + \frac{n}{2}; t \right)\right\}
				F\left(1-\frac{\alpha}{2}, l + \frac{n-\alpha}{2};  l + \frac{n}{2}; t \right)\\
				=&\frac{\left(l +\frac{ n - \alpha}{2}\right)^{2}}{\left(l  + \frac{n}{2}\right)^{2}}tF\left(1-\frac{\alpha}{2}, l + 1+\frac{n-\alpha}{2};  l + 1+\frac{n}{2}; t \right)^2+\\
				&\frac{ \frac{\alpha}{2}}{l + \frac{n}{2}}
				\frac{l + \frac{n-\alpha}{2}}{l + \frac{n}{2}}tF\left(1-\frac{\alpha}{2}, l + \frac{n-\alpha}{2};  l+1 + \frac{n}{2}; t \right)F\left(1-\frac{\alpha}{2}, l +1+ \frac{n-\alpha}{2};  l+1 + \frac{n}{2}; t \right),
			\end{align*}
			which proves (\ref{2.19}).

			Observe that
			\begin{align*}
				0<&-\frac{ \frac{\alpha}{2}}{l + \frac{n}{2}}F\left(1-\frac{\alpha}{2}, l + \frac{n-\alpha}{2};  l+1 + \frac{n}{2}; t \right)\\			=&\sum_{j=0}^{\infty}\frac{(-\frac{\alpha}{2})_{j}( l + \frac{n-\alpha}{2})_j}{(l+ \frac{n}{2})_{j}j!}\frac{j-\frac{\alpha}{2}}{l+\frac{n}{2}+j}t^{j}\\
				<&\sum_{j=0}^{\infty}\frac{(-\frac{\alpha}{2})_{j}( l + \frac{n-\alpha}{2})_j}{(l+ \frac{n}{2})_{j}j!}t^{j}
				=F\left(-\frac{\alpha}{2}, l + \frac{n-\alpha}{2};  l + \frac{n}{2}; t \right)
			\end{align*}
			and
			\begin{align*}
				0<&\frac{l + \frac{n-\alpha}{2}}{l + \frac{n}{2}}tF\left(1-\frac{\alpha}{2}, l +1+ \frac{n-\alpha}{2};  l+1 + \frac{n}{2}; t \right)\\
				=&\sum_{j=0}^{\infty}\frac{(1-\frac{\alpha}{2})_{j}( l + \frac{n-\alpha}{2})_{j+1}}{(l+ \frac{n}{2})_{j+1}j!}t^{j+1}=
				\sum_{j=0}^{\infty}\frac{(1-\frac{\alpha}{2})_{j+1}( l + \frac{n-\alpha}{2})_{j+1}}{(l+ \frac{n}{2})_{j+1}(j+1)!}\frac{j+1}{j+1-\frac{\alpha}{2}}t^{j+1}\\
				<&\sum_{j=0}^{\infty}\frac{(1-\frac{\alpha}{2})_{j+1}( l + \frac{n-\alpha}{2})_{j+1}}{(l+ \frac{n}{2})_{j+1}(j+1)!}t^{j+1}\\
				=&F\left(1-\frac{\alpha}{2}, l + \frac{n-\alpha}{2};  l + \frac{n}{2}; t \right)-1
				<F\left(1-\frac{\alpha}{2}, l + \frac{n-\alpha}{2};  l + \frac{n}{2}; t \right).
			\end{align*}
			Consequently,
			\begin{align}
				&-\frac{ \frac{\alpha}{2}}{l + \frac{n}{2}}
				\frac{l + \frac{n-\alpha}{2}}{l + \frac{n}{2}}t
				F\left(1-\frac{\alpha}{2}, l + \frac{n-\alpha}{2};  l+1 + \frac{n}{2}; t \right)
				F\left(1-\frac{\alpha}{2}, l +1+ \frac{n-\alpha}{2};  l+1 + \frac{n}{2}; t \right)\nonumber\\
				<&F\left(-\frac{\alpha}{2}, l + \frac{n-\alpha}{2};  l + \frac{n}{2}; t \right)
				F\left(1-\frac{\alpha}{2}, l + \frac{n-\alpha}{2};  l+ \frac{n}{2}; t \right). \label{2.21}
			\end{align}
			Combining \eqref{2.21} with \eqref{2.19} yields
			\begin{align*}
				F\left(1-\frac{\alpha}{2}, l + \frac{n-\alpha}{2};  l + \frac{n}{2}; t \right)^{2}
				>\frac{\left(l +\frac{ n - \alpha}{2}\right)^{2}}{\left(l  + \frac{n}{2}\right)^{2}}tF\left(1-\frac{\alpha}{2}, l + 1+\frac{n-\alpha}{2};  l + 1+\frac{n}{2}; t \right)^2,
			\end{align*}
			which is exactly   \eqref{a2.16}. This completes the proof for \(\alpha \le 0\).
			
		\end{pf}
		
		\textbf{Proof of Lemma \ref{Gap Lem}}.
		Let $\varphi=\sum_{l=0}^{+\infty}a_lY_l$ be the  spherical harmonic expansion of $\varphi$. Using \eqref{Gap lem equ a}, we compute
		\begin{align}\label{Gap lem A0}
			\int_{\B^n}d^{q-2}_{\alpha,\beta}Q^2_{\alpha,\beta}(\varphi)
			=&\frac{\pi^{n}2^{2\alpha}}{\Gamma^2(\frac{n-\alpha}{2})}\sum_{l=0}^{+\infty}a_l^2\int_{0}^{1}d^{q-2}_{\alpha,\beta}(r)\left(\frac{1-r^2}{2}\right)^{2(\beta+\alpha-1)}
			r^{n+2l-1}\varphi^2_l(r^2)\ud r\nonumber\\
			=&\frac{\pi^{n}2^{1-2\beta}}{\Gamma^2(\frac{n-\alpha}{2})}\sum_{l=0}^{+\infty}a_l^2\int_{0}^{1}
			d^{q-2}_{\alpha,\beta}(\sqrt{t})\left(1-t\right)^{2(\beta+\alpha-1)}t^{n/2+l-1}\varphi^2_l(t)\ud t\nonumber\\
			=&\frac{\pi^{n}2^{1-2\beta}}{\Gamma^2(\frac{n-\alpha}{2})} \sum_{l=0}^{+\infty} a_l^2 A^{\{\alpha,\beta\}}_{l},
		\end{align}
		where
		\begin{align}\label{Gap lem Al}
			A^{\{\alpha,\beta\}}_{l}=& \int_{0}^{1}
			d^{q-2}_{\alpha,\beta}(\sqrt{t})\left(1-t\right)^{2(\beta+\alpha-1)}t^{n/2+l-1}\varphi^2_l(t)\ud t\nonumber\\
			=&\frac{\Gamma\left(l +\frac{ n - \alpha}{2}\right)^{2}}{\Gamma\left(l + \frac{n}{2}\right)^{2}}
			\int_0^1 d^{q-2}_{\alpha,\beta}(\sqrt{t})(1 - t)^{2(\beta + \alpha-1)} t^{n/2 + l - 1} F\left(l+\frac{n+\alpha}{2}-1,\frac{\alpha}{2}; l+\frac{n}{2}; t\right)^{2} \mathrm{d}t.
		\end{align}
		By Lemma \ref{lem:hypergeo_monotonicity},  we have
		\begin{align}\label{Gap lem claim}
			A^{\{\alpha,\beta\}}_{l+1}<A^{\{\alpha,\beta\}}_{l},\qquad l\in \mathbb{N}.
		\end{align}
		
		With the monotonicity \eqref{Gap lem claim} established, the conclusion follows from the orthogonality condition. Indeed, the condition $\varphi\in\mathcal{H}^\perp$ is equivalent to $a_0=a_1=0$. Hence
		\begin{align*}
			\int_{\B^n}d^{q-2}_{\alpha,\beta}Q^2_{\alpha,\beta}(\varphi)
			&=\frac{\pi^{n}2^{1-2\beta}}{\Gamma^2(\frac{n-\alpha}{2})} \sum_{l=2}^{+\infty} a_l^2 A^{\{\alpha,\beta\}}_{l}\\
			&\leq\frac{\pi^{n}2^{1-2\beta}}{\Gamma^2(\frac{n-\alpha}{2})} A^{\{\alpha,\beta\}}_{2}\|\varphi\|_{L^2(\mathbb{S}^{n-1})}^2.
		\end{align*}
		This shows that the supremum is attained by the second-order spherical harmonics and yields the explicit constant in \eqref{Gap lem K const}. The proof of Lemma \ref{Gap Lem} is thereby completed.
		\begin{rem}
			Consider the special case \(\alpha = 0\) and \(\beta = 1\). A straightforward calculation gives
			\[
			A^{\{0,1\}}_{2} = \frac{2}{n+4}\; d_{0,1}^{\,q-2},
			\qquad\text{where}\qquad
			d_{0,1} = \frac{\pi^{\frac{n}{2}}}{\Gamma\!\left(\frac{n}{2}\right)}.
			\]
			Consequently, the sharp constant becomes
			\[
			K_{n,0,1} = \frac{4(n+4)}{|\mathbb{S}^{n-1}|^2}\; d_{0,1}^{\,2-q}.
			\]
			Up to a multiplicative factor arising from a different normalization of the spherical measure in \cite{Frank&Peteranderl&Read}, this constant coincides with the one obtained in Lemma~2.7 of that work.
			
		\end{rem}
		
		Now, we begin to give the effective upper bound of $K^{-1}_{n,\alpha,\beta}$. This bound will play a crucial role in the proofs   of Proposition \ref{Limit Prop} and Proposition \ref{Loc Bound Prop}.

		\begin{lem}\label{Gap Inequ}
			Let $K_{n,\alpha,\beta}$ be defined by \eqref{Gap lem K const}. Then
			\begin{align}\label{Gap Inequ-equ}
				\frac{1}{K_{n,\alpha,\beta}}<\frac{p-1}{q-1}\frac{\int_{\B^n}d^q_{\alpha,\beta}}{|\S^{n-1}|}.
			\end{align}
		\end{lem}
		
		\begin{pf}
			Fix $\e>0$ and $\varphi\in C^{\infty}(\S^{n-1})$. Consider the functional 
			\begin{align*}
				f(\e)=\|Q_{\alpha,\beta}(1+\e\varphi)\|_{L^q(\B^{n})}\|1+\e\varphi\|^{-1}_{L^p(\S^{n-1})}.
			\end{align*}
			Differentiating directly gives
			\begin{align*}
				f^{'}(\e)=&\int_{\B^n}Q_{\alpha,\beta}(\varphi)\left(d_{\alpha,\beta}+\e Q_{\alpha,\beta}(\varphi) \right)^{q-1}
				\left(\int_{\B^n}\left(d_{\alpha,\beta}+\e Q_{\alpha,\beta}(\varphi) \right)^{q} \right)^{-1}f(\e)\\
				-&\int_{\S^{n-1}}\varphi(1+\e\varphi)^{p-1}\left(\int_{\S^{n-1}}(1+\e\varphi)^{p} \right)^{-1}f(\e).
			\end{align*}
			Since \(f(0)=c_{\alpha,\beta}\) and the constant function \(1\) is an extremal, the point \(\e = 0\) is a local maximizer of \(f\).  Hence $f'(0)=0$, and we obtain
			\begin{align*}
				0=f^{'}(0)=c_{\alpha,\beta}\int_{\B^n}d^{q-1}_{\alpha,\beta}Q_{\alpha,\beta}(\varphi)\left(\int_{\B^n}d^q_{\alpha,\beta}\right)^{-1}
				-c_{\alpha,\beta}\fint_{\S^{n-1}}\varphi,
			\end{align*}
			which implies
			\begin{align}\label{Variation 1}
				\frac{\int_{\B^n}d^{q-1}_{\alpha,\beta}Q_{\alpha,\beta}(\varphi)}{\int_{\B^n}d^q_{\alpha,\beta}}=\fint_{\S^{n-1}}\varphi.
			\end{align}
			
			On the other hand, since $f'(0)=0$, the second derivative at $\e=0$ satisfies $f''(0)\le 0$. A straightforward computation gives
			\begin{align*}
				0\geq \frac{f^{''}(0)}{c_{\alpha,\beta}}=& (q-1)\frac{\int_{\B^n}d^{q-2}_{\alpha,\beta}Q^2_{\alpha,\beta}(\varphi)}{\int_{\B^n}d^q_{\alpha,\beta}}
				-q  \left(\frac{\int_{\B^n}d^{q-1}_{\alpha,\beta}Q_{\alpha,\beta}(\varphi) }{\int_{\B^n}d^q_{\alpha,\beta}}\right)^2\\
				-&(p-1)\fint_{\S^{n-1}}\varphi^2+p\left(\fint_{\S^{n-1}}\varphi\right)^2.
			\end{align*}
			In particular, if $\fint_{\S^{n-1}}\varphi=0$, then by \eqref{Variation 1} we also have $\int_{\B^n}d^{q-1}_{\alpha,\beta}Q_{\alpha,\beta}(\varphi)=0$, and the above inequality reduces to
			\begin{align}\label{Variation 2}
				\fint_{\S^{n-1}}\varphi^2\geq  \frac{q-1}{p-1}\,
				\frac{\int_{\B^n}d^{q-2}_{\alpha,\beta}Q^2_{\alpha,\beta}(\varphi)}{\int_{\B^n}d^q_{\alpha,\beta}}
				\qquad\mathrm{for}\qquad  \int_{\S^{n-1}}\varphi=0.
			\end{align}
			By \eqref{Gap lem A0} and \eqref{Gap lem claim}, the operator $d_{\alpha,\beta}^{\frac{q-2}{2}}Q_{\alpha,\beta}$ is bounded from $L^2(\S^{n-1})$ to $L^2(\B^{n})$, and hence it is continuous. 
			Therefore \eqref{Variation 2} remains valid for every
			$\varphi \in L^2(\mathbb{S}^{n-1})$. Now recall that if $\varphi=\sum_{l=1}^{+\infty}a_lY_l\in L^2(\S^{n-1})$, then
			\begin{align*}
				\int_{\B^n}d^{q-2}_{\alpha,\beta}Q^2_{\alpha,\beta}(\varphi)
				=&\frac{\pi^{n}2^{1-2\beta}}{\Gamma^2(\frac{n-\alpha}{2})} \sum_{l=0}^{+\infty} a_l^2 A^{\{\alpha,\beta\}}_{l},
			\end{align*}
			and $A^{\{\alpha,\beta\}}_{l}$ is strictly decreasing in $l$. Therefore,
			\begin{align}\label{a2.25}
				\frac{\pi^{n}2^{1-2\beta}}{\Gamma^2(\frac{n-\alpha}{2})} A^{\{\alpha,\beta\}}_{1}
				=\sup_{\int_{\S^{n-1}}\varphi=0,\ \varphi\not\equiv 0\in L^2(\S^{n-1})}
				\frac{\int_{\B^n}d^{q-2}_{\alpha,\beta}Q^2_{\alpha,\beta}(\varphi)}{\int_{\S^{n-1}}\varphi^2}.
			\end{align}
			Combing (\ref{Variation 2}) and (\ref{a2.25}) yields
			\begin{align*}
				\frac{\pi^{n}2^{1-2\beta}}{\Gamma^2(\frac{n-\alpha}{2})} A^{\{\alpha,\beta\}}_{1}\leq\frac{p-1}{q-1}\frac{\int_{\B^n}d^q_{\alpha,\beta}}{|\S^{n-1}|}.
			\end{align*}
			Consequently,  
			\begin{align*}
				\frac{1}{K_{n,\alpha,\beta}}=\frac{\pi^{n}2^{1-2\beta}}{\Gamma^2(\frac{n-\alpha}{2})} A^{\{\alpha,\beta\}}_{2}<\frac{\pi^{n}2^{1-2\beta}}{\Gamma^2(\frac{n-\alpha}{2})} A^{\{\alpha,\beta\}}_{1}\leq \frac{p-1}{q-1}\frac{\int_{\B^n}d^q_{\alpha,\beta}}{|\S^{n-1}|},
			\end{align*}
			which  completes the proof.
		\end{pf}
		\begin{rem}\label{Rem 2}
			Here we emphasize that the inequality \eqref{Gap Inequ-equ} is genuinely nontrivial. When $\alpha=0$, a direct computation shows that \eqref{Gap Inequ-equ} is equivalent to
			\begin{align*}
				B\left(\frac{n}{2}+2,q(\beta-1)+1\right)<\frac{p-1}{q-1}B\left(\frac{n}{2},q(\beta-1)+1\right),
			\end{align*}
			where $B(\cdot,\cdot)$ denotes the Beta function.
			This inequality is, in turn, equivalent to the condition $q(\beta-1)+1>0$, which is exactly \eqref{Index Condi}; see also \eqref{Compact lem0-equ}. 
			
		\end{rem}

		\subsubsection{Compactness}
		The second variation involves, via Lemma \ref{Gap Inequ}, the quadratic term \(\int_{\mathbb{B}^n} d_{\alpha,\beta}^{\,q-2} Q_{\alpha,\beta}^2(\varphi)\). In this subsection, we prove the compactness of the weighted operator $d^{\frac{q-2}{2}}_{\alpha,\beta}Q_{\alpha,\beta}$ from $L^2(\S^{n-1})$ to $L^2(\B^n)$. As mentioned in the introduction, the weight $d_{\alpha,\beta}$ may be unbounded when $\alpha+\beta<1$. This lack of boundedness requires a refined approach and additional estimates to address the resulting technical difficulty.

		\begin{lem}\label{Bound d Lem}
			Suppose $( \alpha,\beta)$ satisfy \eqref{Index Condi} and let $d_{\alpha,\beta}$ be defined in \eqref{d Def}. Then the following asymptotic estimate holds: 
			\begin{align*}
				d_{\alpha,\beta}(\xi)\lesssim \begin{cases}
					\displaystyle (1-|\xi|)^{\beta} \qquad&\mathrm{for}\qquad \alpha>1,\\
					\displaystyle (1-|\xi|)^{\beta} \log \frac{1}{1-|\xi|}\qquad&\mathrm{for}\qquad \alpha=1,\\
					\displaystyle (1-|\xi|)^{\alpha+\beta-1} \qquad&\mathrm{for}\qquad \alpha<1.
				\end{cases}
			\end{align*}
		\end{lem}
		\begin{pf}
			Recall that
			\begin{align*}
				d_{\alpha,\beta}(\xi)=\frac{\pi^{\frac{n}{2}}2^{1-\beta}}{\Gamma(\frac{n}{2})}(1-|\xi|^2)^{\beta+\alpha-1}F\left(\frac{n+\alpha}{2}-1,\frac{\alpha}{2}; \frac{n}{2}; |\xi|^2\right).
			\end{align*}
			If $\alpha<1$, then
			\begin{align*}
				\frac{n}{2}-\left(\frac{n+\alpha}{2}-1\right)-\frac{\alpha}{2}=1-\alpha>0,
			\end{align*}
			which ensures convergence of $F\left(\frac{n+\alpha}{2}-1,\frac{\alpha}{2}; \frac{n}{2}; |\xi|^2\right)$  at $|\xi|=1$, see \cite[P. 1017. 9.122-1]{Gradshteyn&Ryzhik}. Consequently, $F\left(\frac{n+\alpha}{2}-1,\frac{\alpha}{2}; \frac{n}{2}; |\xi|^2\right)$ is bounded in $\B^n$. 
			
			If $\alpha>1$, using \eqref{Pre 1.4}, we can get
			\begin{align*}
				d_{\alpha,\beta}(\xi)=\frac{\pi^{\frac{n}{2}}2^{1-\beta}}{\Gamma(\frac{n}{2})}(1-|\xi|^2)^{\beta}F\left(1-\frac{\alpha}{2},\frac{n-\alpha}{2}; \frac{n}{2}; |\xi|^2\right).
			\end{align*}
			Here
			\begin{align*}
				\frac{n}{2}-\left(1-\frac{\alpha}{2}\right)-\frac{n-\alpha}{2}=\alpha-1>0.
			\end{align*}
			So again $F\left(1-\frac{\alpha}{2},\frac{n-\alpha}{2}; \frac{n}{2}; |\xi|^2\right)$ is bounded in $\B^n$.
			
			The remaining case  is $\alpha=1$. By the formula \cite[P. 75 (4)]{Hypergeometry}, we can get
			\begin{align*}
				d_{\alpha,\beta}(\xi)=\frac{\pi^{\frac{n}{2}}2^{1-\beta}}{\Gamma(\frac{n}{2})}(1-|\xi|^2)^{\beta}F\left(\frac{n-1}{2},\frac{1}{2}; \frac{n}{2}; |\xi|^2\right)=O\left((1-|\xi|)^{\beta} \log \frac{1}{1-|\xi|} \right).
			\end{align*}
			This completes the proof of Lemma \ref{Bound d Lem}.
		\end{pf}

		To prove compactness, a standard approach is to first establish boundedness and then combine interior uniform convergence with a uniform estimate near the boundary; see, for instance, \cite[Lemma 2.5]{Gluck}. However, this method does not cover all parameter regimes in \eqref{Index Condi}, due to the possible unboundedness of $d_{\alpha,\beta}$. To overcome this, our key idea is  to show that the expansion coefficients $A^{\alpha,\beta}_l$ exhibit sufficient decay as $l\to+\infty$. This decay immediately yields compactness; see Lemma \ref{Compact lem0}.
		
		\begin{lem}\label{lem 2.7}	
			Suppose \(\alpha,\beta\) satisfy \eqref{Index Condi} and let \(A_{l}^{\{\alpha,\beta\}}\) be defined by \eqref{Gap lem Al}. Then for sufficiently large \(l\) and any \(\epsilon > 0\),
			\begin{align*}
				A_{l}^{\{\alpha,\beta\}}=\left\{
				\begin{array}{ll}
					O(l^{-(1+q(\alpha+\beta-1))} ), & \hbox{$\alpha<1$}; \\
					O(l^{-1-q\beta+\epsilon} ), & \hbox{$\alpha=1$}; \\
					O(l^{-1-q\beta} ) , & \hbox{ $\alpha>1$.}
				\end{array}
				\right.
			\end{align*}
		\end{lem}
		\begin{pf}
			\textbf{Case 1: \(\alpha > 0\).}
			By Lemma \ref{lem:hypergeo_monotonicity}, the coefficient \(A^{\{\alpha,\beta\}}_{l}\) satisfies
			\begin{align*}
				A^{\{\alpha,\beta\}}_{l}
				&= \frac{\Gamma\!\left(l + \frac{n-\alpha}{2}\right)^2}{\Gamma\!\left(l + \frac{n}{2}\right)^2}
				\int_0^1 d_{\alpha,\beta}^{\,q-2}(\sqrt{t}) (1 - t)^{2(\beta+\alpha-1)} t^{\,n/2 + l - 1}
				F\!\left(l+\frac{n+\alpha}{2}-1,\frac{\alpha}{2}; l+\frac{n}{2}; t\right)^2 \ud t \\
				&\le \frac{\Gamma\!\left(\frac{n-\alpha}{2}\right)^2}{\Gamma\!\left(\frac{n}{2}\right)^2}
				\int_0^1 d_{\alpha,\beta}^{\,q-2}(\sqrt{t}) (1 - t)^{2(\beta+\alpha-1)} t^{\,n/2 + l - 1}
				F\!\left(\frac{n+\alpha}{2}-1,\frac{\alpha}{2}; \frac{n}{2}; t\right)^2 \ud t \\
				&= \frac{\Gamma\!\left(\frac{n-\alpha}{2}\right)^2}{\Gamma\!\left(\frac{n}{2}\right)^2}
				\left(\frac{\Gamma(\frac{n}{2})}{\pi^{\frac{n}{2}}2^{1-\beta}}\right)^{\!2}
				\int_0^1 d_{\alpha,\beta}^{\,q}(\sqrt{t}) t^{\,n/2 + l - 1} \ud t,
			\end{align*}
			where the last equality uses \eqref{d Def}. Applying Lemma~\ref{Bound d Lem} we obtain:
			\begin{itemize}
				\item For \(0< \alpha < 1\),
				\[
				A^{\{\alpha,\beta\}}_{l} \lesssim \int_0^1 (1 - t)^{q(\beta+\alpha-1)} t^{\,n/2 + l - 1} \ud t
				= \frac{\Gamma\!\left(q(\beta+\alpha-1)+1\right) \Gamma(n/2 + l)}
				{\Gamma\!\left(q(\beta+\alpha-1)+1 + n/2 + l\right)}
				\lesssim l^{-(1+q(\alpha+\beta-1))}.
				\]
				
				\item For \(\alpha > 1\),
				\[
				A^{\{\alpha,\beta\}}_{l} \lesssim \int_0^1 (1 - t)^{q\beta} t^{\,n/2 + l - 1} \ud t
				= \frac{\Gamma\!\left(q\beta+1\right) \Gamma(n/2 + l)}
				{\Gamma\!\left(q\beta+1 + n/2 + l\right)}
				\lesssim l^{-1-q\beta}.
				\]
				
				\item For \(\alpha = 1\),
				\begin{align*}					A_{l}^{\{\alpha,\beta\}}\lesssim&  \int_0^1 (1 - t)^{q\beta}|\log(1-t)|^{q} t^{n/2 + l - 1}	 \mathrm{d}t\\
					\lesssim& \int_0^1 (1 - t)^{q\beta-\epsilon} t^{n/2 + l - 1}	 \mathrm{d}t\\
					=&\frac{\Gamma\left(q\beta-\epsilon +1\right)\Gamma(n/2 + l )}{\Gamma\left(q\beta-\epsilon + 1+n/2 + l \right)}\lesssim l^{-1-q\beta+\epsilon}.
				\end{align*}
			\end{itemize}
			
			\textbf{Case 2: \(\alpha \leq 0\).}
			Using the integral representation \eqref{Pre 1.8} of the hypergeometric function, we have, for  $l>1-\frac{n+\alpha}{2}$, 
			\begin{align*}
				F\!\left(l+\frac{n+\alpha}{2}-1,\frac{\alpha}{2}; l+\frac{n}{2}; t\right) 
				&= F\!\left(\frac{\alpha}{2}, l+\frac{n+\alpha}{2}-1; l+\frac{n}{2}; t\right) \\
				&= \frac{\Gamma\!\left(l+\frac{n}{2}\right)}
				{\Gamma\!\left(l+\frac{n+\alpha}{2}-1\right)\Gamma\!\left(1-\frac{\alpha}{2}\right)}
				\int_0^1 s^{l+\frac{n+\alpha}{2}-2} (1-s)^{-\frac{\alpha}{2}} (1-ts)^{-\frac{\alpha}{2}} \ud s .
			\end{align*}
			Observing that
			\[
			(1-ts)^{-\frac{\alpha}{2}} = [(1-s) + s(1-t)]^{-\frac{\alpha}{2}}
			\lesssim (1-s)^{-\frac{\alpha}{2}} + s^{-\frac{\alpha}{2}}(1-t)^{-\frac{\alpha}{2}},
			\]
			we have, for  $l>1-\frac{n+\alpha}{2}$, 
			\begin{align*}
				0 &< F\!\left(l+\frac{n+\alpha}{2}-1,\frac{\alpha}{2}; l+\frac{n}{2}; t\right) \\
				&\lesssim \frac{\Gamma\!\left(l+\frac{n}{2}\right)}
				{\Gamma\!\left(l+\frac{n+\alpha}{2}-1\right)\Gamma\!\left(1-\frac{\alpha}{2}\right)}
				\int_0^1 s^{l+\frac{n+\alpha}{2}-2} (1-s)^{-\alpha} \ud s \\
				&\quad + \frac{\Gamma\!\left(l+\frac{n}{2}\right)}
				{\Gamma\!\left(l+\frac{n+\alpha}{2}-1\right)\Gamma\!\left(1-\frac{\alpha}{2}\right)}
				(1-t)^{-\frac{\alpha}{2}} \int_0^1 s^{l+\frac{n}{2}-2} (1-s)^{-\frac{\alpha}{2}} \ud s \\
				&= \frac{\Gamma\!\left(l+\frac{n}{2}\right)\Gamma(1-\alpha)}
				{\Gamma\!\left(l+\frac{n-\alpha}{2}\right)\Gamma\!\left(1-\frac{\alpha}{2}\right)}
				+ \frac{\Gamma\!\left(l+\frac{n}{2}\right)\Gamma\!\left(l+\frac{n}{2}-1\right)}
				{\Gamma\!\left(l+\frac{n+\alpha}{2}-1\right)\Gamma\!\left(l+\frac{n-\alpha}{2}\right)}
				(1-t)^{-\frac{\alpha}{2}} \\
				&\lesssim \frac{\Gamma\!\left(l+\frac{n}{2}\right)}{\Gamma\!\left(l+\frac{n-\alpha}{2}\right)}
				+ (1-t)^{-\frac{\alpha}{2}} .
			\end{align*}
			Consequently,
			\[
			F\!\left(l+\frac{n+\alpha}{2}-1,\frac{\alpha}{2}; l+\frac{n}{2}; t\right)^2
			\lesssim \frac{\Gamma\!\left(l+\frac{n}{2}\right)^2}{\Gamma\!\left(l+\frac{n-\alpha}{2}\right)^2}
			+ (1-t)^{-\alpha}.
			\]
			Applying Lemma~\ref{Bound d Lem} again, we estimate
			\begin{align*}
				A^{\{\alpha,\beta\}}_{l}
				&= \frac{\Gamma\!\left(l + \frac{n-\alpha}{2}\right)^2}{\Gamma\!\left(l + \frac{n}{2}\right)^2}
				\int_0^1 d_{\alpha,\beta}^{\,q-2}(\sqrt{t}) (1 - t)^{2(\beta+\alpha-1)} t^{\,n/2 + l - 1}
				F\!\left(l+\frac{n+\alpha}{2}-1,\frac{\alpha}{2}; l+\frac{n}{2}; t\right)^2 \ud t \\
				&\lesssim \int_0^1 d_{\alpha,\beta}^{\,q-2}(\sqrt{t}) (1 - t)^{2(\beta+\alpha-1)} t^{\,n/2 + l - 1} \ud t \\
				&\quad + \frac{\Gamma\!\left(l + \frac{n-\alpha}{2}\right)^2}{\Gamma\!\left(l + \frac{n}{2}\right)^2}
				\int_0^1 d_{\alpha,\beta}^{\,q-2}(\sqrt{t}) (1 - t)^{2(\beta+\alpha-1)} t^{\,n/2 + l - 1} (1-t)^{-\alpha} \ud t \\
				&\lesssim \int_0^1 (1 - t)^{q(\beta+\alpha-1)} t^{\,n/2 + l - 1} \ud t  + \frac{\Gamma\!\left(l + \frac{n-\alpha}{2}\right)^2}{\Gamma\!\left(l + \frac{n}{2}\right)^2}
				\int_0^1 (1 - t)^{q(\beta+\alpha-1)-\alpha} t^{\,n/2 + l - 1} \ud t \\
				&= \frac{\Gamma\!\left(q(\beta+\alpha-1)+1\right)\Gamma(n/2 + l)}
				{\Gamma\!\left(q(\beta+\alpha-1)+1 + n/2 + l\right)}  + \frac{\Gamma\!\left(l + \frac{n-\alpha}{2}\right)^2}{\Gamma\!\left(l + \frac{n}{2}\right)^2}
				\frac{\Gamma\!\left(q(\beta+\alpha-1)-\alpha+1\right)\Gamma(n/2 + l)}
				{\Gamma\!\left(q(\beta+\alpha-1)-\alpha+1 + n/2 + l\right)} \\
				&\lesssim l^{-(1+q(\alpha+\beta-1))}.
			\end{align*}
			This completes the proof of Lemma~\ref{lem 2.7}.
		\end{pf}
		
		\begin{rem}
			When $\alpha=0$ and $\beta=1$, we have $A_l^{\{0,1\}}=O(l^{-1})$, which is consistent with the decay rate in \cite[Lemma 2.7]{Frank&Peteranderl&Read}.
		\end{rem}

		With the decay of $A^{\alpha,\beta}_l$ at hand, we can now prove that the operator $d^{\frac{q-2}{2}}_{\alpha,\beta}Q_{\alpha,\beta}$ can be approximated by a sequence of compact operators; in particular, it is the limit of projection operators.
		
		\begin{lem}\label{Compact lem0}
			Suppose $( \alpha,\beta)$ satisfy \eqref{Index Condi}. Then the operator
			\begin{align*}
				d^{\frac{q-2}{2}}_{\alpha,\beta}Q_{\alpha,\beta}: L^2(\S^{n-1})\to  L^2(\B^n)
			\end{align*}
			is compact.
		\end{lem}
		\begin{pf}
			Let $\varphi=\sum_{l=0}^{+\infty}a_lY_l$ with $\|\varphi\|^2_{L^2(\S^{n-1})}=\sum_{l=0}^{+\infty}a_l^2$. From \eqref{Gap lem A0} we have
			\begin{align*}
				\int_{\B^n}d^{q-2}_{\alpha,\beta}Q^2_{\alpha,\beta}(\varphi)	=\frac{\pi^{n}2^{1-2\beta}}{\Gamma^2(\frac{n-\alpha}{2})} \sum_{l=0}^{+\infty} a_l^2 A^{\{\alpha,\beta\}}_{l},
			\end{align*}
			where \(A^{\{\alpha,\beta\}}_{l+1} < A^{\{\alpha,\beta\}}_{l}\) for all \(l\). Notice the following equivalence of two inequalities \begin{align}\label{Compact lem0-equ}
				q(\alpha+\beta-1)+1>0\qquad\Leftrightarrow \qquad\frac{n-\alpha-2\beta}{2n}+\frac{n-\alpha}{2(n-1)}<1,       \end{align}
			which is exactly the last constraint in \eqref{Index Condi}.					Together with Lemma~\ref{lem 2.7} , we obtain $ A_{l}^{\alpha,\beta}=O(l^{-\delta})$ for some $\delta>0$.
			
			Let $T$ be the operator
			$
			d^{\frac{q-2}{2}}_{\alpha,\beta}Q_{\alpha,\beta}: L^2(\S^{n-1})\to  L^2(\B^n) .$     
			Let $P_N:L^2(\S^{n-1})\to L^2(\S^{n-1})$ be the orthogonal projection onto
			$\bigoplus_{l=0}^{N}\mathcal{H}_l$ and set $T_N:=T\circ P_N$.
			Then \(T_N\) has finite rank and is therefore compact. For any $\varphi$ with $\|\varphi\|_{L^2(\S^{n-1})}=\sum_{l=0}^{+\infty}a_l^2=1$,
			\begin{align*}
				\|(T-T_N)\varphi\|_{L^2(\B^n)}^2
				&=\|T(\varphi-P_N\varphi)\|_{L^2(\B^n)}^2\\
				&=\frac{\pi^{n}2^{1-2\beta}}{\Gamma^2(\frac{n-\alpha}{2})}
				\sum_{l\ge N+1} a_{l}^2\,A^{\alpha,\beta}_{l}\\
				&\leq  \frac{\pi^{n}2^{1-2\beta}}{\Gamma^2(\frac{n-\alpha}{2})}A^{\alpha,\beta}_{N+1}\sum_{l \ge N+1} a_{l}^2\\
				&\leq \frac{\pi^{n}2^{1-2\beta}}{\Gamma^2(\frac{n-\alpha}{2})}A^{\alpha,\beta}_{N+1}.
			\end{align*}
			Hence, we can get that 
			\[
			\|T-T_N\|_{\mathcal{L}(L^2(\S^{n-1}),L^2(\B^n))}
			\le 
			\left(\frac{\pi^{n}2^{1-2\beta}}{\Gamma^2(\frac{n-\alpha}{2})}A^{\alpha,\beta}_{N+1}\right)^{\frac12}.
			\]
			Since $A^{\alpha,\beta}_{N+1}\to0$ as $N\to\infty$, we obtain  $\|T-T_N\|\to0$.
			Thus \(T\) is the operator-norm limit of the finite-rank operators \(T_N\), and consequently \(T\) is compact.
		\end{pf}
		
		To treat the case $p< 2$, we need the following compactness lemma, which essentially relies on Lemma \ref{Compact lem0}. Our argument is partly inspired by \cite[Lemma 3.4]{Figalli&Zhang} and \cite[Lemma 3.4]{Frank&Peteranderl&Read}.

		\begin{lem}\label{Compact cor}
			Let $\varphi_i$ be a sequence in $L^{p}(\S^{n-1})$ for $1<p<2$ and let $(\e_i)\subset \R_{+}$ with  $\e_i\to 0$. If
			\begin{align*}
				\int_{\S^{n-1}}\frac{\varphi_i^2}{(1+\e_i|\varphi_i|)^{2-p}}\leq 1,
			\end{align*}
			then there exists a subsequence (still denoted by $\varphi_i$) such that $\varphi_i \rightharpoonup\varphi$ in $L^{p}(\S^{n-1})$,  $d^{\frac{q-2}{2}}_{\alpha,\beta}Q_{\alpha,\beta}(\varphi)\in L^2(\B^{n})$, and
			\begin{align*}
				\lim_{i\to+\infty}\int_{\B^{n}}d^{q-2}_{\alpha,\beta}Q^2_{\alpha,\beta}(\varphi_i)=\int_{\B^{n}}d^{q-2}_{\alpha,\beta} Q^2_{\alpha,\beta}(\varphi).
			\end{align*}
		\end{lem}
		\begin{pf}
			For $1<p<2$, by H\"older's inequality we obtain
			\begin{align*}
				\int_{\S^{n-1}}|\varphi_i|^{p}\leq&
				\left(\int_{\S^{n-1}}(1+\e_i|\varphi_i|)^{p-2}\varphi_i^2\right)^{\frac{p}{2}}
				\left(\int_{\S^{n-1}}(1+\e_i|\varphi_i|)^{p}\right)^{1-\frac{p}{2}}\\
				\leq&\left(1+\|\varphi_i\|^{p}_{L^{p}(\S^{n-1})}\right)^{1-\frac{p}{2}},
			\end{align*}
			which implies $\|\varphi_i\|_{L^{p}(\S^{n-1})}\leq C$. Hence, up to a subsequence, we may assume that $\varphi_i\rightharpoonup\varphi$ in $L^{p}(\S^{n-1})$.
			
			Define
			\begin{align*}
				A_i=\{\eta\in  \S^{n-1}:\e_i|\varphi_i(\eta)|<1\} \qquad\mathrm{and}\qquad
				A^c_i=\{\eta\in  \S^{n-1}:\e_i|\varphi_i(\eta)|\geq 1\}.
			\end{align*}
			Then
			\begin{align*}
				\int_{A_i}\varphi^2_i+\e_i^{p-2}\int_{A^c_i}\varphi^p_i\lesssim 1.
			\end{align*}
			Since $\e_i^{-p}|A_i^c|\leq \|\varphi_i\|^p_{L^{p}(\S^{n-1})}\leq C$, we have $|A_i^c|\to 0$ as $i\to +\infty$.
			Moreover, $\varphi_i\chi_{A_i}$ is a bounded sequence in $L^2(\S^{n-1})$, and hence $\varphi_i\chi_{A_i}\rightharpoonup \varphi'$ in $L^2(\S^{n-1})$ for some $\varphi'$.
			
			We claim that $\varphi'=\varphi$.
			Indeed, writing $\varphi_i\chi_{A_i}=\varphi_i-\varphi_i\chi_{A^c_i}$ and using $\varphi_i\chi_{A^c_i}\to 0$ in $L^p(\S^{n-1})$, we also have $\varphi_i\chi_{A_i}\rightharpoonup \varphi$ in $L^p(\S^{n-1})$, whence $\varphi'=\varphi$.
			
			Therefore, by the compactness lemma \ref{Compact lem0}, we obtain
			\begin{align}\label{Compact cor equa-a}
				\lim_{i\to+\infty}\int_{\B^{n}}d^{q-2}_{\alpha,\beta}Q^2_{\alpha,\beta}(\varphi_i\chi_{A_i})
				=\int_{\B^{n}}d^{q-2}_{\alpha,\beta} Q^2_{\alpha,\beta}(\varphi).
			\end{align}
			On the other hand, 
			\begin{align*}
				\left|\int_{\B^{n}}d^{q-2}_{\alpha,\beta}Q^2_{\alpha,\beta}(\varphi_i\chi_{A^c_i})\right|
				\lesssim \|Q_{\alpha,\beta}(\varphi_i\chi_{A^c_i})\|^2_{L^q(\B^n)}
				\lesssim \|\varphi_i\chi_{A^c_i}\|^2_{L^p(\S^{n-1})}
				\lesssim \e_i^{2(2-p)/p}\to 0,
			\end{align*}
			and similarly,
			\begin{align*}
				\left|\int_{\B^{n}}d^{q-2}_{\alpha,\beta}Q_{\alpha,\beta}(\varphi_i\chi_{A^c_i})Q_{\alpha,\beta}(\varphi_i\chi_{A_i}) \right|
				\lesssim \|Q_{\alpha,\beta}(\varphi_i\chi_{A^c_i})\|_{L^q(\B^n)}
				\lesssim \|\varphi_i\chi_{A^c_i}\|_{L^p(\S^{n-1})}
				\lesssim \e_i^{(2-p)/p}\to 0.
			\end{align*}
			Consequently,
			\begin{align}\label{Compact cor equa-b}
				\lim_{i\to+\infty}\int_{\B^{n}}d^{q-2}_{\alpha,\beta}\bigl(Q^2_{\alpha,\beta}(\varphi_i\chi_{A_i})-Q^2_{\alpha,\beta}(\varphi_i)\bigr)=0.
			\end{align}
			Combining \eqref{Compact cor equa-a} and \eqref{Compact cor equa-b} yields the desired limit.
		\end{pf}

		\subsubsection{Comparison of second order variation}
		Formally, the proof of Theorem~\ref{Thm 1} reduces to establishing a lower bound for
		\begin{align}\label{Sec 2.23 euq}
			\|1+\varphi\|^p_{L^p(\S^{n-1})}c^p_{\alpha,\beta}
			-
			\|Q_{\alpha,\beta}(1+\varphi)\|^p_{L^q(\B^n)}
		\end{align}
		when $\varphi$ is sufficiently  small and satisfies $\varphi\perp \mathcal{H}$. 
		To this end, we shall employ the following fundamental inequality concerning the second variation, due to Figalli and Zhang~\cite{Figalli&Zhang}; see also~\cite{Figalli&Robin}. 
		These inequalities\footnote{When $r=2$, the two types of inequalities in these two Lemmas are both equivalent.} 
		will play a crucial role in establishing the desired stability estimate.

		\begin{lem}[\protect{\cite[Lemma 2.1]{Figalli&Zhang}}]\label{Inequ lem}
			For any $\kappa>0$, there exists $c_{\kappa}$ such that for any $a\in \R$, the following inequalities hold:
			\begin{itemize}
				\item If  $r\geq 2$, then
				\begin{align}\label{Fundamental Ine 1}  
					|1+a|^r\geq 1+ra+\frac{r(1-\kappa)}{2}(a^2+(r-2)\zeta(a)(1-|1+a|)^2)+c_{\kappa}|a|^r,
				\end{align}
				where 
				\begin{align*}
					\zeta(a)=\begin{cases}
						\displaystyle 1 \qquad&\mathrm{for}\qquad a\not\in [-2,0]\\
						\displaystyle |1+a|^{r-1} \qquad&\mathrm{for}\qquad a\in [-2,0].        \end{cases}
				\end{align*}
				\item If  $1<r< 2$, then 
				\begin{align}\label{Fundamental Ine 2}
					|1+a|^r\geq 1+ra+\frac{r(1-\kappa)}{2}(a^2+(r-2)\zeta(a)(1-|1+a|)^2)+c_{\kappa}\min\{|a|^r,|a|^2\}
				\end{align}
				where 
				\begin{align*}
					\zeta(a)=\begin{cases}
						\displaystyle \frac{|1+a|}{(2-r)|1+a|+(r-1)}\qquad&\mathrm{for}\qquad a\not\in [-2,0]\\
						\displaystyle 1 \qquad&\mathrm{for}\qquad a\in [-2,0].        \end{cases}
				\end{align*}

			\end{itemize}    
		\end{lem}
		\begin{rem}\label{Rem 1}
			In \eqref{Fundamental Ine 2}, we note that $1\leq\zeta< \frac{1}{2-r}$. If $a>-1$, then
			\begin{align*}
				a^2+(r-2)\zeta(a)(1-|1+a|)^2=a^2(1+(r-2)\zeta(a))\geq 0.
			\end{align*}
			If $a<-1$, we observe that $|a|>|2+a|$; consequently, 
			\begin{align*}
				a^2+(r-2)\zeta(a)(1-|1+a|)^2=&a^2+(r-2)\zeta(a)(2+a)^2\\
				\geq&a^2+(r-2)\zeta(a)a^2 \geq 0.
			\end{align*}       
		\end{rem}
		
		\begin{lem}[\protect{\cite[Lemma 2.4]{Figalli&Zhang}}]\label{Inequ lem-2}    
			For any $\kappa>0$, there exists $c_{\kappa}$ such that for any $a\in \R$, the following inequalities hold:
			\begin{itemize}
				\item If $2\leq r<+\infty$, there exists a constant $C_{\kappa}$ such that
				\begin{align*}
					|a+b|^r\leq |a|^r+r|a|^{r-1}|b|+\left(\frac{r(r-1)}{2}+\kappa\right)|a|^{r-2}|b|^2+C_{\kappa}|b|^r.
				\end{align*}
				\item If $1<r< 2$, there exists a constant $C_{\kappa}$ such that for any $a,b\in \R$, 
				\begin{align*}
					|a+b|^r\leq |a|^r+r|a|^{r-1}|b|+\left(\frac{r(r-1)}{2}+\kappa\right)\left(\frac{(|a|+C_{\kappa}|b|)^r}{|a|^2+|b|^2}|b|^2\right).
				\end{align*}
			\end{itemize}
			
		\end{lem}

		Now, using the spectral gap inequality from Lemma \ref{Gap Inequ} together with the compactness results of Lemma \ref{Compact lem0} and Lemma \ref{Compact cor},  we can compare the two second-variation terms in \eqref{Sec 2.23 euq}, as they appear in Lemma \ref{Inequ lem} and Lemma \ref{Inequ lem-2}. The following proof is essentially a modification of \cite[Proposition 2.5 and Proposition 3.7]{Frank&Peteranderl&Read} or \cite[Proposition 3.8]{Figalli&Zhang}.

		\begin{pro}\label{Limit Prop}
			For $n\geq 3$ and for any $\lambda\in (0,K_{n,\alpha,\beta})$ and any $\gamma>0$, there exists $\delta=\delta(\lambda,\gamma)$ such that for any $\varphi\in L^p(\S^{n-1})\cap\mathcal{H}^{\perp}$ with $\|\varphi\|_{L^p(\S^{n-1})}<\delta$,
			\begin{align*}
				\|\varphi\|^2_{L^2(\S^{n-1})}
				+&(p-2)\int_{\S^{n-1}}\zeta(\varphi)(1-|1+\varphi|)^2
				+\gamma\mathrm{1}_{\{1<p< 2\}}\int_{\S^{n-1}}\min\{|\varphi|^p,|\varphi|^2\}\nonumber\\
				\geq&(p-1)(K_{n,\alpha,\beta}-\lambda)\int_{\B^n}d^{q-2}_{\alpha,\beta}Q^2_{\alpha,\beta}(\varphi),
			\end{align*}
			where $\mathrm{1}_{\{1<p< 2\}}=1$ for $1<p< 2$, and $\mathrm{1}_{\{1<p< 2\}}=0$ for $p\geq 2$.
		\end{pro}
		
		\begin{pf}
			\medskip	
			We argue by contradiction, following the scheme of \cite{Frank&Peteranderl&Read}. Suppose there exists a sequence $\{\varphi_i\}\subset L^p(\S^{n-1})\cap \mathcal{H}^{\perp}$ with $\|\varphi_i\|_{L^p(\S^{n-1})}\to 0$ and $\gamma_0,\lambda_0>0$ such that
			\begin{align}\label{Limit Prop eua-1}
				\|\varphi_i\|_{L^2(\S^{n-1})}
				+&(p-2)\int_{\S^{n-1}}\zeta(\varphi_i)\left(1-|1+\varphi_i|\right)^2
				+\gamma_0\mathrm{1}_{\{1<p<2\}}\int_{\S^{n-1}}\min\{|\varphi_i|^p,|\varphi_i|^2\}  \nonumber\\
				<& (p-1)(K_{n,\alpha,\beta}-\lambda_0)\int_{\B^n}d^{q-2}_{\alpha,\beta}Q^2_{\alpha,\beta}(\varphi_i).
			\end{align}
			\noindent\textbf{Case 1:} $\alpha\leq  1$ (i.e. $p\geq  2$).	Set $\hat{\varphi}_i=\frac{\varphi_i}{\|\varphi_i\|_{L^2(\S^{n-1})}}$.
			Clearly, $\|\hat{\varphi}_i\|_{L^2(\S^{n-1})}=1$, hence up to a subsequence we may assume that $\hat{\varphi}_i\rightharpoonup\hat{\varphi}$ in $L^2(\S^{n-1})$.	Set $\psi_i=\hat{\varphi}_i-\hat{\varphi}$. Then
			\begin{align*}
				\int_{\S^{n-1}}\zeta(\varphi_i)\left(\frac{1-|1+\varphi_i|}{\|\varphi_i\|_{L^2(\S^{n-1})}}\right)^2
				\geq \int_{\{|\varphi_i|\leq 1\}}\zeta(\varphi_i)|\hat{\varphi}_i|^2
				\geq\int_{\{|\varphi_i|\leq 1\}}\zeta(\varphi_i)(\hat{\varphi}^2+2\hat{\varphi}\psi_i).
			\end{align*}
			Since $\varphi_i\to 0$ almost everywhere, we have $\chi_{\{|\varphi_i|\leq 1\}}\to 1$  and $\zeta(\varphi_i)\chi_{\{|\varphi_i|\leq 1\}}\to 1$  pointwise, and  $|\zeta(\varphi_i)\chi_{\{|\varphi_i|\leq 1\}}|\leq 1$. Hence
			$\zeta(\varphi_i)\chi_{\{|\varphi_i|\leq 1\}}\hat{\varphi}\to \hat{\varphi}$ in $L^2(\S^{n-1})$, and therefore
			\begin{align*}
				\left|\int_{\S^{n-1}}\zeta(\varphi_i)\chi_{\{|\varphi_i|\leq 1\}}\hat{\varphi}\psi_i-\int_{\S^{n-1}}\hat{\varphi}\psi_i\right|
				\lesssim\|\zeta(\varphi_i)\chi_{\{|\varphi_i|\leq 1\}}\hat{\varphi}- \hat{\varphi} \|_{L^2(\S^{n-1})}
				\to 0
			\end{align*}
			as $i\to +\infty$.
			Since $\psi_i\rightharpoonup 0$ in $L^2(\S^{n-1})$, we obtain
			\begin{align*}
				\liminf_{i\to+\infty}\int_{\S^{n-1}}\zeta(\varphi_i)\left(\frac{1-|1+\varphi_i|}{\|\varphi_i\|_{L^2(\S^{n-1})}}\right)^2
				\geq \|\hat{\varphi}\|^2_{L^2(\S^{n-1})}.
			\end{align*}
			By the Compactness Lemma \ref{Compact lem0}, we have
			\begin{align*}
				\lim_{i\to+\infty}\int_{\B^n}d^{q-2}_{\alpha,\beta}Q^2_{\alpha,\beta}(\hat{\varphi}_i)
				=\int_{\B^n}d^{q-2}_{\alpha,\beta}Q^2_{\alpha,\beta}(\hat{\varphi}).
			\end{align*}
			Dividing both sides of \eqref{Limit Prop eua-1} by $\|\varphi_i\|^2_{L^2(\S^{n-1})}$ and letting $i\to+\infty$ yields
			\begin{align}\label{Sec 2 Gap Lem equ-a}
				1+(p-2)\|\hat{\varphi}\|^2_{L^2(\S^{n-1})}
				\leq (p-1)(K_{n,\alpha,\beta}-\lambda_0)\int_{\B^n}d^{q-2}_{\alpha,\beta}Q^2_{\alpha,\beta}(\hat{\varphi}).
			\end{align}
			Observing   $\|\hat{\varphi}\|_{L^2(\S^{n-1})}\leq 1$ and $\hat{\varphi}\in \mathcal{H}^{\perp}$, we obtain
			\begin{align*}
				K_{n,\alpha,\beta}\int_{\B^n}d^{q-2}_{\alpha,\beta}Q^2_{\alpha,\beta}(\hat{\varphi})
				\leq \|\hat{\varphi}\|^2_{L^2(\S^{n-1})}
				\leq (K_{n,\alpha,\beta}-\lambda_0)\int_{\B^n}d^{q-2}_{\alpha,\beta}Q^2_{\alpha,\beta}(\hat{\varphi}).
			\end{align*}
			This forces $\hat{\varphi}\equiv 0$, contradicting  \eqref{Sec 2 Gap Lem equ-a}.
			
			\medskip
			\noindent\textbf{Case 2:} $n>\alpha> 1$ (i.e.  $1<p< 2$).
			We note that for any $\e>0$, there exists $\delta(\e)\in (0,1)$ such that 
			on the set
			\[
			I_i=\{|\varphi_i|\le\delta\},
			\]
			we have $\zeta(\varphi_i)\le 1+\frac{p-1}{2-p}\e$. In this case, we have   
			\begin{align*}
				\varphi_i^2+(p-2)\zeta(\varphi_i)(1-|\varphi_i+1|)^2\geq (p-1)(1-\e)\varphi_i^2.
			\end{align*}
			Set $\e_i=\left(\int_{\S^{n-1}}(1+|\varphi_i|)^{p-2}\varphi_i^2\right)^{1/2}\lesssim \|\varphi_i\|^{\frac{p}{2}}_{L^p(\S^{n-1})}\to 0$ and $\hat{\varphi}_i=\e_i^{-1}\varphi_i$, it means that 
			\begin{align*}
				\int_{\S^{n-1}}\frac{\hat{\varphi}_i^2}{(1+\e_i|\hat{\varphi}_i|)^{2-p}}= 1            \end{align*}
			Using Remark \ref{Rem 1},  we can  estimate
			\begin{align}\label{Sec 2 Gap Lem equ-b}
				(1-\e)\int_{I_i}\hat{\varphi}_i^2
				+\frac{\gamma_0}{p-1}\e_i^{p-2}\delta^{2-p}\int_{I_i^{c}} |\hat{\varphi}_i|^p
				<(K_{n,\alpha,\beta}-\lambda_0)\int_{\B^n}d^{q-2}_{\alpha,\beta}Q^2_{\alpha,\beta}(\hat{\varphi}_i).
			\end{align}
			Clearly, $\hat{\varphi}_i\chi_{I_i}$ is uniformly bounded in $L^2(\S^{n-1})$. We claim that $\hat{\varphi}_i\chi_{I_i}\rightharpoonup \hat{\varphi}$ in $L^2(\S^{n-1})$.
			Indeed, for any $f\in L^{p'}(\S^{n-1})$, we have
			\begin{align*}
				\left|\int_{I_i}f\hat{\varphi}_i-\int_{\S^{n-1}}f\hat{\varphi}\right|
				\lesssim \left|\int_{\S^{n-1}}f(\hat{\varphi}_i-\hat{\varphi})\right|
				+\|f\|_{L^{p'}(I_i^{c})}\|\hat{\varphi}_i\|_{L^p(I_i^{c})}
				\to 0
			\end{align*}
			as $i\to+\infty$.
			The first term follows from $\hat{\varphi}_i\rightharpoonup \hat{\varphi}$, while the second term follows from \eqref{Sec 2 Gap Lem equ-b}.
			
			Taking $i\to+\infty$ in \eqref{Sec 2 Gap Lem equ-b} and using Lemma \ref{Compact cor}, we obtain
			\begin{align*}
				(1-\e)\|\hat{\varphi}\|^2_{L^2(\S^{n-1})}
				\leq (K_{n,\alpha,\beta}-\lambda_0)\int_{\B^n}d^{q-2}_{\alpha,\beta}Q^2_{\alpha,\beta}(\hat{\varphi}).
			\end{align*}
			Moreover,
			\begin{align*}
				1=\int_{\B^n}(1+\e_i|\hat{\varphi}_i|)^{p-2}\hat{\varphi}^2_i
				\leq \int_{I_i}|\hat{\varphi}_i|^2+\e_i^{p-2}\int_{I_i^{c}}|\hat{\varphi}_i|^{p},
			\end{align*}
			combining with \eqref{Sec 2 Gap Lem equ-b}, which implies $\int_{\B^{n}}d^{q-2}_{\alpha,\beta}Q^2_{\alpha,\beta}(\hat{\varphi})\geq C_0>0$, and hence $\hat{\varphi}\not\equiv 0$. Since $\e>0$ is arbitrary, this again yields a contradiction. The proof is thereby completed.
		\end{pf}

		\subsection{Proof of Theorem \ref{Thm 1}}\label{Sec 2.3}
		For the concentrating sequence $\{u_i\}$ in Proposition \ref{Prop 1}, we will impose different orthogonality conditions depending on whether $p$ is greater than $2$. When $p\geq  2$, an almost orthogonality condition is sufficient, whereas for $1<p< 2$, we must enforce exact orthogonality.

		\begin{lem}[\protect{\cite[Proposition 2.8]{Frank&Peteranderl&Read}}]\label{Almost Orothgonal Lem}		
			Let $p\geq 2$ and let $\{u_i\}_{i=1}^{+\infty}$  be a sequence satisfying
			\begin{align*}
				\|u_i\|_{L^p(\S^{n-1})}\to 1\qquad\mathrm{and}\qquad \inf_{\Psi}\||\S^{n-1}|^{\frac{1}{p}}(u_i)_{\Psi}-1\|_{L^p(\S^{n-1})}\to 0.
			\end{align*}
			Then, there exists a sequence of conformal transformations $\Psi_i$, setting $r_i:=|\S^{n-1}|^{\frac{1}{p}}(u_i)_{\Psi_i}-1$, such that
			\begin{align*}
				\|r_i\|_{L^2(\S^{n-1})}=\inf_{\Psi}\||\S^{n-1}|^{\frac{1}{p}}(u_i)_{\Psi}-1\|_{L^2(\S^{n-1})},\quad \|r_i\|_{L^p(\S^{n-1})}\lesssim\inf_{\Psi}\||\S^{n-1}|^{\frac{1}{p}}(u_i)_{\Psi}-1\|_{L^p(\S^{n-1})}
			\end{align*}	
			and
			\begin{align}\label{Orothgonal Lem f1}
				\left|\int_{\S^{n-1}}r_i\right|+\sum_{l=1}^{n}\left|\int_{\S^{n-1}}\xi_lr_i\right|\lesssim \|r_i\|^2_{L^2(\S^{n-1})} +\|r_i\|^p_{L^p(\S^{n-1})}\end{align}

		\end{lem}
		
		The following orthogonality lemma for $1<p< 2$ is obtained by  a straightforward modification of \cite[Proposition 3.9]{Frank&Peteranderl&Read}. More precisely, one considers the minimization of the functional
		$
		F_u(\xi)=\int_{\S^{n-1}} u_{\Psi_\xi}$,
		rather than
		$
		F_v(\xi)=\int_{\B^{n}} v_{\Psi_\xi}
		$
		as in \cite{Frank&Peteranderl&Read}. Since the argument carries over verbatim after this adjustment, we omit the details.
		
		\begin{lem}\label{Orothgonal Lem} Let  $1<p< 2$ and  let $\{u_i\}_{i=1}^{+\infty}$ satisfying 
			\begin{align*}
				\|u_i\|_{L^p(\S^{n-1})}\to 1\qquad\mathrm{and}\qquad \inf_{\Psi}\||\S^{n-1}|^{\frac{1}{p}}(u_i)_{\Psi}-1\|_{L^p(\S^{n-1})}\to 0.
			\end{align*}
			Then there exists a sequence of conformal transformations $\Psi_i$, setting $r_i:=|\S^{n-1}|^{\frac{1}{p}}(u_i)_{\Psi_i}-1$, such that $\|r_i\|_{L^p(\S^{n-1})}\to 0$ as $i\to+\infty$ and 
			\begin{align*}
				\int_{\S^{n-1}}\xi_lr_i=0 \qquad\mathrm{for}\qquad l=1,2,\cdots,n.
			\end{align*}
		\end{lem}

		For any $u\in L^2(\S^{n-1})$,  define the projection operator $\Pi: L^2(\S^{n-1})\to \mathcal{H}$ by
		\begin{align*}
			\Pi (u)=\fint_{\S^{n-1}}u+n\sum_{l=1}^{n}\xi_l\fint_{\S^{n-1}}\xi_lu\in \mathcal{H},
		\end{align*}
		where $\fint_{\S^{n-1}}=\frac{1}{|\S^{n-1}|}\int_{\S^{n-1}}$. Set
		\[
		u^{\perp}=u-\Pi (u)\in\mathcal{H}^{\perp}.
		\]
		For an almost-orthogonal sequence, using the similar arguments in \cite[Lemma 2.9]{Frank&Peteranderl&Read}, we can  show that the component in $\mathcal{H}^{\perp}$ provides the dominant contribution, while the projected part $\Pi(u)$ is negligible in the relevant estimates. We leave the details of the proof to the interested reader.
		\begin{lem}			Let $r_i$ be defined as in Lemma \ref{Almost Orothgonal Lem}. For $p\geq 2$, the following properties hold:
			\begin{align}\label{Loc Ana f1}
				\begin{cases}
					\displaystyle \|r_i\|^2_{L^2(\S^{n-1})}=&\|r^{\perp}_i\|^2_{L^2(\S^{n-1})}+o(\|r_i\|^2_{L^2(\S^{n-1})} +\|r_i\|^p_{L^p(\S^{n-1})})\\
					\displaystyle \|r_i\|^p_{L^p(\S^{n-1})}=&\|r^{\perp}_i\|^p_{L^p(\S^{n-1})}+o(\|r_i\|^2_{L^2(\S^{n-1})} +\|r_i\|^p_{L^p(\S^{n-1})})
				\end{cases}
			\end{align}
			and
			\begin{align}\label{Loc Ana f2}
				\int_{\B^n}d^{q-2}_{\alpha,\beta}Q^2_{\alpha,\beta}(r_i)=&\int_{\B^n}d^{q-2}_{\alpha,\beta}Q^2_{\alpha,\beta}(r^{\perp}_i)+o(\|r_i\|^2_{L^2(\S^{n-1})} +\|r_i\|^p_{L^p(\S^{n-1})}).
			\end{align}
			Moreover, it also holds that
			\begin{align}\label{Loc Ana f3}
				&\int_{\S^{n-1}}\zeta(r_i)(1-|1+r_i|)^2\nonumber\\
				=&\int_{\S^{n-1}}\zeta(r_i^{\perp})(1-|1+r_i^{\perp}|)^2+o(\|r_i\|^2_{L^2(\S^{n-1})} +\|r_i\|^p_{L^p(\S^{n-1})}).
			\end{align}
		\end{lem}

		\begin{pro}\label{Loc Bound Prop}
			For any sequence $\{u_i\}\in L^p(\S^{n-1})$ satisfying $\|u_i\|_{L^p(\S^{n-1})}=1$ and $$\inf_{\Psi}\||\S^{n-1}|^{\frac{1}{p}}(u_i)_{\Psi}-1\|_{L^{p}(\S^{n-1})}\to 0,$$ there exists $C_n>0$ such that the following statements hold:
			\begin{itemize}
				\item[(1)] If $\alpha\leq 1$ (i.e.  $p\geq 2$), 
				\begin{align*}
					\liminf_{i\to+\infty}\frac{c^p_{\alpha,\beta}-\|Q_{\alpha,\beta}(u_i)\|^p_{L^q(\B^{n})}}{\inf_{\Psi}\| |\S^{n-1}|^{\frac{1}{p}}(u_i)_{\Psi}-1\|^p_{L^{p}(\S^{n-1})}+\||\S^{n-1}|^{\frac{1}{p}}(u_i)_{\Psi}-1\|^2_{L^{2}(\S^{n-1})}}\geq C_n.
				\end{align*}
				\item[(2)] If $n>\alpha> 1$ (i.e.  $1<p< 2$),
				\begin{align*}
					\liminf_{i\to+\infty}\frac{c^p_{\alpha,\beta}-\|Q_{\alpha,\beta}(u_i)\|^p_{L^q(\B^{n})}}{\inf_{\Psi}\||\S^{n-1}|^{\frac{1}{p}}(u_i)_{\Psi}-1\|^2_{L^{2}(\S^{n-1})}}\geq C_n.
				\end{align*}
			\end{itemize}
		\end{pro}
		
		\begin{pf}
			\textbf{Case 1}: Let $\alpha\leq 1$ (i.e, $p\geq 2$). By Lemma \ref{Almost Orothgonal Lem}, there exists a sequence of conformal transformations $\Psi_i$. Set $r_i=|\S^{n-1}|^{\frac{1}{p}}(u_i)_{\Psi_i}-1$; then   $r_i$ satisfies \eqref{Loc Ana f1} and \eqref{Loc Ana f2}. Using the conformal covariance of $Q_{\alpha,\beta}$, we have
			\begin{align*}
				|\S^{n-1}|\left[c^p_{\alpha,\beta}-\|Q_{\alpha,\beta}(u_i)\|^p_{L^q(\B^{n})}\right]
				=\|1+r_i\|^p_{L^p(\S^{n-1})}c^p_{\alpha,\beta}-\|Q_{\alpha,\beta}(1+r_i)\|^p_{L^q(\B^n)}.
			\end{align*}
			Let $d_{\alpha,\beta}(\xi)=Q_{\alpha,\beta}(1)(\xi)$ and, taking $u\equiv 1$ in  Theorem B \ref{thm B}, we can see
			\begin{align*}
				c_{\alpha,\beta}^p=\frac{1}{|\S^{n-1}|}\left(\int_{\B^n}d^q_{\alpha,\beta}\right)^{\frac{p}{q}}.
			\end{align*}
			Applying the first inequality of Lemma \ref{Inequ lem-2}, we obtain
			\begin{align*}
				&\|Q_{\alpha,\beta}(1+r_i)\|^q_{L^q(\B^n)}\\
				\leq& \int_{\B^n}d^q_{\alpha,\beta}+q\int_{\B^n}d^{q-1}_{\alpha,\beta}Q_{\alpha,\beta}(r_i)
				+\left(\frac{q(q-1)}{2}+\kappa\right)\int_{\B^n}d^{q-2}_{\alpha,\beta}Q^2_{\alpha,\beta}(r_i)
				+C_{\kappa}\|Q_{\alpha,\beta}(r_i)\|^q_{L^q(\B^n)}.
			\end{align*}
			Using $(1+a)^{\frac{p}{q}}\leq 1+\frac{p}{q}a$, we further have
			\begin{align*}
				&c_{\alpha,\beta}^{-p}|\S^{n-1}|^{-1}\|Q_{\alpha,\beta}(1+r_i)\|^p_{L^q(\B^n)} \\
				\leq& 1+\frac{p}{\int_{\B^n}d^q_{\alpha,\beta}}\int_{\B^n}d^{q-1}_{\alpha,\beta}Q_{\alpha,\beta}(r_i)
				+\left(\frac{p(q-1)}{2}+\frac{p\kappa}{q}\right)\frac{\int_{\B^n}d^{q-2}_{\alpha,\beta}Q^2_{\alpha,\beta}(r_i)}{\int_{\B^n}d^q_{\alpha,\beta}}
				+C_{\kappa}\|Q_{\alpha,\beta}(r_i)\|^q_{L^q(\B^n)}.
			\end{align*}
			From the first inequality \eqref{Fundamental Ine 1} in Lemma \ref{Inequ lem}, we deduce
			\begin{align*}
				&\|1+r_i\|^p_{L^p(\S^{n-1})}\\
				\geq& |\S^{n-1}|+p\int_{\S^{n-1}}r_i
				+\frac{(1-\kappa)p}{2}\int_{\S^{n-1}}\bigl(r_i^2+(p-2)\zeta(r_i)(1-|1+r_i|)^2\bigr)
				+c_{\kappa}\|r_i\|^p_{L^p(\S^{n-1})}.
			\end{align*}
			By the definition of $c_{\alpha,\beta}$ and  \eqref{Variation 1}, it follows that
			\begin{align*}
				c^p_{\alpha,\beta}\int_{\S^{n-1}}r_i
				=\left(\int_{\B^n}d^q_{\alpha,\beta}\right)^{\frac{p}{q}-1}\int_{\B^n}d^{q-1}_{\alpha,\beta}Q_{\alpha,\beta}(r_i).
			\end{align*}
			Therefore, we can estimate
			\begin{align*}
				&\|1+r_i\|^p_{L^p(\S^{n-1})}c^p_{\alpha,\beta}-\|Q_{\alpha,\beta}(1+r_i)\|^p_{L^q(\B^n)}\\
				\geq&\frac{(1-\kappa)pc^p_{\alpha,\beta}}{2}\int_{\S^{n-1}}\bigl(r_i^2+(p-2)\zeta(r_i)(1-|1+r_i|)^2\bigr)
				+c_{\kappa}\|r_i\|^p_{L^p(\S^{n-1})}\\
				-&\left(\frac{p(q-1)}{2}+\frac{p\kappa}{q}\right)c^p_{\alpha,\beta}|\S^{n-1}|\frac{\int_{\B^n}d^{q-2}_{\alpha,\beta}Q^2_{\alpha,\beta}(r_i)}{\int_{\B^n}d^q_{\alpha,\beta}}
				-C_{\kappa}\|Q_{\alpha,\beta}(r_i)\|^q_{L^q(\B^n)}.
			\end{align*}
			By \eqref{Loc Ana f2} and \eqref{Loc Ana f3}, we have
			\begin{align*}
				\int_{\B^n}d^{q-2}_{\alpha,\beta}Q^2_{\alpha,\beta}(r_i)=&\int_{\B^n}d^{q-2}_{\alpha,\beta}Q^2_{\alpha,\beta}(r^{\perp}_i)+o(\|r_i\|^2_{L^2(\S^{n-1})}+\|r_i\|^p_{L^p(\S^{n-1})})\\
				\leq&\frac{1}{(p-1)(K_{n,\alpha,\beta}-\lambda)}\int_{\S^{n-1}}(r^{\perp}_i)^2+(p-2)\zeta(r^{\perp}_i)(1-|1+r^{\perp}_i|)^2\\
				+&o(\|r_i\|^2_{L^2(\S^{n-1})}+\|r_i\|^p_{L^p(\S^{n-1})})\\
				=&\frac{1}{(p-1)(K_{n,\alpha,\beta}-\lambda)}\int_{\S^{n-1}}r_i^2+(p-2)\zeta(r_i)(1-|1+r_i|)^2\\
				+&o(\|r_i\|^2_{L^2(\S^{n-1})}+\|r_i\|^p_{L^p(\S^{n-1})}).
			\end{align*}
			Moreover, since $q>p$, then $\|Q_{\alpha,\beta}(r_i)\|^q_{L^q(\B^n)}\lesssim \|r_i\|^q_{L^p(\S^{n-1})}=o(\|r_i\|^p_{L^p(\S^{n-1})})$, we conclude that
			\begin{align*}
				&\|1+r_i\|^p_{L^p(\S^{n-1})}c^p_{\alpha,\beta}-\|Q_{\alpha,\beta}(1+r_i)\|^p_{L^q(\B^n)}\\
				\geq&  \frac{pc^p_{\alpha,\beta}}{2}\left[1-\frac{q-1}{p-1}\frac{|\S^{n-1}|}{\int_{\B^n}d^q_{\alpha,\beta}}\frac{1}{K_{n,\alpha,\beta}-\lambda}+O(\kappa)\right]
				\int_{\S^{n-1}}\bigl(r_i^2+(p-2)\zeta(r_i)(1-|1+r_i|)^2\bigr)\\
				+&c_{\kappa}\|r_i\|^p_{L^p(\S^{n-1})}+o(\|r_i\|^2_{L^2(\S^{n-1})}+\|r_i\|^p_{L^p(\S^{n-1})}).
			\end{align*}
			Finally, invoking the spectral gap inequality \eqref{Gap Inequ-equ} and choosing $\lambda$ and $\kappa$ sufficiently small, we obtain the required lower bound in Case 1.

			\textbf{Case 2}: Let $n>\alpha> 1$ (i.e. $1<p<2$). By Lemma \ref{Orothgonal Lem}, there exists a sequence $\Psi_i$. Setting $r_i:=|\S^{n-1}|^{\frac{1}{p}}(u_i)_{\Psi_i}-1$, we have  $\|r_i\|_{L^p(\S^{n-1})}\to 0$ as $i\to+\infty$ and
			\begin{align*}
				\int_{\S^{n-1}}\xi_lr_i=0 \qquad\mathrm{for}\qquad l=1,2,\cdots,n.
			\end{align*}
			Let $\alpha_i=|\S^{n-1}|^{\frac{1}{p}}\fint_{\S^{n-1}}(u_i)_{\Psi_i}$. Then
			\begin{align*}
				|\S^{n-1}|^{\frac{1}{p}}|\alpha_i-1|
				\leq \||\S^{n-1}|^{\frac{1}{p}}(u_i)_{\Psi_i}-1\|_{L^{p}(\S^{n-1})}
			\end{align*}
			and
			\begin{align*}
				|\S^{n-1}|^{\frac{1}{p}}|\alpha_i-1|
				=|\S^{n-1}|^{\frac{1}{p}}|\alpha_i-\|(u_i)_{\Psi_i}\|_{L^{p}(\S^{n-1})}|
				\leq \||\S^{n-1}|^{\frac{1}{p}}(u_i)_{\Psi_i}-\alpha_i\|_{L^{p}(\S^{n-1})}.
			\end{align*}
			Consequently,
			\begin{align}\label{Prop vital equ-a}
				\frac{1}{2}\||\S^{n-1}|^{\frac{1}{p}}(u_i)_{\Psi_i}-1\|_{L^{p}(\S^{n-1})}
				\leq\||\S^{n-1}|^{\frac{1}{p}}(u_i)_{\Psi_i}-\alpha_i\|_{L^{p}(\S^{n-1})}
				\leq 2 \||\S^{n-1}|^{\frac{1}{p}}(u_i)_{\Psi_i}-1\|_{L^{p}(\S^{n-1})}.
			\end{align}
			Set $\Tilde{r}_i=\frac{|\S^{n-1}|^{\frac{1}{p}}}{\alpha_i}(u_i)_{\Psi_i}-1$. Then $\Tilde{r}_i\in \mathcal{H}^{\perp}$, and
			\begin{align*}
				|\S^{n-1}|\left[c^p_{\alpha,\beta}-\|Q_{\alpha,\beta}(u_i)\|^p_{L^q(\B^{n})}\right]
				=\alpha_i^p\left[\|1+\Tilde{r}_i\|^p_{L^p(\S^{n-1})}c^p_{\alpha,\beta}-\|Q_{\alpha,\beta}(1+\Tilde{r}_i)\|^p_{L^q(\B^n)}\right].
			\end{align*}
			Using the second inequality \eqref{Fundamental Ine 2} in Lemma \ref{Inequ lem}, we obtain
			\begin{align*}
				\|1+\Tilde{r}_i\|^p_{L^p(\S^{n-1})}
				\geq& |\S^{n-1}|+p\int_{\S^{n-1}}\Tilde{r}_i
				+\frac{(1-\kappa)p}{2}\int_{\S^{n-1}}\bigl(\Tilde{r}_i^2+(p-2)\zeta(\Tilde{r}_i)(1-|1+\Tilde{r}_i|)^2\bigr)\\
				&+c_{\kappa}\int_{\S^{n-1}}\min\{|\Tilde{r}_i|^p,|\Tilde{r}_i|^2\}.
			\end{align*}
			Using Proposition \ref{Limit Prop} and applying the same estimates as Case 1, we derive
			\begin{align*}
				&\|1+\Tilde{r}_i\|^p_{L^p(\S^{n-1})}c^p_{\alpha,\beta}-\|Q_{\alpha,\beta}(1+\Tilde{r}_i)\|^p_{L^q(\B^n)}\\
				\geq&  \frac{pc^p_{\alpha,\beta}}{2}\left[1-\frac{q-1}{p-1}\frac{|\S^{n-1}|}{\int_{\B^n}d^q_{\alpha,\beta}}\frac{1}{K_{n,\alpha,\beta}-\lambda}+O(\kappa)\right]
				\int_{\S^{n-1}}\bigl(\Tilde{r}_i^2+(p-2)\zeta(\Tilde{r}_i)(1-|1+\Tilde{r}_i|)^2\bigr) \\
				&+c_{\kappa}\int_{\S^{n-1}}\min\{|\Tilde{r}_i|^p,|\Tilde{r}_i|^2\}+O(\kappa)\int_{\B^n}d^{q-2}_{\alpha,\beta}Q^2_{\alpha,\beta}(\Tilde{r}_i)-C_{\kappa}\|Q_{\alpha,\beta}(r_i)\|^q_{L^q(\B^n)}.
			\end{align*}
			Observe that $\|Q_{\alpha,\beta}(r_i)\|^q_{L^q(\B^n)}\lesssim \|r_i\|^q_{L^p(\S^{n-1})}=o(\|r_i\|^2_{L^p(\S^{n-1})})$, 
			\begin{align*}
				\int_{\S^{n-1}}\min\{|\Tilde{r}_i|^p,|\Tilde{r}_i|^2\}
				\geq& \int_{|\Tilde{r}_i|\leq 1}\Tilde{r}_i^2+\int_{|\Tilde{r}_i|> 1}\Tilde{r}_i^p\\
				\geq& |\S^{n-1}|^{1-\frac{2}{p}}\left(\int_{|\Tilde{r}_i|\leq 1}\Tilde{r}_i^p\right)^{\frac{2}{p}}+\int_{|\Tilde{r}_i|> 1}\Tilde{r}_i^p \\
				\geq& C \|\Tilde{r}_i\|^2_{L^p(\S^{n-1})},
			\end{align*}
			and 
			\begin{align*}
				\int_{\B^n}d^{q-2}_{\alpha,\beta}Q^2_{\alpha,\beta}(\Tilde{r}_i)\lesssim \left(\int_{\B^n}d^q_{\alpha,\beta}\right)^{\frac{q-2}{q}}\left(\int_{\B^n}|Q_{\alpha,\beta}(\Tilde{r}_i)|^q\right)^{\frac{2}{q}} \lesssim\|\Tilde{r}_i\|^2_{L^p(\S^{n-1})}.
			\end{align*} 
			
			Therefore, using the spectral gap inequality \eqref{Gap Inequ-equ} and choosing $\lambda$ and $\kappa$ sufficiently small,  we obtain  
			\begin{align*}
				\|1+\Tilde{r}_i\|^p_{L^p(\S^{n-1})}c^p_{\alpha,\beta}-\|Q_{\alpha,\beta}(1+\Tilde{r}_i)\|^p_{L^q(\B^n)}
				\geq C\|\Tilde{r}_i\|^2_{L^p(\S^{n-1})}.
			\end{align*}
			Using \eqref{Prop vital equ-a}, we conclude that there exist two universal  constants  $C_1,C_2$ such that
			\[
			C_1\|r_i\|^2_{L^p(\S^{n-1})}\leq \|\Tilde{r}_i\|^2_{L^p(\S^{n-1})}\leq C_2 \|r_i\|^2_{L^p(\S^{n-1})},
			\]
			which completes the proof of Case 2.
		\end{pf}

		\textbf{Proof of Theorem \ref{Thm 1}}: Here we only treat Case~1 ($1\geq \alpha$, i.e.  $p\geq 2$); the proof of Case~2 follows exactly the same line of reasoning.
		
		Suppose, to the contrary, that there exists a sequence $u_i\in L^p(\S^{n-1})$ such that
		\begin{align*}
			\frac{c^p_{\alpha,\beta}-\|Q_{\alpha,\beta}u_i\|^p_{L^q(\B^n)}/\|u_i\|^p_{L^p(\S^{n-1})}}{\inf_{\Psi,\lambda}(\|\lambda|\S^{n-1}|^{\frac{1}{p}}(u_i)_{\Psi}-1\|^p_{L^p(\S^{n-1})} +\|\lambda|\S^{n-1}|^{\frac{1}{p}}(u_i)_{\Psi}-1\|^2_{L^2(\S^{n-1})})}\to 0
		\end{align*}
		as $i\to+\infty$. Without loss of generality, we may assume $\|u_i\|_{L^p(\S^{n-1})}=1$. Clearly,
		\begin{align*}
			\inf_{\Psi,\lambda}&\|\lambda|\S^{n-1}|^{\frac{1}{p}}(u_i)_{\Psi}-1\|^p_{L^p(\S^{n-1})} +\|\lambda|\S^{n-1}|^{\frac{1}{p}}(u_i)_{\Psi}-1\|^2_{L^2(\S^{n-1})}\\
			\leq&\|1\|^p_{L^p(\S^{n-1})}+ \|1\|^p_{L^p(\S^{n-1})},
		\end{align*}
		and hence $\|Q_{\alpha,\beta}u_i\|^p_{L^q(\B^n)}\to c^p_{\alpha,\beta}$ as $i\to+\infty$. By Proposition \ref{Prop 1}, we obtain
		\begin{align*}
			\inf_{\Psi,\lambda\in\{\pm 1\}}\|\lambda|\S^{n-1}|^{\frac{1}{p}}(u_i)_{\Psi}-1\|_{L^p(\S^{n-1})}\to 0.
		\end{align*}
		Replacing $u_i$ by $-u_i$ if necessary, we may assume
		\[
		\inf_{\Psi}\||\S^{n-1}|^{\frac{1}{p}}(u_i)_{\Psi}-1\|_{L^p(\S^{n-1})}\to 0.
		\]
		Therefore,
		\begin{align*}
			&\inf_{\Psi,\lambda}\|\lambda|\S^{n-1}|^{\frac{1}{p}}(u_i)_{\Psi}-1\|^p_{L^p(\S^{n-1})} +\|\lambda|\S^{n-1}|^{\frac{1}{p}}(u_i)_{\Psi}-1\|^2_{L^2(\S^{n-1})}\\
			\leq&\inf_{\Psi}\||\S^{n-1}|^{\frac{1}{p}}(u_i)_{\Psi}-1\|^p_{L^p(\S^{n-1})} +\||\S^{n-1}|^{\frac{1}{p}}(u_i)_{\Psi}-1\|^2_{L^2(\S^{n-1})}.
		\end{align*}
		This contradicts Proposition \ref{Loc Bound Prop}.

		\section{The stability of dual operator}\label{Sec 3}
		
		\subsection{Variation and concentration compactness}\label{Sec 3.1}
		
		For $\psi\in L^{2}(\B^{n})$,  recall the dual  operator $S_{\alpha,\beta}$ associated with $Q_{\alpha,\beta}$, defined on the boundary $\S^{n-1}$ by
		\begin{align*}
			S_{\alpha,\beta}(\psi)(\eta)
			:=\int_{\B^{n}}\left(\frac{1-|\xi|^{2}}{2}\right)^{\beta}
			\frac{\psi(\xi)}{|\xi-\eta|^{\,n-\alpha}}\,\ud \xi,
			\qquad \eta\in \S^{n-1}.
		\end{align*}
		In particular, $S_{\alpha,\beta}(\psi)$ is a (weighted) Riesz-type potential of $\psi$ evaluated on the boundary.
		By Fubini's theorem together with   the definitions of $Q_{\alpha,\beta}$ and $S_{\alpha,\beta}$, we obtain the following duality identity: for any
		$\varphi\in L^{2}(\S^{n-1})$ and $\psi\in L^{2}(\B^{n})$,
		\begin{align}\label{Dual f1}
			\int_{\S^{n-1}}\varphi(\eta)\, S_{\alpha,\beta}(\psi)(\eta)\,\ud V_{\S^{n-1}}(\eta)
			=\int_{\B^{n}}Q_{\alpha,\beta}(\varphi)(\xi)\,\psi(\xi)\,\ud \xi ,
		\end{align}
		which shows that  $S_{\alpha,\beta}$ is the $L^{2}$-adjoint of $Q_{\alpha,\beta}$.
		
		The Funk--Hecke formula will be our main tool for diagonalizing spherical convolution operators on  $\S^{n-1}$.
		Let $K\in L^{1}\!\left((-1,1),(1-t^{2})^{\frac{n-3}{2}}\,\ud t\right)$ and consider the zonal kernel $K(\xi\cdot\eta)$.
		Then for each spherical harmonic $Y_{l}\in \mathscr{H}_{l}$, the Funk--Hecke formula (see \cite[Chapter~22]{Abramowitz&Stegun}) asserts that
		\begin{align}\label{Funk-Hecke}
			\int_{\S^{n-1}} K(\xi\cdot\eta)\, Y_{l}(\eta)\,\ud V_{\S^{n-1}}(\eta)
			=\lambda_{l}\,Y_{l}(\xi),\qquad \xi\in \S^{n-1},
		\end{align}
		where the eigenvalue $\lambda_{l}$ is 
		\begin{align*}
			\lambda_{l}
			=(4\pi)^{\frac{n-2}{2}}
			\frac{\Gamma\!\left(\frac{n-2}{2}\right)l!}{\Gamma(l+n-2)}
			\int_{-1}^{1}K(t)\,C_{l}^{\frac{n-2}{2}}(t)\,(1-t^{2})^{\frac{n-3}{2}}\,\ud t.
		\end{align*}
		Here (see \cite[p. 1000. 8.934-1]{Gradshteyn&Ryzhik})
		\begin{align}\label{Gegenbauer polynomial}
			C_l^{\frac{n-2}{2}}(t) = \frac{(-1)^l}{2^l} \frac{\Gamma(n-2+l)\Gamma\left(\frac{n-1}{2}\right)}{\Gamma(n-2)\Gamma\left(\frac{n-1}{2} + l\right)}
			\frac{(1-t^2)^{\frac{3-n}{2}}}{l!}
			\frac{d^l}{dt^l}(1-t^2)^{\frac{n-3}{2}+l}
		\end{align}
		denotes the $l$-th Gegenbauer polynomial with parameter $\frac{n-2}{2}$ .   
		
		For the dual inequality, the minimizer is in general non-constant. 
		Motivated by the role of the constant function $1$ in the primal inequality, 
		we introduce the following radial function:
		\begin{align*}
			\tilde{\mathbf{1}}(\xi)
			=
			\left(\fint_{\B^n} d_{\alpha,\beta}^q \right)^{\frac{1-q}{q}}
			d_{\alpha,\beta}^{\,q-1}(\xi),
		\end{align*}
		which satisfies the normalization condition 
		\begin{align}\label{normalization condition}
			\|\tilde{\mathbf{1}}\|_{L^{q'}(\B^n)}=|\B^n|^{1/q'}.
		\end{align}
		In particular, when $\alpha=0$ and $\beta=1$, we have $\tilde{\mathbf{1}}(\xi)\equiv 1$.
		To be consistent with the definition of $d_{\alpha,\beta}$, we further define
		\begin{align*}
			\widetilde d_{\alpha,\beta}(\eta)
			:=
			S_{\alpha,\beta}(\tilde{\mathbf{1}})(\eta),
			\qquad \eta\in \S^{n-1}.
		\end{align*}
		Since $\tilde{\mathbf{1}}(\xi)$ is radial,   $\widetilde d_{\alpha,\beta}$ is constant.
		Its value can be determined from the equality case in 
		Theorem~B \ref{thm B}, which yields
		\begin{align*}
			\widetilde d_{\alpha,\beta}
			\equiv
			n^{-1}
			\left(\fint_{\B^n} d_{\alpha,\beta}^q \right)^{\frac{1}{q}}.
		\end{align*}

		An important feature of Gluck's dual inequality is its first-variation identity and second-variation inequality, which are essential for the proof of Proposition
		\ref{Dule Loc Bound Prop}.

		\begin{lem}\label{Integral Identity Lem}
			For any $\psi\in L^{q'}(\B^n)$, 
			\begin{align}\label{Integral Identity}
				\fint_{\B^n}\psi\tilde{\mathbf{1}}^{q'-1}=\fint_{\S^{n-1}}\frac{S_{\alpha,\beta}(\psi)}{S_{\alpha,\beta}(\tilde{\mathbf{1}})}.
			\end{align}
			Moreover, if $\int_{\B^{n}}\psi\tilde{\mathbf{1}}^{q'-1}=0$, then
			\begin{align}\label{Integral Identity-2}
				\fint_{\B^n}\psi^2\tilde{\mathbf{1}}^{q'-2}\geq \frac{p'-1}{q'-1}\Tilde{d}_{\alpha,\beta}^{-2}\fint_{\S^{n-1}}S^2_{\alpha,\beta}(\psi).
			\end{align}
		\end{lem}
		
		\begin{pf} 
			By density, we may assume $\psi\in C^{\infty}(\overline{\mathbb{B}^{n}})$.  
			For any $\e>0$, consider the function
			\begin{align*}
				f(\e)=\|S_{\alpha,\beta}(\tilde{\mathbf{1}}+\e\psi)\|_{L^{p'}(\S^{n-1})}\|\tilde{\mathbf{1}}+\e\psi\|^{-1}_{L^{q'}(\B^n)}.
			\end{align*}
			Then $f(0)=c_{\alpha,\beta}$, and $0$  is a local maximizer of  $f$. Differentiating, 
			\begin{align*}
				f^{'}(\e)=&\,f(\e)\left(\int_{\S^{n-1}}\left(\Tilde{d}_{\alpha,\beta}+\e S_{\alpha,\beta}(\psi)\right)^{p'}\right)^{-1}
				\int_{\S^{n-1}}S_{\alpha,\beta}(\psi)\left(\Tilde{d}_{\alpha,\beta}+\e S_{\alpha,\beta}(\psi)\right)^{p'-1}\\
				&-f(\e)\left(\int_{\B^n}(\tilde{\mathbf{1}}+\e\psi)^{q'}\right)^{-1}\int_{\B^{n}}\psi(\tilde{\mathbf{1}}+\e\psi)^{q'-1}.
			\end{align*}
			Evaluating at $\e=0$ and using the normalization condition \eqref{normalization condition}  yields \eqref{Integral Identity}.

			A direct computation gives
			\begin{align*}
				f^{''}(0)=&(p'-1)\Tilde{d}_{\alpha,\beta}^{-2}f(0)\fint_{\S^{n-1}}S^2_{\alpha,\beta}(\psi)
				-p'f(0)\Tilde{d}^{-2}_{\alpha,\beta}\left(\fint_{\S^{n-1}}S_{\alpha,\beta}(\psi)\right)^2\\
				&-(q'-1)f(0)\fint_{\B^n}\psi^2\tilde{\mathbf{1}}^{q'-2}+f(0)q'\left(\int_{\B^n}\psi\tilde{\mathbf{1}}^{q'-1}\right)^2.
			\end{align*}
			Therefore, if $\int_{\B^n}\psi\tilde{\mathbf{1}}^{q'-1}=0$, then by \eqref{Integral Identity} we obtain
			$$
			\int_{\S^{n-1}} S_{\alpha,\beta}(\psi)=0.
			$$
			Since  $0$ is a local maximum of $f$, we have $f''(0)\le 0$; hence
			\begin{align*}
				\fint_{\B^n}\psi^2\tilde{\mathbf{1}}^{q'-2}\geq \frac{p'-1}{q'-1}\Tilde{d}_{\alpha,\beta}^{-2}\fint_{\S^{n-1}}S^2_{\alpha,\beta}(\psi) ,
			\end{align*}
			which is exactly \eqref{Integral Identity-2}. 
			
		\end{pf}
		
		Using the duality argument for Proposition \ref{Prop 1}, we obtain the following concentration--compactness principle for $S_{\alpha,\beta}$. 
		The proof proceeds analogously to the argument presented in   \cite[Proposition 3.1]{Frank&Peteranderl&Read};  details are omitted here and left to interested readers.

		\begin{pro}
			Let $\{v_i\}_{i=1}^{+\infty}$ be a sequence satisfying
			\begin{align*}
				\|v_i\|_{L^{q'}(\B^n)}\to 1\qquad\mathrm{and}\qquad \|S_{\alpha,\beta}(v_i)\|_{L^{p'}(\S^{n-1})}\to c_{\alpha,\beta}.
			\end{align*}
			Then
			\begin{align*}
				\inf_{\Psi,\lambda\in\{\pm 1\}}\|\lambda|\B^{n}|^{\frac{1}{q'}}(v_i)_{\Psi}-\tilde{\mathbf{1}}\|_{L^{q'}(\B^{n})}\to 0.
			\end{align*}
		\end{pro}

		\subsection{Spectral gaps}\label{Sec 3.2}
		As in the proof of Theorem \ref{Thm 1}, establishing spectral gaps is essential for the overall argument for the dual operator $S_{\alpha,\beta}$. In the present setting, since $\tilde{d}_{\alpha,\beta}$ is a constant, it suffices to work with $S_{\alpha,\beta}$ alone. However, as mentioned in the introduction, we still need to handle a new obstruction. We begin with the following elementary  inequality.

		\begin{lem}\label{Basic Iequ Lem}
			Let $f,g\in L^2([0,1],\ud \mu)$ and suppose that  $\int_{0}^{1}g(x)h(x)\ud \mu=0$ for some positive function $h(x)\in L^2([0,1],\ud \mu)$. Then
			\begin{align}\label{Basic lneq}
				\left(\int_{0}^{1}f(x)g(x)\ud \mu \right)^2 \leq \left(\int_{0}^{1}f^2\ud \mu-\frac{\left(\int_{0}^{1}f(x)h(x)\ud \mu \right)^2}{\int_{0}^{1}h^2(x)\ud \mu}\right)\|g\|^2_{L^2([0,1],\ud \mu)}.
			\end{align}
		\end{lem}
		
		\begin{pf}
			For convenience, set
			\begin{align*}
				\tilde{f}=f-\frac{\int_{0}^{1}f(x)h(x)\ud \mu}{\int_{0}^{1}h^2(x)\ud \mu}\,h(x).
			\end{align*}
			Then $\int_{0}^{1}\tilde{f}(x)h(x)\ud \mu=0$. Using the assumption $\int_{0}^{1}g(x)h(x)\ud \mu=0$, we have
			\begin{align*}
				\int_{0}^{1}f(x)g(x)\ud \mu
				=\int_{0}^{1}\tilde{f}(x)g(x)\ud \mu.
			\end{align*}
			Therefore, by the Cauchy--Schwarz inequality,
			\begin{align}\label{Cauchy--Schwarz inequality}
				\left(\int_{0}^{1}f(x)g(x)\ud \mu \right)^2
				=\left(\int_{0}^{1}\tilde{f}(x)g(x)\ud \mu \right)^2
				\leq \left(\int_{0}^{1}\tilde{f}^2\ud \mu\right)\left(\int_{0}^{1}g^2\ud \mu\right).
			\end{align}
			We now compute  $\int_{0}^{1}\tilde{f}^2\ud \mu$. Expanding $\tilde{f}$ and using $\int_{0}^{1}\tilde{f}(x)h(x)\ud \mu=0$, we have
			\begin{align}\nonumber
				\int_{0}^{1}\tilde{f}^2\ud \mu
				&=\int_{0}^{1}\left(f-\frac{\int_{0}^{1}f(x)h(x)\ud \mu}{\int_{0}^{1}h^2(x)\ud \mu}h(x)\right)^2\ud \mu\\
				\nonumber
				&=\int_{0}^{1}f^2\ud \mu
				-2\frac{\int_{0}^{1}f(x)h(x)\ud \mu}{\int_{0}^{1}h^2(x)\ud \mu}\int_{0}^{1}f(x)h(x)\ud \mu
				+\left(\frac{\int_{0}^{1}f(x)h(x)\ud \mu}{\int_{0}^{1}h^2(x)\ud \mu}\right)^2\int_{0}^{1}h^2(x)\ud \mu\\
				\label{a3.9}
				&=\int_{0}^{1}f^2\ud \mu-\frac{\left(\int_{0}^{1}f(x)h(x)\ud \mu\right)^2}{\int_{0}^{1}h^2(x)\ud \mu}.
			\end{align}
			Substituting (\ref{a3.9}) into (\ref{Cauchy--Schwarz inequality})
			yields \eqref{Basic lneq}. Moreover, equality holds in \eqref{Basic lneq} if $g$ is proportional to $\tilde{f}$; in particular, it holds   for $g=\tilde{f}$.
		\end{pf}
		
		We now state the spectral gap lemma for the dual operator.
		Its proof follows essentially the same pattern as that of Lemma \ref{Gap Lem}, and the derivation of the spectral gap inequality \eqref{Dual gap inequ} requires Lemma \ref{Basic Iequ Lem}.
		
		\begin{lem}\label{Gap Lem 2}
			Assume $\psi\in L^2(\B^{n},\tilde{\mathbf{1}}^{q'-2}\ud \xi)\cap \tilde{\mathcal{H}}^{\perp}$, where $\tilde{\mathcal{H}}=\mathrm{span}\{\tilde{\mathbf{1}}^{q'-1},\xi_1,\xi_2,\cdots,\xi_n\}$. Then
			\begin{align}\label{Dual gap inequ}				
				\sup_{0\not\equiv \psi\in L^2(\B^{n},\tilde{\mathbf{1}}^{q'-2}\ud \xi)\cap \tilde{\mathcal{H}}^{\perp}} \frac{\Tilde{d}^{p'-2}_{\alpha,\beta}\int_{\S^{n-1}}S^2_{\alpha,\beta}(\psi)}{\int_{\B^n}\psi^2\tilde{\mathbf{1}}^{q'-2} }<\frac{q'-1}{p'-1}\frac{\Tilde{d}^{p'}_{\alpha,\beta}|\S^{n-1}|}{|\B^n| }.
			\end{align}
		\end{lem}
		
		\begin{pf}
			For $l\in \mathbb{N}$, fix an orthonormal basis $\{Y_{l,k}\}$ of $\mathcal{H}_{l}$, where $k=1,2,\cdots,N_l:=\dim \mathcal{H}_{l}$. Expand $\psi$ as
			\begin{align*}
				\psi=\sum_{l=0}^{+\infty}\sum_{k=1}^{N_l}\psi_{l,k}(r)Y_{l,k},\qquad
				\psi_{l,k}(r)=\int_{\S^{n-1}}\psi\, Y_{l,k},
			\end{align*}
			so that
			\begin{align*}
				\int_{\B^n}\psi^2\tilde{\mathbf{1}}^{q'-2}
				=\sum_{l=0}^{+\infty}\sum_{k=1}^{N_l}\int_{0}^{1}\psi^2_{l,k}(r)\,r^{n-1}\tilde{\mathbf{1}}^{q'-2}(r)\ud r.
			\end{align*}	
			Applying  the Funk--Hecke formula \eqref{Funk-Hecke} yields 
			\begin{align*}
				S_{\alpha,\beta}(\psi )
				=\sum_{l=0}^{+\infty}\sum_{k=1}^{N_l}\left(\int_{0}^{1}\psi_{l,k}(r)w_l(r)\,r^{n-1}\ud r\right)Y_{l,k},
			\end{align*}
			where
			\begin{align*}
				w_{l}(r)=2^{-\beta}(4\pi)^{\frac{n-2}{2}}
				\frac{\Gamma\!\left(\frac{n-2}{2}\right)l!}{\Gamma(l+n-2)} (1-r^2)^{\beta}\int_{-1}^{1}(r^2+1-2rt)^{-\frac{n-\alpha}{2}}C^{\frac{n-2}{2}}_{l}(t)(1-t^2)^{\frac{n-3}{2}}\ud t.
			\end{align*}
			
			We first claim that
			\begin{align}\label{Const Lem equ-1}
				\int_{-1}^{1}\left(r^2+1-2rs\right)^{-\mu}(1-s^2)^{\nu-\frac{1}{2}}\ud s
				=B\Big(\frac{1}{2},\nu+\frac{1}{2}\Big)F(\mu,\mu-\nu;\mu+1;r^2).
			\end{align}

			To verify this, recall the transformation law  (cf.\ \cite[p.1018: 9.134-3]{Gradshteyn&Ryzhik})
			\begin{align*}
				F\Big(a,b;2b;\frac{4z}{(1+z)^{2}}\Big)=(1+z)^{2a}F\Big(a,a+\frac{1}{2}-b;b+\frac{1}{2};z^{2}\Big),
			\end{align*}
			together with the integral representation \eqref{Pre 1.8}. Setting $t=\frac{1-s}{2}$ and using \eqref{Gamma 2z}, the left-hand side of (\ref{Const Lem equ-1}) becomes
			\begin{align*}
				&2^{2\nu}(1-r)^{-2\mu}\int_{0}^{1}t^{\nu-\frac{1}{2}}(1-t)^{\nu-\frac{1}{2}}
				\left(1+\frac{4rt}{(1-r)^2}\right)^{-\mu}\ud t\\
				=&2^{2\nu}(1-r)^{-2\mu}B\Big(\nu+\frac{1}{2},\nu+\frac{1}{2}\Big)
				F\left(\mu,\nu+\frac{1}{2};2\nu+1;\frac{-4r}{(1-r)^2}\right)\\
				=&B\Big(\frac{1}{2},\nu+\frac{1}{2}\Big)F(\mu,\mu-\nu;\nu+1;r^2),
			\end{align*}
			which establishes \eqref{Const Lem equ-1}.

			Using this, we can simplify $w_l$ as follows
			\begin{align}\label{w_l claim}			w_l(r)
				=&\frac{2^{1-\beta}\pi^{n/2}\Gamma(\frac{n-\alpha}{2}+l)}{\Gamma(\frac{n-\alpha}{2})
					\Gamma(\frac{n}{2}+l)}
				(1-r^2)^{\beta+\alpha-1}r^lF\left(l+\frac{n+\alpha}{2}-1,\frac{\alpha}{2}; l+\frac{n}{2}; r^2\right).
			\end{align}
			The verification proceeds by integrating by parts
			$l$ times and employing  (\ref{Gegenbauer polynomial})  together with \eqref{Const Lem equ-1}: 
			\begin{align*}
				w_{l}(r)=&2^{-\beta}(4\pi)^{\frac{n-2}{2}}\frac{(-1)^l}{2^l}
				\frac{\Gamma\!\left(\frac{n-2}{2}\right)\Gamma\!\left(\frac{n-1}{2}\right)}{\Gamma(n-2)\Gamma\left(\frac{n-1}{2}+l\right)}
				(1-r^2)^{\beta}\int_{-1}^{1}(r^2+1-2rt)^{-\frac{n-\alpha}{2}}\frac{d^l}{dt^l}(1-t^2)^{\frac{n-3}{2}+l}\ud t\\
				=&2^{-\beta}(4\pi)^{\frac{n-2}{2}}
				\frac{\Gamma\!\left(\frac{n-2}{2}\right)\Gamma\!\left(\frac{n-1}{2}\right)\left(\frac{n-\alpha}{2}\right)_l}{\Gamma(n-2)\Gamma\left(\frac{n-1}{2}+l\right)}
				(1-r^2)^{\beta}r^l \int_{-1}^{1}(r^2+1-2rt)^{-\left(\frac{n-\alpha}{2}+l\right)}\left(1-t^2\right)^{\frac{n-3}{2}+l}\ud t\\
				=&2^{-\beta}(4\pi)^{\frac{n-2}{2}}
				\frac{\Gamma\!\left(\frac{n-2}{2}\right)\Gamma\!\left(\frac{n-1}{2}\right)\left(\frac{n-\alpha}{2}\right)_l}{\Gamma(n-2)\Gamma\left(\frac{n-1}{2}+l\right)}
				(1-r^2)^{\beta}r^l B\left(\frac{1}{2},\frac{n-1}{2}+l\right) F\left(\frac{n-\alpha}{2}+l,1-\frac{\alpha}{2};\frac{n}{2}+l;r^2\right)\\
				=&2^{-\beta}(4\pi)^{\frac{n-2}{2}}
				\frac{\Gamma(\frac{1}{2})\Gamma\!\left(\frac{n-2}{2}\right)\Gamma\!\left(\frac{n-1}{2}\right)\left(\frac{n-\alpha}{2}\right)_l}
				{\Gamma(n-2)\Gamma\left(\frac{n}{2}+l\right)}
				(1-r^2)^{\beta}r^l  F\left(\frac{n-\alpha}{2}+l,1-\frac{\alpha}{2};\frac{n}{2}+l;r^2\right)\\
				=&
				\frac{2^{1-\beta}\pi^{n/2}\Gamma(\frac{n-\alpha}{2}+l)}{\Gamma(\frac{n-\alpha}{2})
					\Gamma(\frac{n}{2}+l)}
				(1-r^2)^{\beta+\alpha-1}r^lF\left(\frac{\alpha}{2},l+\frac{n+\alpha}{2}-1; l+\frac{n}{2}; r^2\right),
			\end{align*}
			where the last equality follows from the duplication formula 	\begin{align*}
				\Gamma(n-2)=2^{n-3}\frac{\Gamma\left(\frac{n-2}{2}\right)\Gamma\left(\frac{n-1}{2}\right)}{\sqrt{\pi}}
			\end{align*}
			and transformation formula (\ref{a2.15}). Since $F\left(\frac{\alpha}{2},l+\frac{n+\alpha}{2}-1; l+\frac{n}{2}; r^2\right)=F\left(l+\frac{n+\alpha}{2}-1,\frac{\alpha}{2}; l+\frac{n}{2}; r^2\right)$, 
			this establishes \eqref{w_l claim}.   
			
			Since  $\int_{\B^{n}}\psi\tilde{\mathbf{1}}^{q'-1}=0$ and by Lemma \ref{Integral Identity Lem}, we have $\int_{\S^{n-1}}S_{\alpha,\beta}(\psi)=0$. Hence,
			\begin{align*}
				&\sup_{\substack{\int_{\mathbb{B}^n} \psi \tilde{\mathbf{1}}^{q'-1}=0 \\ 0\not\equiv\psi \in L^2(\mathbb{B}^n,\tilde{\mathbf{1}}^{q'-2}d\xi)}}
				\frac{\|S_{\alpha,\beta}(\psi)\|^2_{L^2(\S^{n-1})}}{\int_{\B^n}\psi^2\tilde{\mathbf{1}}^{q'-2}}\\
				=&\sup_{\substack{\int_{\mathbb{B}^n} \psi \tilde{\mathbf{1}}^{q'-1}=0 \\ 0\not\equiv\psi \in L^2(\mathbb{B}^n,\tilde{\mathbf{1}}^{q'-2}d\xi)}}\frac{\sum_{l=1}^{+\infty}\sum_{k=1}^{N_l}\left(\int_{0}^{1}\psi_{l,k}(r)w_l(r)\,r^{n-1}\ud r\right)^2}{\sum_{l=1}^{+\infty}\sum_{k=1}^{N_l}\int_{0}^{1}\psi^2_{l,k}(r)\,r^{n-1}\tilde{\mathbf{1}}^{q'-2}(r)\ud r}\\
				=&\sup_{l\geq 1}\left(\sup_{ 0\not\equiv\psi  \in L^2_{rad}(\B^n,\tilde{\mathbf{1}}^{q'-2}\ud \xi)}\frac{\left(\int_{0}^{1}\psi(r)w_l(r)\,r^{n-1}\ud r\right)^2}{\int_{0}^{1}\psi^2(r)\,r^{n-1}\tilde{\mathbf{1}}^{q'-2}(r)\ud r} \right),
			\end{align*}
			where
			\begin{align*}
				L^2_{rad}(\B^n,\tilde{\mathbf{1}}^{q'-2}\ud \xi)=\left\{f\in L^2(\B^n,\tilde{\mathbf{1}}^{q'-2}\ud \xi): f \; \textrm{is} \; \textrm{radial} \right\}.
			\end{align*}
			By  Lemma \ref{lem:hypergeo_monotonicity} and \eqref{w_l claim}, we have
			\begin{align}\label{Gap Lem 2 Claim}
				0< w_{l+1}(r)<w_{l}(r),\quad \forall~ r\in(0,1),\; \forall~  l\in \mathbb{N}.
			\end{align}
			Consequently,
			\begin{align}\nonumber
				\sup_{\substack{\int_{\mathbb{B}^n} \psi \tilde{\mathbf{1}}^{q'-1}=0 \\ 0\not\equiv\psi \in L^2(\mathbb{B}^n,\tilde{\mathbf{1}}^{q'-2}d\xi)}}\frac{\|S_{\alpha,\beta}(\psi)\|^2_{L^2(\S^{n-1})}}{\int_{\B^n}\psi^2\tilde{\mathbf{1}}^{q'-2}}=&\sup_{ 0\not\equiv\psi  \in L^2_{rad}(\B^n,\tilde{\mathbf{1}}^{q'-2}\ud \xi)}\frac{\left(\int_{0}^{1}\psi(r)w_1(r)\,r^{n-1}\ud r\right)^2}{\int_{0}^{1}\psi^2(r)\,r^{n-1}\tilde{\mathbf{1}}^{q'-2}(r)\ud r} \\
				\label{Gap Lem 2 Claim-0}
				=& \int_{0}^{1}w^2_1(r)\,r^{n-1}\tilde{\mathbf{1}}^{2-q'}(r)\ud r.
			\end{align}
			Combing  \eqref{Integral Identity-2} with  \eqref{Gap Lem 2 Claim-0} yields 
			\begin{align}\label{Gap Lem ineq}
				\int_{0}^{1}w^2_1(r)r^{n-1}\tilde{\mathbf{1}}^{2-q'}(r)\ud r \leq \frac{(q'-1)\tilde{d}^2_{\alpha,\beta}}{p'-1}\frac{|\S^{n-1}|}{|\B^n|}.
			\end{align}	
			
			On the other hand, for $\psi\in \tilde{\mathcal{H}}^{\perp}$, we have
			\begin{align*}
				\int_{0}^{1}\psi_{1,k}(r)r^{n}\ud r=0,\qquad k=1,2,\cdots,n.
			\end{align*}
			Set $\ud\mu=r^{n-1}\tilde{\mathbf{1}}^{q'-2}\ud r$, $h=r\tilde{\mathbf{1}}^{2-q'}$, $g=\psi_{1,k}(r), f=w_1(r)\tilde{\mathbf{1}}^{2-q'}$. Then 
			\begin{align*}
				\int_{0}^{1}g(r)h(r)\ud \mu(r)=0.
			\end{align*}
			Applying   Lemma \ref{Basic Iequ Lem} gives
			\begin{align}\nonumber
				&\sum_{k=1}^{n}\left(\int_{0}^{1}\psi_{1,k}(r)w_1(r)r^{n-1}\ud r\right)^2\\
				\label{a3.16}
				\leq&\left( \int_{0}^{1}w^2_1(r)r^{n-1}\tilde{\mathbf{1}}^{2-q'}\ud r-\frac{\left(\int_{0}^{1}w_1(r)r^{n}\tilde{\mathbf{1}}^{2-q'}\ud r\right)^2}{\int_{0}^{1}r^{n+1}\tilde{\mathbf{1}}^{2-q'}\ud r} \right)
				\sum_{k=1}^{n}\int_{0}^{1}\psi^2_{1,k}(r)r^{n-1}\tilde{\mathbf{1}}^{q'-2}\ud r.
			\end{align}
			Therefore, by (\ref{Gap Lem 2 Claim}) and (\ref{a3.16}), we have
			\begin{align}\nonumber
				&
				\sup_{0\not\equiv \psi\in L^2(\B^{n},\tilde{\mathbf{1}}^{q'-2}\ud \xi)\cap \tilde{\mathcal{H}}^{\perp}} \frac{\Tilde{d}^{p'-2}_{\alpha,\beta}\int_{\S^{n-1}}S^2_{\alpha,\beta}(\psi)}{\int_{\B^n}\psi^2\tilde{\mathbf{1}}^{q'-2} } \\
				=&
				\sup_{0\not\equiv \psi\in L^2(\B^{n},\tilde{\mathbf{1}}^{q'-2}\ud \xi)\cap \tilde{\mathcal{H}}^{\perp}} \frac{\sum_{l=1}^{+\infty}\sum_{k=1}^{N_l}\left(\int_{0}^{1}\psi_{l,k}(r)w_l(r)\,r^{n-1}\ud r\right)^2}{\sum_{l=1}^{+\infty}\sum_{k=1}^{N_l}\int_{0}^{1}\psi^2_{l,k}(r)\,r^{n-1}\tilde{\mathbf{1}}^{q'-2}(r)\ud r}\nonumber\\
				\label{Gap Lem ineq-1}
				\leq&
				\max\left\{\int_{0}^{1}w^2_2(r)r^{n-1}\tilde{\mathbf{1}}^{2-q'}\ud r, \;\;\int_{0}^{1}w^2_1(r)r^{n-1}\tilde{\mathbf{1}}^{2-q'}\ud r-\frac{\left(\int_{0}^{1}w_1(r)r^{n}\tilde{\mathbf{1}}^{2-q'}\ud r\right)^2}{\int_{0}^{1}r^{n+1}\tilde{\mathbf{1}}^{2-q'}\ud r}  \right\}.
			\end{align}
			Combining \eqref{Gap Lem 2 Claim}, \eqref{Gap Lem ineq}, and \eqref{Gap Lem ineq-1}, it remains to verify 
			\begin{align}\label{a3.21}
				\int_{0}^{1}w_1(r)r^{n}\tilde{\mathbf{1}}^{2-q'}\ud r\not =0.
			\end{align}
			This is immediate because, for every $r\in (0,1)$,
			\begin{align*}				
				w_1(r)=\begin{cases}
					\displaystyle C_1
					(1-r^2)^{\beta+\alpha-1}r^lF\left(l+\frac{n+\alpha}{2}-1,\frac{\alpha}{2}; l+\frac{n}{2}; r^2\right)>0,\quad&\mathrm{for}\quad 0\leq \alpha<n,\\[0.8em]
					\displaystyle C_1
					(1-r^2)^{\beta}r^lF\left(1-\frac{\alpha}{2},l+\frac{n-\alpha}{2}; l+\frac{n}{2}; r^2\right)>0, \quad&\mathrm{for}\quad \alpha<0,
				\end{cases}
			\end{align*}
			with the positive constant  $$C_1=\frac{2^{1-\beta}\pi^{n/2}\Gamma(\frac{n-\alpha}{2}+1)}{\Gamma(\frac{n-\alpha}{2})
				\Gamma(\frac{n}{2}+1)}>0.$$  Hence the integral in (\ref{a3.21}) is strictly positive.
		\end{pf}

		The following compactness lemma was essentially proved by Gluck~\cite{Gluck}. It can be viewed as a generalization of \cite[Corollary~2.2]{Hang&Wang&Yan}; see also \cite[Lemma~2.6]{Frank&Peteranderl&Read}.

		\begin{lem}[\protect{\cite[Corollary 2.2 and Lemma 2.5]{Gluck}}]\label{Sec 2 Compactness thm}			Assume that  $( \alpha,\beta)$ satisfies \eqref{Index Condi},  for $1< s<+\infty$ and $\frac{1}{t}>\frac{n-1}{n}\left(\frac{1}{s}-\frac{\alpha+\beta-1}{n-1}\right)$, then $Q_{\alpha,\beta}: L^s(\S^{n-1})\to L^t(\B^n)$ is compact. Especially, $Q_{\alpha,\beta}: L^2(\S^{n-1})\to L^2(\B^n)$ is compact.
		\end{lem}

		We now state an Orlicz-type compactness result. Using the compactness Lemma \ref{Sec 2 Compactness thm} and defining $\hat{S}_{\alpha,\beta}=S^{-1}_{\alpha,\beta}(\tilde{\mathbf{1}})S_{\alpha,\beta}$, we can obtain the first statement \eqref{Sec 3 Compactness thm-2 equ-a} by repeating the proof of \cite[Lemma 3.4]{Frank&Peteranderl&Read}; see also \cite[Lemma 3.4]{Figalli&Zhang}. The second statement \eqref{Sec 3 Compactness thm-2 equ-b} follows from Lemma \ref{Sec 2 Compactness thm} by an argument similar to that in Lemma \ref{Compact cor}. For the convenience of the reader, we provide the details below.

		\begin{lem}\label{Sec 3 Compactness thm-2}  
			Let $\psi_i$ be a sequence in $L^{q'}(\B^n)$ and $\e_i\in \R_{+}$ satisfy $\e_i\to 0$. If 
			\begin{align*}
				\int_{\B^n}\frac{\psi_i^2}{(\tilde{\mathbf{1}}+\e_i|\psi_i|)^{2-q'}}\leq 1,
			\end{align*}
			then there exists a subsequence $\psi_i \rightharpoonup\psi$ in $L^{q'}(\B^n)$ with $S_{\alpha,\beta}(\psi)\in L^2(\S^{n-1})$. Moreover, for any constant   $C>0$,  
			\begin{align}\label{Sec 3 Compactness thm-2 equ-a}
				\lim_{i\to+\infty}\int_{\S^{n-1}}\frac{(1+C\e_i|S_{\alpha,\beta}(\psi_i)|S^{-1}_{\alpha,\beta}(\tilde{\mathbf{1}}))^{p'}}{1+\e_i^2S^2_{\alpha,\beta}(\psi_i)S^{-2}_{\alpha,\beta}(\tilde{\mathbf{1}})}S^2_{\alpha,\beta}(\psi_i)=\int_{\S^{n-1}}S^2_{\alpha,\beta}(\psi) \quad\mathrm{for}\quad p^{'}\leq 2
			\end{align}
			and
			\begin{align}\label{Sec 3 Compactness thm-2 equ-b}
				\lim_{i\to+\infty}\int_{\S^{n-1}}S^2_{\alpha,\beta}(\psi_i)=\int_{\S^{n-1}}S^2_{\alpha,\beta}(\psi) \quad\mathrm{for}\quad p^{'}>2.
			\end{align}
		\end{lem}	
		\begin{pf}
			By the H\"older inequality, we can get
			\begin{align*}
				\int_{\B^n}|\psi_i|^{q'}\leq& \left(\int_{\B^{n}}(\tilde{\mathbf{1}}+\e_i|\psi_i|)^{q'-2}\psi_i^2\right)^{\frac{q'}{2}}\left(\int_{\B^{n}}(\tilde{\mathbf{1}}+\e_i|\psi_i|)^{q'}\right)^{1-\frac{q'}{2}}\\
				\lesssim&\left(1+\|\psi_i\|^{q'}_{L^{q'}(\B^n)}\right)^{1-\frac{q'}{2}},
			\end{align*}
			which implies $\|\psi_i\|_{L^{q'}(\B^n)}\leq C$. Hence we can assume that $\psi_i\rightharpoonup\psi$ in $L^{q'}(\B^n)$.
			
			Since $S_{\alpha,\beta}: L^{q'}(\B^n)\to L^{p'}(\S^{n-1})$ is bounded, we can get $S_{\alpha,\beta}(\psi_i)\rightharpoonup S_{\alpha,\beta}(\psi)$ in $L^{p'}(\S^{n-1})$.  Since $Q_{\alpha,\beta}: L^r(\S^{n-1})\to L^q(\B^n)$ is compact for $r>p$, then $S_{\alpha,\beta}: L^{q'}(\B^n)\to L^{\frac{r}{r-1}}(\S^{n-1})$ is also compact, which means $S_{\alpha,\beta}(\psi_i)\to  S_{\alpha,\beta}(\psi)$ in $L^{\frac{r}{r-1}}(\S^{n-1})$ for $r>p$. Especially, we can get $S_{\alpha,\beta}(\psi_i)$ converges to $S_{\alpha,\beta}(\psi)$ pointwisely.
			
			\textbf{Case 1:} $p'\leq 2$.	Set $\hat{S}_{\alpha,\beta}=\frac{1}{S_{\alpha,\beta}(\tilde{\mathbf{1}})}S_{\alpha,\beta}$.  We also introduce the two sets
			\begin{align*}
				A_i=\{\eta\in \S^{n-1}| \e_i\hat{S}_{\alpha,\beta}(\psi_i)(\eta)\leq 2\},\qquad A^{c}_i=\{\eta\in \S^{n-1}| \e_i\hat{S}_{\alpha,\beta}(\psi_i)(\eta)>2\}.
			\end{align*}
			Now, we begin to prove that 
			\begin{align}\label{Compact lem-2 equa 1}
				\lim_{i\to+\infty}\int_{A_i}\frac{(1+C\e_i|\hat{S}_{\alpha,\beta}(\psi_i)|)^{p'}}{1+\e^2_i\hat{S}^2_{\alpha,\beta}(\psi_i)}\hat{S}^2_{\alpha,\beta}(\psi_i)=\int_{\S^{n-1}}\hat{S}^2_{\alpha,\beta}(\psi)
			\end{align}
			and
			\begin{align}\label{Compact lem-2 equa 2}
				\lim_{i\to+\infty}\int_{A^{c}_i}\frac{(1+C\e_i|\hat{S}_{\alpha,\beta}(\psi_i)|)^{p'}}{1+\e^2_i
					\hat{S}^2_{\alpha,\beta}(\psi_i)}\hat{S}^2_{\alpha,\beta}(\psi_i)=0.
			\end{align}
			To achieve \eqref{Compact lem-2 equa 1}, we firstly claim that 
			\begin{align*}
				\int_{A_i}|\hat{S}_{\alpha,\beta}(\psi_i)|^{\frac{2p'}{q'}}\leq C.
			\end{align*}
			Since $|t|^{\frac{2}{q'}}(1+\e |t|)^{\frac{q'-2}{q'}}$ is convex on $\R$ and $\hat{S}_{\alpha,\beta}(\tilde{\mathbf{1}})=1$, then we can see
			\begin{align*}
				(1+\e_i|\hat{S}_{\alpha,\beta}(\psi_i)|)^{\frac{q'-2}{q'}}|\hat{S}_{\alpha,\beta}(\psi_i)|^{\frac{2}{q'}}\leq \hat{S}_{\alpha,\beta}\left((\tilde{\mathbf{1}}+\e_i|\psi_i|)^{\frac{q'-2}{q'}}|\psi_i|^{\frac{2}{q'}}\right).
			\end{align*}
			In $A_i$, there holds
			\begin{align*}
				1\leq 3^{\frac{2-q'}{q'}}(1+\e_i|\hat{S}_{\alpha,\beta}(\psi_i)|)^{\frac{q'-2}{q'}},
			\end{align*}
			then we can get \begin{align*}
				\int_{A_i}|\hat{S}_{\alpha,\beta}(\psi_i)|^{\frac{2p'}{q'}}\lesssim& \int_{A_i}(1+\e_i|\hat{S}_{\alpha,\beta}(\psi_i)|)^{\frac{(q'-2)p'}{q'}}|\hat{S}_{\alpha,\beta}(\psi_i)|^{\frac{2p'}{q'}}\\
				\lesssim&\int_{A_i}\hat{S}^{p'}_{\alpha,\beta}\left((\tilde{\mathbf{1}}+\e_i|\psi_i|)^{\frac{q'-2}{q'}}|\psi_i|^{\frac{2}{q'}}\right)\\
				\lesssim&\left(\int_{A_i}\left((\tilde{\mathbf{1}}+\e_i|\psi_i|)^{\frac{q'-2}{q'}}|\psi_i|^{\frac{2}{q'}}\right)^{q'}\right)^{\frac{p'}{q'}}.
			\end{align*}
			Since $|\hat{S}_{\alpha,\beta}(\psi_i)|^{\frac{2p'}{q'}}1_{A_i}\to |\hat{S}_{\alpha,\beta}(\psi)|^{\frac{2p'}{q'}}$ pointwisely, using the Fatou's Lemma, we can $S_{\alpha,\beta}(\psi)\in L^{\frac{2p'}{q'}}(\S^{n-1})$. Next, we claim that 
			\begin{align*}
				S_{\alpha,\beta}(\psi_i)1_{A_i}\to S_{\alpha,\beta}(\psi) \qquad\mathrm{in}\qquad L^s(\S^{n-1})\quad\mathrm{for}\quad s<\frac{2p'}{q'}.
			\end{align*}
			This is because we have the following interpolation estimate:
			\begin{align}\label{Compact lem-2 equa-3}
				&\|S_{\alpha,\beta}(\psi_i)1_{A_i}-S_{\alpha,\beta}(\psi)\|_{L^s(\S^{n-1})}\nonumber\\
				\leq& \|S_{\alpha,\beta}(\psi_i)1_{A_i}-S_{\alpha,\beta}(\psi)\|^{\theta}_{L^{\frac{2p'}{q'}}(\S^{n-1})}\|S_{\alpha,\beta}(\psi_i)1_{A_i}-S_{\alpha,\beta}(\psi)\|^{1-\theta}_{L^{1}(\S^{n-1})},
			\end{align}
			where $\theta\in(0,1)$ is chosen so that
			\begin{align*}
				s=\frac{2p'\theta}{q'}+(1-\theta).
			\end{align*}
			Return to the proof the \eqref{Compact lem-2 equa 1}, consider the function 
			\begin{align*}
				f_i=\frac{(1+C\e_i|\hat{S}_{\alpha,\beta}(\psi_i)|)^{p'}}{1+\e^2_i\hat{S}^2_{\alpha,\beta}(\psi_i)}\hat{S}^2_{\alpha,\beta}(\psi_i)1_{A_i}\to f=S^2_{\alpha,\beta}(\psi).
			\end{align*}
			Denote $m=\sup_{t\in [0,2]}\frac{(1+Ct)^{p'}}{1+t^2}$, then we can see $f_i\leq m S^2_{\alpha,\beta}(\psi_i)1_{A_i}:=g_i\to g=mS_{\alpha,\beta}^2(\psi)$, and using \eqref{Compact lem-2 equa-3}, we get $g_i\to g$ in $L^1(\S^{n-1})$. Therefore, the \eqref{Compact lem-2 equa 1} follows by dominated convergence theorem.

			To prove \eqref{Compact lem-2 equa 2}, we can estimate that 
			\begin{align*}
				\int_{A^{c}_i}\frac{(1+C\e_i|\hat{S}_{\alpha,\beta}(\psi_i)|)^{p'}}{1+\e^2_i
					\hat{S}^2_{\alpha,\beta}(\psi_i)}\hat{S}^2_{\alpha,\beta}(\psi_i)\lesssim& \int_{A^{c}_i}(1+C\e_i|\hat{S}_{\alpha,\beta}(\psi_i)|)^{p'-2}\hat{S}^2_{\alpha,\beta}(\psi_i)\\
				\lesssim&\e_i^{p'-2}\int_{A^{c}_i}|\hat{S}_{\alpha,\beta}(\psi_i)|^{p'}.
			\end{align*}
			We define 
			\begin{align*}
				\psi_{i,1}=\psi_i1_{\{\e_i|\psi_i|\leq \tilde{\mathbf{1}}\}} \qquad\mathrm{and}\qquad 	\psi_{i,2}=\psi_i1_{\{\e_i|\psi_i|>\tilde{\mathbf{1}}\}}.
			\end{align*}
			Since $\hat{S}_{\alpha,\beta}(\tilde{\mathbf{1}})=1$, we can get 
			\begin{align*}
				\e_i\|\hat{S}_{\alpha,\beta}(\psi_{i,1})\|_{L^{\infty}(\S^{n-1})}
				\leq 1.
			\end{align*}
			By the definition of $A_i^{c}$, we can get 
			\begin{align*}
				\e_i|\hat{S}_{\alpha,\beta}(\psi_{i,2})|\geq \e_i|\hat{S}_{\alpha,\beta}(\psi_{i})|-\e_i\|\hat{S}_{\alpha,\beta}(\psi_{i,1})\|_{L^{\infty}(\S^{n-1})}\geq 2-1=1,
			\end{align*}
			and 
			\begin{align*}
				\e_i|\hat{S}_{\alpha,\beta}(\psi_{i})|\leq \e_i|\hat{S}_{\alpha,\beta}(\psi_{i,1})|+\e_i|\hat{S}_{\alpha,\beta}(\psi_{i,2})|\leq 2\e_i|\hat{S}_{\alpha,\beta}(\psi_{i,2})|.
			\end{align*}
			For $\hat{S}_{\alpha,\beta}(\psi_{i,2})$, we can estimate that 
			\begin{align*}
				\e_i^{p'-2}\int_{A^{c}_i}|\hat{S}_{\alpha,\beta}(\psi_{i,2})|^{p'}\leq& \e_i^{p'-2}\left(\int_{\{\e_i|\psi_i|>\tilde{\mathbf{1}}\}}|\psi_i|^{q'}\right)^{\frac{p'}{q'}}\\
				\lesssim&\e_i^{\frac{2(p'-q')}{q'}}\left(\int_{\B^n}(\tilde{\mathbf{1}}+\e_i|\psi_i|)^{q'-2}|\psi_i|^{2}\right)^{\frac{p'}{q'}}\to 0.
			\end{align*}
			\textbf{Case 2}: $p'>2$.	Define
			\begin{align*}
				A_i=\{\xi\in  \B^{n}:\e_i|\psi_i(\xi)|<\tilde{\mathbf{1}}\} \qquad\mathrm{and}\qquad
				A^c_i=\{\xi\in  \B^{n}:\e_i|\psi_i(\xi)|\geq \tilde{\mathbf{1}}\}.
			\end{align*}
			Then
			\begin{align*}
				\int_{A_i}\psi^2_i\tilde{\mathbf{1}}^{q'-2}+\e_i^{q'-2}\int_{A^c_i}\psi^{q'}_i\lesssim 1.
			\end{align*}
			Since $\e_i^{-q'}|A_i^c|\leq \|\psi_i\|^{q'}_{L^{q'}(\B^{n})}\leq C$, we have $|A_i^c|\to 0$ as $i\to +\infty$.
			Indeed, writing $\psi_i\chi_{A_i}=\psi_i-\psi_i\chi_{A^c_i}$ and using $\psi_i\chi_{A^c_i}\to 0$ in $L^{q'}(\B^{n})$, we also have $\psi_i\chi_{A_i}\rightharpoonup \psi$ in $L^{q'}(\B^{n})$.
			
			Therefore, by the compactness lemma \ref{Sec 2 Compactness thm} and $p'>2$, we obtain $S_{\alpha,\beta}: L^{q'}(\B^n)\to L^{\frac{r}{r-1}}(\S^{n-1})$ is compact by choosing  $2>r>p$, then 
			\begin{align}\label{Compact cor equa-a-1}
				\lim_{i\to+\infty}\int_{\S^{n-1}}S^2_{\alpha,\beta}(\psi_i\chi_{A_i})
				=\int_{\S^{n-1}} S^2_{\alpha,\beta}(\psi).
			\end{align}
			On the other hand, 
			\begin{align*}
				\left|\int_{\S^{n-1}}S^2_{\alpha,\beta}(\psi_i\chi_{A^c_i})\right|
				\lesssim \|S_{\alpha,\beta}(\psi_i\chi_{A^c_i})\|^2_{L^{p'}(\S^{n-1})}
				\lesssim \|\psi_i\chi_{A^c_i}\|^2_{L^{q'}(\B^{n})}
				\lesssim \e_i^{2(2-q')/q'}\to 0,
			\end{align*}
			and similarly,
			\begin{align*}
				\left|\int_{\S^{n-1}}S_{\alpha,\beta}(\psi_i\chi_{A^c_i})S_{\alpha,\beta}(\psi_i\chi_{A_i}) \right|
				\lesssim \|S_{\alpha,\beta}(\psi_i\chi_{A^c_i})\|_{L^{p'}(\S^{n-1})}
				\lesssim \|\psi_i\chi_{A^c_i}\|_{L^{q'}(\B^{n})}
				\lesssim \e_i^{(2-q')/q'}\to 0.
			\end{align*}
			Consequently,
			\begin{align}\label{Compact cor equa-b-1}
				\lim_{i\to+\infty}\int_{\S^{n-1}}S^2_{\alpha,\beta}(\psi_i\chi_{A_i})-S^2_{\alpha,\beta}(\psi_i)=0.
			\end{align}
			Combining \eqref{Compact cor equa-a-1} and \eqref{Compact cor equa-b-1} yields the desired limit.
		\end{pf}

		\begin{pro}\label{Limit Prop-2}
			For $n\geq 3$ and for any $\lambda\in (0,K_{n,\alpha,\beta})$ and any $\gamma>0$, there exists $\delta=\delta(\lambda,\gamma,C)$ such that for any $\psi\in L^{q'}(\B^n)\cap\tilde{\mathcal{H}}^{\perp}$ with $\|\psi\|_{L^{q'}(\B^n)}<\delta$, the following inequalities hold:
			\begin{align*}
				\int_{\B^n}\psi^2\tilde{\mathbf{1}}^{q'-2}+&(q'-2)\int_{\B^n}	\tilde{\zeta}(\psi)(\tilde{\mathbf{1}}-|\tilde{\mathbf{1}}+\psi|)^2+\gamma\int_{\B^n}\min\{|\psi|^{q'},|\psi|^2\tilde{\mathbf{1}}^{q'-2}\}\nonumber\\
				\geq& (q'-1)(\Tilde{K}_{n,\alpha,\beta}-\lambda)\Tilde{d} ^{p'-2}_{\alpha,\beta} 	F_{p^{'}}(\psi),
			\end{align*}
			where
			\begin{align*}
				\tilde{\zeta}(a)=\begin{cases}
					\displaystyle \frac{|\tilde{\mathbf{1}}+a|\tilde{\mathbf{1}}^{q'-2}}{(2-q')|\tilde{\mathbf{1}}+a|+(q'-1)\tilde{\mathbf{1}}}\qquad&\mathrm{for}\qquad a\not\in [-2\times\tilde{\mathbf{1}},0]\\
					\displaystyle \tilde{\mathbf{1}}^{q'-2} \qquad&\mathrm{for}\qquad a\in [-2\times\tilde{\mathbf{1}},0].   
				\end{cases}
			\end{align*}
			and
			\begin{align*}
				F_{p^{'}}(\psi)=\begin{cases}
					\displaystyle   \int_{\S^{n-1}}\frac{(1+C|S_{\alpha,\beta}(\psi)|S^{-1}_{\alpha,\beta}(\tilde{\mathbf{1}}))^{p'}}{1 +S^2_{\alpha,\beta}(\psi)S^{-2}_{\alpha,\beta}(\tilde{\mathbf{1}})}S^2_{\alpha,\beta}(\psi) \quad&\mathrm{for}\quad p^{'}\leq 2;\\
					\displaystyle  \int_{\S^{n-1}}S^2_{\alpha,\beta}(\psi) \quad&\mathrm{for}\quad p^{'}>2.
				\end{cases}
			\end{align*}
		\end{pro}
		
		\begin{pf}

			We argue by contradiction. Assume there exist $\lambda_0,\gamma_0>0$ and a sequence $\{\psi_i\}\subset L^{q'}(\B^n)\cap \mathcal{H}^{\perp}$ with $\|\psi_i\|_{L^{q'}(\B^n)}\to 0$ such that
			\begin{align*}
				\int_{\B^n}\psi_i^2\tilde{\mathbf{1}}^{q'-2}+&(q'-2)\int_{\B^n}	\tilde{\zeta}(\psi_i)(\tilde{\mathbf{1}}-|\tilde{\mathbf{1}}+\psi_i|)^2+\gamma_0\int_{\B^n}\min\{|\psi_i|^{q'},|\psi_i|^2\tilde{\mathbf{1}}^{q'-2}\}\nonumber\\
				<& (q'-1)(\Tilde{K}_{n,\alpha,\beta}-\lambda_0)\Tilde{d} ^{p'-2}_{\alpha,\beta} 	F_{p^{'}}(\hat{\psi}_i).
			\end{align*}
			
			For any $\e>0$, there exists $\delta(\e)\in (0,1)$ such that
			on the set $I_{i}=\{|\psi_i|\leq \delta\times\tilde{\mathbf{1}} \}$
			we have $\tilde{\zeta}(\psi_i)\leq \left(1+\frac{q'-1}{2-q'}\e\right)\tilde{\mathbf{1}}^{q'-2}$; 
			hence
			\begin{align*}
				\psi_i^2\tilde{\mathbf{1}}^{q'-2}+(q'-2)	\tilde{\zeta}(\psi_i)(\tilde{\mathbf{1}}-|\psi_i+\tilde{\mathbf{1}}|)^2\geq (q'-1)(1-\e)\psi_i^2\tilde{\mathbf{1}}^{q'-2}.
			\end{align*}
			Define
			\begin{align*}
				\e_i=\left(\int_{\B^n}(\tilde{\mathbf{1}}+|\psi_i|)^{q'-2}\psi_i^2\right)^{1/2} \qquad\mathrm{and}\qquad \hat{\psi}_i=\e_i^{-1}\psi_i.
			\end{align*}
			It is straightforward that $\e^2_i\lesssim  \|\psi_i\|^{q'}_{L^{q'}(\B^n)}\to 0$ as $i\to+\infty$. Using Remark \ref{Rem 1}, we obtain
			\begin{align}\label{Sec 2 Gap Lem 2 equ-e}
				(1-\e)\int_{I_i}\hat{\psi}_i^2\tilde{\mathbf{1}}^{q'-2}+\frac{\gamma_0}{q'-1}( \e_i)^{q'-2}\delta^{2-q'}\int_{I_i^{c}} |\hat{\psi}_i|^{q'}
				<(\Tilde{K}_{n,\alpha,\beta}-\lambda_0)\tilde{d}^{p'-2}_{\alpha,\beta}	F_{p^{'}}(\hat{\psi}_i).
			\end{align}
			By definition,
			\begin{align*}
				\int_{\B^n}\frac{\hat{\psi}_i^2}{(\tilde{\mathbf{1}}+\e_i\hat{\psi}_i)^{2-q'}}=1.
			\end{align*}
			Using Lemma \ref{Sec 3 Compactness thm-2}, we conclude that $\hat{\psi}_i \rightharpoonup\hat{\psi}$ in $L^{q'}(\B^n)$ with $S_{\alpha,\beta}(\hat{\psi})\in L^2(\S^{n-1})$, and for any constant $C>0$,
			\begin{align*}
				\lim_{i\to+\infty}	F_{p^{'}}(\hat{\psi}_i)
				=\int_{\S^{n-1}}S^2_{\alpha,\beta}(\hat{\psi}).
			\end{align*}
			Therefore,
			\begin{align}\label{Sec 2 Gap Lem 2 equ-f}
				\limsup_{i\to+\infty}(1-\e)\int_{I_i}\hat{\psi}_i^2\tilde{\mathbf{1}}^{q'-2}+&\frac{\gamma_0}{q'-1}( \e_i)^{q'-2}\delta^{2-q'}\int_{I_i^{c}} |\hat{\psi}_i|^{q'}
				\leq (\Tilde{K}_{n,\alpha,\beta}-\lambda_0)\tilde{d}^{p'-2}_{\alpha,\beta}\int_{\S^{n-1}}S^2_{\alpha,\beta}(\hat{\psi}).
			\end{align}
			Moreover,
			\begin{align*}
				1=\int_{\B^n}(\tilde{\mathbf{1}}+|\psi_i|)^{q'-2}\hat{\psi}^2_i
				\lesssim \int_{I_i}|\hat{\psi}_i|^2\tilde{\mathbf{1}}^{q'-2}+\e_i^{q'-2}\int_{I_i^{c}}|\hat{\psi}_i|^{q'},
			\end{align*}
			which implies $\int_{\S^{n-1}}S^2_{\alpha,\beta}(\hat{\psi})\geq C_0>0$; hence $\hat{\psi}\not\equiv 0$. Using \eqref{Sec 2 Gap Lem 2 equ-f} and the definition of $I_i^{c}$, we can get $\e_i^{q'-2}\int_{I_i^{c}}\tilde{\mathbf{1}}^{q'}\leq C$, this leads to $|I_i^{c}|\to 0$. Using $\hat{\psi}_i\chi_{I_i}=\hat{\psi}_i-\hat{\psi}_i\chi_{I^{c}_i}$, we can get $\hat{\psi}_i\chi_{I_i} \rightharpoonup \hat{\psi}$ in $L^{q'}(\B^n)$.
			
			Clearly, $\hat{\psi}_i\chi_{I_i}\tilde{\mathbf{1}}^{q'/2-1}$ is uniformly bounded in $L^2(\B^n)$, we can assume that $\hat{\psi}_i\chi_{I_i}\tilde{\mathbf{1}}^{q'/2-1}\rightharpoonup \hat{\psi}'\tilde{\mathbf{1}}^{q'/2-1}$   We claim that $\hat{\psi}=\hat{\psi}'$ in $L^2(\B^n)$. For any $r<1$, $f\in C_{c}^{\infty}(\B_r)$, $f\tilde{\mathbf{1}}^{q'/2-1}\in C^{\infty}_{c}(\B_r)$, then
			\begin{align*}
				\int_{\B^n}f\hat{\psi}'\tilde{\mathbf{1}}^{q'/2-1}=\lim_{i\to+\infty}\int_{I_i}f\hat{\psi}_i\tilde{\mathbf{1}}^{q'/2-1}=\int_{\B^n}f\hat{\psi}\tilde{\mathbf{1}}^{q'/2-1}
			\end{align*}
			Therefore, we can see $\hat{\psi}=\hat{\psi}'$ in $\B_r$	for any $r\in (0,1)$, it means that $\hat{\psi}=\hat{\psi}'$ in $\B^n$ almost everwhere.
			Using the weak lower continuity of $L^2$ norm for \eqref{Sec 2 Gap Lem 2 equ-e}, we obtain
			\begin{align*}
				(1-\e)\int_{\B^n}\hat{\psi}^2\tilde{\mathbf{1}}^{q'-2}
				\leq(\Tilde{K}_{n,\alpha,\beta}-\lambda_0)\tilde{d}^{p'-2}_{\alpha,\beta}\int_{\S^{n-1}}S^2_{\alpha,\beta}(\hat{\psi})
			\end{align*}
			for any $\e>0$, which again yields a contradiction.
		\end{pf}

		\subsection{Proof of Theorem \ref{Thm 2}}\label{Sec 3.3} 	In this subsection we prove Theorem \ref{Thm 2}. Thanks to our previous preparations, the remaining argument follows the same general strategy as  the one used for the earlier case. The following orthogonality lemma can be obtained by a slight modification of \cite[Proposition~3.9]{Frank&Peteranderl&Read}, replacing $1$ with $\tilde{\mathbf{1}}$.
		\begin{lem}[\protect{\cite[Proposition 3.9]{Frank&Peteranderl&Read}}]\label{Dule Orothgonal Lem} For $1<q'< 2$, let $\{v_i\}_{i=1}^{+\infty}$ be a sequence satisfying 
			\begin{align*}
				\|v_i\|_{L^{q'}(\B^{n})}\to 1\qquad\mathrm{and}\qquad \inf_{\Psi}\||\B^{n}|^{\frac{1}{q'}}(v_i)_{\Psi}-\tilde{\mathbf{1}}\|_{L^{q'}(\B^{n})}\to 0.
			\end{align*}
			Then, there exists a sequence of conformal transformations $\Psi_i$, setting $r_i:=|\B^{n}|^{\frac{1}{q'}}(v_i)_{\Psi_i}-\tilde{\mathbf{1}}$,  such that $\|r_i\|_{L^{q'}(\B^{n})}\to 0$ as $i\to+\infty$ and 
			\begin{align*}
				\int_{\B^{n}}\xi_lr_i=0 \qquad\mathrm{for}\qquad l=1,2,\cdots,n.
			\end{align*}
		\end{lem}

		\begin{pro}\label{Dule Loc Bound Prop}
			For any sequence $\{v_i\}\subset L^{q'}(\B^{n})$ satisfying $\|v_i\|_{L^{q'}(\B^{n})}=1$ and
			\[
			\inf_{\Psi}\||\B^{n}|^{\frac{1}{q'}}(v_i)_{\Psi}-\tilde{\mathbf{1}}\|_{L^{q'}(\B^{n})}\to 0,
			\]
			there exists a constant $C_n>0$ such that
			\begin{align*}
				\liminf_{i\to+\infty}\frac{c^{q'}_{\alpha,\beta}-\|S_{\alpha,\beta}(v_i)\|^{q'}_{L^{p'}(\S^{n-1})}}{\inf_{\Psi}\| |\B^{n}|^{\frac{1}{q'}}(v_i)_{\Psi}-\tilde{\mathbf{1}}\|^{2}_{L^{q'}(\B^{n})}}\geq C_n.
			\end{align*}
		\end{pro}
		
		\begin{pf}
			Clearly, $1<q'<2$, while $p'$ may be larger than $2$, so we divide the proof into two cases.

			By Lemma \ref{Dule Orothgonal Lem}, there exists a sequence $\{\Psi_i\}$ such that, setting
			\[
			r_i:=|\B^{n}|^{\frac{1}{q'}}(v_i)_{\Psi_i}-\tilde{\mathbf{1}},
			\]
			we have $\|r_i\|_{L^{q'}(\B^{n})}\to 0$ as $i\to+\infty$ and
			\begin{align*}
				\int_{\B^{n}}\xi_lr_i=0 \qquad\mathrm{for}\qquad l=1,2,\cdots,n.
			\end{align*}
			Let $\alpha_i=|\B^{n}|^{\frac{1}{q'}-1}\int_{\B^{n}}(v_i)_{\Psi_i}\tilde{\mathbf{1}}^{q'-1}$. Then, using the H\"older inequality, we can get
			\begin{align*}
				|\B^{n}|^{\frac{1}{q'}}|\alpha_i-1|
				\leq \||\B^{n}|^{\frac{1}{q'}}(v_i)_{\Psi_i}-\tilde{\mathbf{1}}\|_{L^{q'}(\B^{n})},
			\end{align*}
			and 
			\begin{align*}
				|\B^{n}|^{\frac{1}{q'}}|\alpha_i-1|
				=|\B^{n}|^{\frac{1}{q'}}\big|\alpha_i-\|(v_i)_{\Psi_i}\|_{L^{q'}(\B^{n})}\big|
				\leq \||\B^{n}|^{\frac{1}{q'}}(v_i)_{\Psi_i}-\alpha_i\tilde{\mathbf{1}}\|_{L^{q'}(\B^{n})}.
			\end{align*}
			Consequently,
			\begin{align}\label{Prop vital equ-a-1}
				\frac{1}{2}\||\B^{n}|^{\frac{1}{q'}}(v_i)_{\Psi_i}-\tilde{\mathbf{1}}\|_{L^{q'}(\B^{n})}
				\leq\||\B^{n}|^{\frac{1}{q'}}(v_i)_{\Psi_i}-\alpha_i\tilde{\mathbf{1}}\|_{L^{q'}(\B^{n})}
				\leq 2 \||\B^{n}|^{\frac{1}{q'}}(v_i)_{\Psi_i}-\tilde{\mathbf{1}}\|_{L^{q'}(\B^{n})}.
			\end{align}
			Define the new normalization function
			\[
			\Tilde{r}_i=\frac{|\B^{n}|^{\frac{1}{q'}}}{\alpha_i}(v_i)_{\Psi_i}-\tilde{\mathbf{1}},
			\]
			noticing that $\tilde{\mathbf{1}}$ is a radial function, then $\int_{\B^n}\xi_l\tilde{\mathbf{1}}=0$ for $l=1,2,\cdots,n$, which means that  $\Tilde{r}_i\in \tilde{\mathcal{H}}^{\perp}$. And, 
			\begin{align*}
				|\B^{n}|\left[c^{q'}_{\alpha,\beta}-\|S_{\alpha,\beta}(v_i)\|^{q'}_{L^{p'}(\S^{n-1})}\right]
				=\alpha_i^{q'}\left[\|\tilde{\mathbf{1}}+\Tilde{r}_i\|^{q'}_{L^{q'}(\B^{n})}c^{q'}_{\alpha,\beta}-\|S_{\alpha,\beta}(\tilde{\mathbf{1}}+\Tilde{r}_i)\|^{q'}_{L^{p'}(\S^{n-1})}\right].
			\end{align*}
			Using the second inequality \eqref{Fundamental Ine 2} in Lemma \ref{Inequ lem}, we obtain
			\begin{align*}
				\|\tilde{\mathbf{1}}+\Tilde{r}_i\|^{q'}_{L^{q'}(\B^{n})}
				\geq& |\B^{n}|+q'\int_{\B^{n}}\tilde{\mathbf{1}}^{q'-1}\Tilde{r}_i+\frac{(1-\kappa)q'}{2}\int_{\B^{n}}\big(\tilde{\mathbf{1}}^{q'-2}\Tilde{r}_i^2+(q'-2)\tilde{\zeta}(\Tilde{r}_i)(\tilde{\mathbf{1}}-|\tilde{\mathbf{1}}+\Tilde{r}_i|)^2\big)\\
				&+c_{\kappa}\int_{\B^{n}}\min\{|\Tilde{r}_i|^{q'},|\Tilde{r}_i|^2\tilde{\mathbf{1}}^{q'-2}\}.
			\end{align*}
			
			\textbf{Case 1:} $\alpha>1$ ( i.e., $p'>2$). By the first inequality in Lemma \ref{Inequ lem-2}, we have
			\begin{align*}
				\|S_{\alpha,\beta}(\tilde{\mathbf{1}}+\Tilde{r}_i)\|^{p'}_{L^{p'}(\S^{n-1})}
				\leq& \Tilde{d}^{p'}_{\alpha,\beta}|\S^{n-1}|+p'\Tilde{d}^{p'-1}_{\alpha,\beta}\int_{\S^{n-1}}S_{\alpha,\beta}(\Tilde{r}_i)\\
				&+\left(\frac{p'(p'-1)}{2}+\kappa\right)\Tilde{d}^{p'-2}_{\alpha,\beta}\int_{\S^{n-1}}S^2_{\alpha,\beta}(\Tilde{r}_i)
				+C_{\kappa}\|S_{\alpha,\beta}(\Tilde{r}_i)\|^{p'}_{L^{p'}(\S^{n-1})}.
			\end{align*}
			Taking $v=\tilde{\mathbf{1}}$ in Gluck's theorem (see Theorem B \ref{thm B}) gives 
			\begin{align*}
				c_{\alpha,\beta}|\B^n|^{\frac{1}{q'}}=\Tilde{d}_{\alpha,\beta}|\S^{n-1}|^{\frac{1}{p'}}.
			\end{align*}
			Using $(1+a)^{\frac{q'}{p'}}\leq 1+\frac{q'}{p'}a$ for $a>-1$, we obtain
			\begin{align*}
				&c^{-q'}_{\alpha,\beta}|\B^n|^{-1}\|S_{\alpha,\beta}(\tilde{\mathbf{1}}+\Tilde{r}_i)\|^{q'}_{L^{p'}(\S^{n-1})}\\
				\leq& 1+\frac{q'}{\Tilde{d}_{\alpha,\beta}|\S^{n-1}|}\int_{\S^{n-1}}S_{\alpha,\beta}(\Tilde{r}_i)
				+\left(\frac{q'(p'-1)}{2}+\frac{q'\kappa}{p'}\right)\frac{\int_{\S^{n-1}}S^2_{\alpha,\beta}(\Tilde{r}_i)}{\Tilde{d}^2_{\alpha,\beta}|\S^{n-1}|}
				+C_{\kappa}\|S_{\alpha,\beta}(\Tilde{r}_i)\|^{p'}_{L^{p'}(\S^{n-1})}.
			\end{align*}			
			Applying \cite[Lemma 2.1(ii)]{Figalli&Zhang}, Proposition \ref{Limit Prop-2} and the identity \eqref{Integral Identity}, and arguing as in the derivation of the corresponding expansion, we obtain
			\begin{align*}
				&\|\tilde{\mathbf{1}}+\Tilde{r}_i\|^{q'}_{L^{q'}(\B^{n})}c^{q'}_{\alpha,\beta}-\|S_{\alpha,\beta}(\tilde{\mathbf{1}}+\Tilde{r}_i)\|^{q'}_{L^{p'}(\S^{n-1})}\\
				\geq& \frac{q'c^{q'}_{\alpha,\beta}}{2}\left[1-\frac{p'-1}{q'-1}\frac{|\B^{n}|}{\int_{\S^{n-1}}\Tilde{d}^{p'}_{\alpha,\beta}}\frac{1}{\Tilde{K}_{n,\alpha,\beta}-\lambda}+O(\kappa)\right]
				\int_{\B^{n}}\big(\tilde{\mathbf{1}}^{q'-2}\Tilde{r}_i^2+(q'-2)\tilde{\zeta}(\Tilde{r}_i)(\tilde{\mathbf{1}}-|\tilde{\mathbf{1}}+\Tilde{r}_i|)^2\big)\\
				&+c_{\kappa}\int_{\B^{n}}\min\{|\Tilde{r}_i|^{q'},|\Tilde{r}_i|^2\tilde{\mathbf{1}}^{q'-2}\}
				+O(\kappa)\int_{\S^{n-1}}\Tilde{d}^{p'-2}_{\alpha,\beta}S^2_{\alpha,\beta}(\Tilde{r}_i)
				-C_{\kappa}\|S_{\alpha,\beta}(\Tilde{r}_i)\|^{p'}_{L^{p'}(\S^{n-1})}.
			\end{align*}
			Notice that 
			\[
			\|S_{\alpha,\beta}(\Tilde{r}_i)\|^{p'}_{L^{p'}(\S^{n-1})}\lesssim \|\Tilde{r}_i\|^{p'}_{L^{q'}(\B^{n})}
			=o\big(\|\Tilde{r}_i\|^2_{L^{q'}(\B^{n})}\big),
			\]
			and
			\begin{align*}
				\int_{\B^{n}}\min\{|\Tilde{r}_i|^{q'},|\Tilde{r}_i|^2\tilde{\mathbf{1}}^{q'-2}\}
				\geq& \int_{|\Tilde{r}_i|\leq 1}\Tilde{r}_i^2\tilde{\mathbf{1}}^{q'-2}+\int_{|\Tilde{r}_i|> 1}\Tilde{r}_i^{q'}\\
				\geq& \left(\int_{|\Tilde{r}_i|\leq 1}\Tilde{r}_i^{q'}\right)^{\frac{2}{q'}}\left(\int_{|\Tilde{r}_i|\leq 1}\tilde{\mathbf{1}}^{q'}\right)^{\frac{q'-2}{q'}}
				+\int_{|\Tilde{r}_i|> 1}\Tilde{r}_i^{q'}
				\geq C \|\Tilde{r}_i\|^2_{L^{q'}(\B^{n})}.
			\end{align*}
			Choosing $\lambda$ and $\kappa$ sufficiently small, we obtain
			\begin{align*}
				\|\tilde{\mathbf{1}}+\Tilde{r}_i\|^{q'}_{L^{q'}(\B^{n})}c^{q'}_{\alpha,\beta}-\|S_{\alpha,\beta}(\tilde{\mathbf{1}}+\Tilde{r}_i)\|^{q'}_{L^{p'}(\S^{n-1})}
				\geq C\|\Tilde{r}_i\|^2_{L^{q'}(\B^{n})}+O(\kappa)\int_{\S^{n-1}}S^2_{\alpha,\beta}(\Tilde{r}_i).
			\end{align*}
			Since $p'>2$, we have
			\[
			\int_{\S^{n-1}}S^2_{\alpha,\beta}(\Tilde{r}_i)\lesssim\|S_{\alpha,\beta}(\Tilde{r}_i)\|^2_{L^{p'}(\S^{n-1})}\lesssim\|\Tilde{r}_i\|^2_{L^{q'}(\B^{n})}.
			\]
			In addition, by \eqref{Prop vital equ-a-1} there exist universal constants $C_1,C_2>0$ such that
			\[
			C_1\|r_i\|^2_{L^{q'}(\B^{n})}\leq \|\Tilde{r}_i\|^2_{L^{q'}(\B^{n})}\leq C_2\|r_i\|^2_{L^{q'}(\B^{n})}.
			\]
			This yields the desired lower bound in Case~1.\\
			
			\textbf{Case 2:} $\alpha\leq 1$ (i.e., $1<p'\leq 2$). By the second inequality in Lemma \ref{Inequ lem-2}, we have
			\begin{align*}
				\|S_{\alpha,\beta}(\tilde{\mathbf{1}}+\Tilde{r}_i)\|^{p'}_{L^{p'}(\S^{n-1})}
				\leq& \Tilde{d}^{p'}_{\alpha,\beta}|\S^{n-1}|+p'\Tilde{d}^{p'-1}_{\alpha,\beta}\int_{\S^{n-1}}S_{\alpha,\beta}(\Tilde{r}_i)\\
				&+\left(\frac{p'(p'-1)}{2}+\kappa\right)\int_{\S^{n-1}}\frac{\big(|\Tilde{d}_{\alpha,\beta}|+C_{\kappa}|S_{\alpha,\beta}(\Tilde{r}_i)|\big)^{p'}}{\Tilde{d}^2_{\alpha,\beta}+S^2_{\alpha,\beta}(\Tilde{r}_i)}S^2_{\alpha,\beta}(\Tilde{r}_i).
			\end{align*}
			Using $(1+a)^{\frac{q'}{p'}}\leq 1+\frac{q'}{p'}a$ for $a>-1$, we obtain
			\begin{align*}
				&c^{-q'}_{\alpha,\beta}|\B^n|^{-1}\|S_{\alpha,\beta}(\tilde{\mathbf{1}}+\Tilde{r}_i)\|^{q'}_{L^{p'}(\S^{n-1})}\\
				\leq& 1+\frac{q'}{\Tilde{d}_{\alpha,\beta}|\S^{n-1}|}\int_{\S^{n-1}}S_{\alpha,\beta}(\Tilde{r}_i)
				+\left(\frac{q'(p'-1)}{2}+\frac{q'\kappa}{p'}\right)
				\times\\
				&\frac{1}{\Tilde{d}^2_{\alpha,\beta}|\S^{n-1}|}\int_{\S^{n-1}}\frac{(1+C_{\kappa}|S_{\alpha,\beta}(\Tilde{r}_i)S^{-1}_{\alpha,\beta}(\tilde{\mathbf{1}})|)^{p'}}{1 +S^2_{\alpha,\beta}(\Tilde{r}_i)S^{-2}_{\alpha,\beta}(\tilde{\mathbf{1}})}S^2_{\alpha,\beta}(\Tilde{r}_i).
			\end{align*}
			Similarly, applying  \cite[Lemma 2.1(i)]{Figalli&Zhang} and Proposition \ref{Limit Prop-2}, we obtain
			\begin{align*}
				&\|\tilde{\mathbf{1}}+\Tilde{r}_i\|^{q'}_{L^{q'}(\B^{n})}c^{q'}_{\alpha,\beta}-\|S_{\alpha,\beta}(\tilde{\mathbf{1}}+\Tilde{r}_i)\|^{q'}_{L^{p'}(\S^{n-1})}\\
				\geq& \frac{(1-\kappa)q'c^{q'}_{\alpha,\beta}}{2}\int_{\B^{n}}\big(\tilde{\mathbf{1}}^{q'-2}\Tilde{r}_i^2+(q'-2)\tilde{\zeta}(\Tilde{r}_i)(\tilde{\mathbf{1}}-|\tilde{\mathbf{1}}+\Tilde{r}_i|)^2\big)
				+c_{\kappa}\int_{\B^{n}}\min\{|\Tilde{r}_i|^{q'},|\Tilde{r}_i|^2\tilde{\mathbf{1}}^{q'-2}\}\\
				&-c^{q'}_{\alpha,\beta}|\B^n|\left(\frac{q'(p'-1)}{2}+\frac{q'\kappa}{p'}\right)\frac{1}{\Tilde{d}^2_{\alpha,\beta}|\S^{n-1}|}
				\int_{\S^{n-1}}\frac{(1+C_{\kappa}|S_{\alpha,\beta}(\Tilde{r}_i)S^{-1}_{\alpha,\beta}(\tilde{\mathbf{1}})|)^{p'}}{1 +S^2_{\alpha,\beta}(\Tilde{r}_i)S^{-2}_{\alpha,\beta}(\tilde{\mathbf{1}})}S^2_{\alpha,\beta}(\Tilde{r}_i)\\
				\geq&\frac{q'c^{q'}_{\alpha,\beta}}{2}\left(1-\frac{(p'-1)|\B^n|}{(q'-1)(\tilde{K}_{n,\alpha,\beta}-\lambda)\tilde{d}_{\alpha,\beta}^{p'}|\S^{n-1}|}+O(\kappa)\right)
				\int_{\B^{n}}\big(\tilde{\mathbf{1}}^{q'-2}\Tilde{r}_i^2+(q'-2)\tilde{\zeta}(\Tilde{r}_i)(\tilde{\mathbf{1}}-|\tilde{\mathbf{1}}+\Tilde{r}_i|)^2\big)\\
				&+\frac{c_{\kappa}}{2}\int_{\B^{n}}\min\{|\Tilde{r}_i|^{q'},|\Tilde{r}_i|^2\tilde{\mathbf{1}}^{q'-2}\}.
			\end{align*}
			Choosing $\lambda$ and $\kappa$ sufficiently small and using Lemma \ref{Gap Lem 2}, we obtain
			\begin{align*}
				\|\tilde{\mathbf{1}}+\Tilde{r}_i\|^{q'}_{L^{q'}(\B^{n})}c^{q'}_{\alpha,\beta}-\|S_{\alpha,\beta}(\tilde{\mathbf{1}}+\Tilde{r}_i)\|^{q'}_{L^{p'}(\S^{n-1})}
				\geq\frac{c_{\kappa}}{2}\int_{\B^{n}}\min\{|\Tilde{r}_i|^{q'},|\Tilde{r}_i|^2\tilde{\mathbf{1}}^{q'-2}\}.
			\end{align*}
			The remaining arguments are analogous to Case~1 and are omitted.
		\end{pf}
		
		\textbf{Proof of Theorem \ref{Thm 2}.}
		Assume, for contradiction, that a sequence  $\{v_i\}\subset L^{q'}(\B^n)$ satisfies
		\begin{align*}
			\frac{c^{q'}_{\alpha,\beta}-\|S_{\alpha,\beta}v_i\|^{q'}_{L^{p'}(\S^{n-1})}/\|v_i\|^{q'}_{L^{q'}(\B^n)}}{\inf_{\Psi,\lambda}\|\lambda|\B^n|^{\frac{1}{q'}}(v_i)_{\Psi}-\tilde{\mathbf{1}}\|^{2}_{L^{q'}(\B^n)} } \to 0
		\end{align*}
		as $i\to+\infty$. Without loss of generality, we  assume $\|v_i\|_{L^{q'}(\B^n)}=1$. Since
		\begin{align*}
			\inf_{\Psi,\lambda}\|\lambda|\B^n|^{\frac{1}{q'}}(v_i)_{\Psi}-\tilde{\mathbf{1}}\|^{q'}_{L^{q'}(\B^n)}
			\leq \|\tilde{\mathbf{1}}\|^{q'}_{L^{q'}(\B^n)}=|\B^n|,
		\end{align*}
		we have $\|S_{\alpha,\beta}v_i\|^{q'}_{L^{p'}(\S^{n-1})}\to c^{q'}_{\alpha,\beta}$ as $i\to+\infty$. By Proposition \ref{Prop 1}, we obtain
		\begin{align*}
			\inf_{\Psi,\lambda\in\{\pm 1\}}\|\lambda|\B^n|^{\frac{1}{q'}}(v_i)_{\Psi}-\tilde{\mathbf{1}}\|_{L^{q'}(\B^n)}\to 0.
		\end{align*}
		Replacing $v_i$ by $-v_i$ if necessary, we may assume
		\[
		\inf_{\Psi}\||\B^n|^{\frac{1}{q'}}(v_i)_{\Psi}-\tilde{\mathbf{1}}\|_{L^{q'}(\B^n)}\to 0.
		\]
		Therefore,
		\begin{align*}
			\inf_{\Psi,\lambda}\|\lambda|\B^n|^{\frac{1}{q'}}(v_i)_{\Psi}-\tilde{\mathbf{1}}\|^{2}_{L^{q'}(\B^n)}
			\leq \inf_{\Psi}\||\B^n|^{\frac{1}{q'}}(v_i)_{\Psi}-\tilde{\mathbf{1}}\|^{2}_{L^{q'}(\B^n)}.
		\end{align*}
		This contradicts Proposition \ref{Dule Loc Bound Prop}.
		
		\medspace

		\textbf{Optimality of the exponent of distance functional}:
		We now verify that the exponents appearing in the distance functionals of Theorems \ref{Thm 1} and \ref{Thm 2} are optimal.
		The argument follows closely the construction given in \cite[Section 4]{Frank&Peteranderl&Read}; hence we only point out the necessary modifications relative to \cite[Section 4]{Frank&Peteranderl&Read}.

		(1).  Optimality for Theorem~\ref{Thm 1}.
		We treat the two possible values of the exponent separately.
		
		\begin{itemize}
			\item \textbf{The quadratic power of $L^p(\S^{n-1})$.}  
			Choose
			\[
			u_{\e}:=\lambda_{\e}\bigl(1+\e\varphi\bigr),
			\]
			where \(0\not\equiv\varphi\in C^{\infty}(\mathbb{S}^{n-1})\subset L^{p}(\mathbb{S}^{n-1})\) satisfies
			\begin{align}\label{Appendix euq-e}
				\int_{\mathbb{S}^{n-1}}\varphi=0,\qquad
				\int_{\mathbb{S}^{n-1}}\varphi\,\eta_{i}=0\quad(i=1,\dots,n),
			\end{align}
			and \(\lambda_{\e}>0\) is taken so that \(\|u_{\e}\|_{L^{p}(\S^{n-1})}=1\). Then we obtain
			$
			\lambda_{\e}=|\S^{n-1}|^{-1/p}(1+O(\e^2)).
			$
			Firstly, using \eqref{Appendix euq-e} and the similar proof as Lemma \ref{Gap Inequ}, we have
			\begin{align*}
				c^p_{\alpha,\beta}-
				\|Q_{\alpha,\beta}(u_{\e})\|^p_{L^q(\B^{n})}=&\frac{p(p-1)}{2} \e^2c^p_{\alpha,\beta}\left(\fint_{\S^{n-1}}\varphi^2-  \frac{q-1}{p-1}\,
				\frac{\int_{\B^n}d^{q-2}_{\alpha,\beta}Q^2_{\alpha,\beta}(\varphi)}{\int_{\B^n}d^q_{\alpha,\beta}}\right)+o(\e^2) \\
				\simeq& \e^2\fint_{\S^{n-1}}\varphi^2+o(\e^2).
			\end{align*}
			On the one hand, we easily obtain
			\begin{align*}
				\inf_{\Psi}\||\S^{n-1}|^{\frac{1}{p}}(u_\e)_{\Psi}-1\|_{L^p(\S^{n-1})} \lesssim \e.
			\end{align*}
			On the other hand, by the definition of infimum, there exists $\Psi_{\xi_{\e}}$ such that
			\begin{align}\label{Appendix euq-d}
				\||\S^{n-1}|^{\frac{1}{p}}(u_\e)_{\Psi_{\xi_{\e}}}-1\|_{L^p(\S^{n-1})}\leq 2\inf_{\Psi}\||\S^{n-1}|^{\frac{1}{p}}(u_\e)_{\Psi}-1\|_{L^p(\S^{n-1})}\lesssim \e,
			\end{align}
			which yields
			\begin{align*}
				\|1_{\Psi_{\xi_{\e}}}-1\|_{L^p(\S^{n-1})} -\e\|(\varphi)_{\Psi_{\xi_{\e}}}\|_{L^p(\S^{n-1})}  \lesssim\|(1+\e\varphi)_{\Psi_{\xi_{\e}}}-1\|_{L^p(\S^{n-1})}\lesssim \e.
			\end{align*}
			Hence, we obtain $\xi_{\e}\to 0$. Moreover, using \eqref{Appendix equ-a}, we get
			\begin{align}\label{Appendix euq-b}
				1_{\Psi_{\xi_{\e}}}(\eta)-1=(\theta_{\e}\cdot\eta+o(1))|\xi_{\e}|,\qquad \theta_{\e}=\frac{2(n-1)}{p}|\xi_{\e}|^{-1}\xi_{\e}.
			\end{align}
			Then we obtain $|\xi_{\e}|\lesssim\e$. Therefore, using the inequality in \cite[Lemma 2.1 (i)]{Figalli&Zhang}, we estimate
			\begin{align}\label{Appendix euq-c}
				\|(1+\e\varphi)_{\Psi_{\xi_{\e}}}-1\|^p_{L^p(\S^{n-1})}\geq& \e^p\|(\varphi)_{\Psi_{\xi_{\e}}}\|^p_{L^p(\S^{n-1})}-p\e\int_{\S^{n-1}}|1_{\Psi_{\xi_{\e}}}-1|^{p-2}(1_{\Psi_{\xi_{\e}}}-1)(\varphi)_{\Psi_{\xi_{\e}}}\nonumber\\
				\geq& \e^p\|\varphi\|^p_{L^p(\S^{n-1})}-o(\e^p).
			\end{align}
			For the estimate of the second term, it follows from
			\begin{align*}
				\int_{\S^{n-1}}|1_{\Psi_{\xi_{\e}}}-1|^{p-2}(1_{\Psi_{\xi_{\e}}}-1)(\varphi)_{\Psi_{\xi_{\e}}}=p\e|\xi_{\e}|^{p-1}\int_{\S^{n-1}}|\theta_{\e}\cdot\eta|^{p-2}\theta_{\e}\cdot\eta (\varphi)_{\Psi_{\xi_{\e}}}+o(\e^{p}).
			\end{align*}
			Moreover, without loss of generality, we can choose $\e_i$ such that $\theta_{\e_i}\to \theta$, where $|\theta|=2(n-1)/p$. Taking $\varphi=\eta_1\eta_2$ for $i\not=j$, we obtain
			\begin{align*}
				\lim_{\e\to 0}\int_{\S^{n-1}}|\theta_{\e}\cdot\eta|^{p-2}\theta_{\e}\cdot\eta (\varphi)_{\Psi_{\xi_{\e}}}
				=\int_{\S^{n-1}}|\theta\cdot\eta|^{p-2}\theta\cdot\eta \eta_1\eta_2=0.
			\end{align*}
			For the second equality, we can find an orthonormal matrix $A$ such that 
			$\theta = |\theta| A e_1$, where $e_1=(1,0,\cdots,0)$. Hence we find
			\begin{align*}
				\int_{\S^{n-1}}|\theta\cdot\eta|^{p-2}\theta\cdot\eta \eta_1\eta_2=&\int_{\S^{n-1}}|\theta\cdot A\eta|^{p-2}\theta\cdot A\eta (A\eta)_1(A\eta)_2\\
				=&\int_{\S^{n-1}}|\eta_1|^{p-2}\eta_1 (A\eta)_1(A\eta)_2,    \end{align*}
			but for all $1\leq i,j\leq n$, there also holds
			\begin{align*}
				\int_{\S^{n-1}}|\eta_1|^{p-2}\eta_1 \eta_i\eta_j=0.    \end{align*}
			Using \eqref{Appendix euq-d} and \eqref{Appendix euq-c}, we conclude that
			\begin{align*}
				\e\lesssim \inf_{\Psi}\||\S^{n-1}|^{\frac{1}{p}}(u_\e)_{\Psi}-1\|_{L^p(\S^{n-1})}\lesssim \e.
			\end{align*}
			It yields
			\begin{align}\label{Sec 4 euq-1}
				\lim_{i\to +\infty}\frac{c^p_{\alpha,\beta}-\|Q_{\alpha,\beta}(u_{\e_i})\|^p_{L^q(\B^{n})}}{\inf_{\Psi}\||\S^{n-1}|^{\frac{1}{p}}(u_{\e_i})_{\Psi}-1\|^2_{L^{p}(\S^{n-1})}}<+\infty.
			\end{align}
			Hence the exponent $2$ of the $L^p(\S^{n-1})$ distance cannot be decreased.

			\item \textbf{The power $p$ of $L^p(\S^{n-1})$.}
			Take a point \(\xi_{*}\in\mathbb{S}^{n-1}\) and let \(\xi\to\xi_{*}\). Set
			\[
			u_{\e,\xi}:=\lambda_{\e,\xi}\bigl(1+\e\,\left(\det \ud\Psi_{\xi}\big|_{\partial\B^n}\right)^{1/p}\bigr),
			\]
			where \(\lambda_{\e,\xi}>0\) is chosen so that \(\|u_{\e,\xi}\|_{L^{p}}=1\).
			Following the argument in \cite{Frank&Peteranderl&Read}, we obtain
			\begin{align}\label{Sec 4 euq-3}
				\liminf_{\e\to 0^+,\ \xi\to\xi_{*}}
				\||\S^{n-1}|^{\frac{1}{p}}(u_{\e,\xi})_{\Psi}-1\|_{L^p(\S^{n-1})}=0
			\end{align}
			and
			\begin{align}\label{Sec 4 euq-4}
				\limsup_{\e\to 0^+,\ \xi\to\xi_{*}}
				\frac{c^p_{\alpha,\beta}-
					\|Q_{\alpha,\beta}(u_{\e,\xi})\|^p_{L^q(\B^{n})}}
				{\||\S^{n-1}|^{\frac{1}{p}}(u_{\e,\xi})_{\Psi}-1\|^p_{L^{p}(\S^{n-1})}}
				<+\infty.
			\end{align}
			Thus the exponent $p$ of the $L^p(\S^{n-1})$ distance cannot be decreased. 
			For the power $2$ of $L^2(\S^{n-1})$, one may also follow the argument on page~30 of \cite{Frank&Peteranderl&Read}.
		\end{itemize}

		(2)   Optimality for Theorem~\ref{Thm 2}.
		For the dual setting we proceed analogously. Define
		\[
		v_{\e}:=\lambda_{\e}\bigl(\tilde{\mathbf{1}}+\e\psi\bigr), \qquad \psi=\xi_1\xi_2,
		\]
		with \(\lambda_{\e}>0\) chosen so that \(\|v_{\e}\|_{L^{q'}(\B^n)}=1\). Noticing that
		\begin{align*}
			\tilde{\mathbf{1}}_{\Psi_{\Psi_{\xi_{\e}}}}-\tilde{\mathbf{1}}
			=
			\left(\fint_{\B^n} d_{\alpha,\beta}^q \right)^{\frac{1-q}{q}}
			\left[(d^{q-1}_{\alpha,\beta})_{\Psi_{\xi_{\e}}}-d^{q-1}_{\alpha,\beta}\right],
		\end{align*}
		and that
		$
		(d^{q-1}_{\alpha,\beta})_{\Psi_{\xi_{\e}}}
		=
		\big((d_{\alpha,\beta})_{\Psi_{\xi_{\e}}}\big)^{q-1},
		$
		therefore, if $\Psi_{\xi_{\e}}$ tends to the identity map, we obtain
		\begin{align*}
			(d^{q-1}_{\alpha,\beta})_{\Psi_{\xi_{\e}}}-d^{q-1}_{\alpha,\beta}
			&=
			\big((d_{\alpha,\beta})_{\Psi_{\xi_{\e}}}\big)^{q-1}-d^{q-1}_{\alpha,\beta} 
			=
			d^{q-2}_{\alpha,\beta} \big((d_{\alpha,\beta})_{\Psi_{\xi_{\e}}}-d_{\alpha,\beta}\big)(q-1)(1+o(1))\\
			&=
			d^{q-2}_{\alpha,\beta}(q-1)(1+o(1))Q_{\alpha,\beta}(1_{\Psi_{\xi_{\e}}}-1).
		\end{align*}
		Using asymptotic formula \eqref{Appendix euq-b} and \eqref{Gap lem equ a}, we can get $$Q_{\alpha,\beta}(1_{\Psi_{\xi_{\e}}}-1)(\xi)=(\theta_{\e}\cdot\xi |\xi|^{-1}+o(1))d'_{\alpha,\beta}(|\xi|)|\xi_{\e}|,$$
		where
		\begin{align*}
			d'_{\alpha,\beta}(|\xi|)=\frac{\pi^{\frac{n}{2}}2^{-\beta}(n-\alpha)}{\Gamma(1+\frac{n}{2})}(1-|\xi|^2)^{\beta+\alpha-1}
			F\left(\frac{n+\alpha}{2},\frac{\alpha}{2}; 1+\frac{n}{2}; |\xi|^2\right)    \end{align*}
		Noticing that $d'_{\alpha,\beta}$ has the same asymptotic behavior as $d_{\alpha,\beta}$ described in Lemma~\ref{Bound d Lem}, applying the same analysis as before yields
		\[
		\lim_{i\to +\infty}
		\inf_{\Psi}
		\bigl\||\mathbb{B}^{n}|^{1/q'}(v_{\e_i})_{\Psi}-\tilde{\mathbf{1}}\bigr\|_{L^{q'}(\B^n)}^{2}=0,
		\qquad
		\lim_{i\to +\infty}
		\frac{c_{\alpha,\beta}^{\,q'}
			-\|S_{\alpha,\beta}(v_{\e_i})\|_{L^{p'}(\mathbb{S}^{n-1})}^{\,q'}}
		{\displaystyle\inf_{\Psi}
			\bigl\||\mathbb{B}^{n}|^{1/q'}(v_{\e_i})_{\Psi}-\tilde{\mathbf{1}}\bigr\|_{L^{q'}(\B^n)}^{2}}
		<+\infty.
		\]
		Therefore, the exponent \(2\) in the \(L^{q'}\)-distance in Theorem~\ref{Thm 2} is also optimal.
		
		\section{Appendix: conformal invariance}\label{Appendix}
		In this appendix, we verify the conformal invariance of the operators
		$E_{\alpha,\beta}$ and $Q_{\alpha,\beta}$. The argument follows essentially the same lines as that given by Frank, Peteranderl and Read  (see \cite{Frank&Peteranderl&Read}, Appendix A).
		
		\textbf{Conformal transformations in $\R^n_{+}$.}
		For $\lambda>0$ and $x_0'\in \R^{n-1}$, the nontrivial  conformal map in $\R^n_{+}$ is given by
		\begin{align*}
			\Phi_{\lambda,x_0'}(x',x_n)=\frac{1}{\lambda}\bigl(x'-x_0',\,x_n\bigr),
			\qquad (x',x_n)\in\R^n_{+}.
		\end{align*}
		Then $\det \ud\Phi_{\lambda,x_0'}=\lambda^{-n}$ and
		$\det \ud\Phi_{\lambda,x_0'}\big|_{\partial\R^n_{+}}=\lambda^{-(n-1)}$.
		For $f\in L^p(\R^{n-1})$ and $g\in L^q(\R^n_{+})$, define
		\begin{align*}
			f_{\Phi_{\lambda,x_0'}}(x')
			:=\lambda^{-(n-1)/p}\,f\!\left(\frac{x'-x_0'}{\lambda}\right),
			\qquad
			g_{\Phi_{\lambda,x_0'}}(x',x_n)
			:=\lambda^{-n/q}\,g\!\left(\frac{x'-x_0'}{\lambda},\frac{x_n}{\lambda}\right).
		\end{align*}
		When the parameters $(\lambda,x_0')$ are clear,  we simply write 
		$\Phi$.
		
		\begin{lem}\label{Appendix lem 1}
			For any conformal transformation $\Phi$ in $\R^n_{+}$ and any $f\in L^p(\R^{n-1})$, we have
			\begin{align*}
				E_{\alpha,\beta}(f_{\Phi})=\bigl(E_{\alpha,\beta}(f)\bigr)_{\Phi},
			\end{align*}
			and
			\begin{align*}
				\|f_{\Phi}\|_{L^p(\R^{n-1})}=\|f\|_{L^p(\R^{n-1})},
				\qquad
				\|E_{\alpha,\beta}(f_{\Phi})\|_{L^q(\R^n_{+})}
				=\|E_{\alpha,\beta}(f)\|_{L^q(\R^n_{+})}.
			\end{align*}
		\end{lem}
		
		\begin{pf}
			Without loss of generality, take  $\Phi=\Phi_{\lambda,x_0'}$.
			A direct computation gives
			\begin{align*}
				E_{\alpha,\beta}(f_{\Phi_{\lambda,x_0'}})(x',x_n)
				&=\int_{\R^{n-1}}
				\frac{x_n^{\beta}\,\lambda^{-(n-1)/p}\,f\!\left(\frac{y'-x_0'}{\lambda}\right)}
				{\bigl(x_n^2+|x'-y'|^2\bigr)^{\frac{n-\alpha}{2}}}\,\ud y' \\
				&=\lambda^{-n/q}\,
				E_{\alpha,\beta}(f)\!\left(\frac{x'-x_0'}{\lambda},\frac{x_n}{\lambda}\right),
			\end{align*}
			where we used the identity
			\[
			\alpha+\beta-1-\frac{n-1}{p}=-\frac{n}{q}.
			\]
			Hence,
			\[
			E_{\alpha,\beta}(f_{\Phi})
			=\bigl(E_{\alpha,\beta}(f)\bigr)_{\Phi}
			:=\lambda^{-n/q}\,E_{\alpha,\beta}(f)\!\left(\frac{x'-x_0'}{\lambda},\frac{x_n}{\lambda}\right).
			\]
			and the norm equalities  follow by a change of variables.
		\end{pf}
		
		Recall that the conformal map from the upper half-space $\R^n_{+}$ to the unit ball $\B^n$
		is given by
		\[
		I:(\R^n_{+},|\ud x|^2)\longrightarrow (\B^n,|\ud \xi|^2),
		\qquad
		\xi=I(x):=-e_n+\frac{2(x+e_n)}{|x+e_n|^2}.
		\]
		For $x=(x',x_n)\in\R^n_{+}$ and $y=(y',0)\in\partial\R^n_{+}$, set $\xi=I(x)$ and $\eta=I(y)$.
		One computes 
		\begin{align*}
			\xi
			&=\left(
			\frac{2x'}{(1+x_n)^2+|x'|^2},
			\frac{1-x_n^2-|x'|^2}{(1+x_n)^2+|x'|^2}
			\right),\\
			\eta
			&=\left(
			\frac{2y'}{1+|y'|^2},
			\frac{1-|y'|^2}{1+|y'|^2}
			\right).
		\end{align*}
		Consequently,
		\begin{align}\label{Sec 4.2 boundary equ b}
			|\xi+e_n|^2=\frac{4}{(1+x_n)^2+|x'|^2},
			\qquad
			|\eta+e_n|^2=\frac{4}{1+|y'|^2}.
		\end{align}
		Moreover, 
		\begin{align}\label{Sec 4.2 boundary equ c}
			2|\xi+e_n|^{-1}|\eta+e_n|^{-1}|\xi-\eta|
			=|x-y|,
			\qquad
			x_n=\frac{1-|\xi|^2}{|\xi+e_n|^2}.
		\end{align}
		Pulling back from $\R^n_{+}$ to $\B^n$ and using \eqref{Sec 4.2 boundary equ b}--\eqref{Sec 4.2 boundary equ c}, we obtain
		\begin{align*}
			E_{\alpha,\beta}(f)\circ I^{-1}(\xi)
			&=\int_{\S^{n-1}}
			\frac{
				\left(\frac{1-|\xi|^2}{|\xi+e_n|^2}\right)^{\beta}
				\,f\circ I^{-1}(\eta)
			}{
				\bigl(2|\xi+e_n|^{-1}|\eta+e_n|^{-1}|\xi-\eta|\bigr)^{n-\alpha}
			}
			\left(\frac{2}{|\eta+e_n|^2}\right)^{n-1}
			\ud V_{\S^{n-1}}(\eta).
		\end{align*}
		Therefore, if $u$ is defined by \eqref{Introdu u} and $Q_{\alpha,\beta}(u)$ is defined by \eqref{Introdu Q},
		then we obtain the desired integral representation
		\begin{align*}
			Q_{\alpha,\beta}(u)(\xi)
			=\frac{(1-|\xi|^2)^{\beta}}{2^{\beta}}
			\int_{\S^{n-1}} |\xi-\eta|^{\alpha-n}\,u(\eta)\,\ud V_{\S^{n-1}}(\eta).
		\end{align*}
		
		\textbf{Conformal transformations in $\B^n$.}
		Fix $\xi\in\B^n$ and define the (nontrivial) conformal map $\Psi_{\xi}:\B^n\to\B^n$ by
		\begin{align*}
			\Psi_{\xi}(\eta)
			=\frac{(1-|\xi|^2)(\eta-\xi)-|\eta-\xi|^2\,\xi}
			{1-2\xi\cdot\eta+|\xi|^2|\eta|^2}.
		\end{align*}
		Its restriction to $\S^{n-1}$ is a conformal transformation of $\S^{n-1}$  and
		\begin{align*}
			\Psi_{\xi}^{*}g_{\S^{n-1}}
			=\left(\frac{1-|\xi|^2}{1-2\xi\cdot \eta+|\xi|^2}\right)^2 g_{\S^{n-1}}.
		\end{align*}
		A straightforward computation implies that
		\begin{align}\label{Appendix equ-a}
			\det \ud\Psi_{\xi}(\eta)
			=\left(\frac{1-|\xi|^2}{1-2\xi\cdot \eta+|\xi|^2|\eta|^2}\right)^{n},
			\qquad
			\det \ud\Psi_{\xi}\big|_{\partial\B^n}(\eta)
			=\left(\frac{1-|\xi|^2}{1-2\xi\cdot \eta+|\xi|^2}\right)^{n-1}.
		\end{align}
		For simplicity, we write $\Psi$ in place of $\Psi_{\xi}$ when $\xi$ is clear from the context; more details   can be found in the appendix of \cite{Frank&Peteranderl&Read}.
		
		Applying Lemma~\ref{Appendix lem 1},  we immediately obtain the analogous conformal invariance for the operator $Q_{\alpha,\beta}$.
		\begin{lem}
			For any conformal transformation $\Psi$ in $\B^n$ and any $u\in L^p(\S^{n-1})$, we have
			\begin{align*}
				Q_{\alpha,\beta}(u_{\Psi})=\bigl(Q_{\alpha,\beta}(u)\bigr)_{\Psi},
			\end{align*}
			and
			\begin{align*}
				\|u_{\Psi}\|_{L^p(\S^{n-1})}=\|u\|_{L^p(\S^{n-1})},
				\qquad
				\|Q_{\alpha,\beta}(u_{\Psi})\|_{L^q(\B^n)}
				=\|Q_{\alpha,\beta}(u)\|_{L^q(\B^n)}.
			\end{align*}
		\end{lem}

		\noindent{\textbf{Acknowledgments.} Q. Yang was supported by the National Natural Science Foundation of China (No.12471056). S. Zhang was supported by the
			Postdoctoral Fellowship Program and China Postdoctoral Science Foundation under Grant Number (BX20250062), as well as the Shui Mu Tsinghua Scholar Program.}  		
		
		{\noindent\small{\bf Data availability:} Data sharing not applicable to this article as no datasets were generated or analysed during the current study.
			\section*{Declarations}
			{\noindent\small{\bf Conflict of interest:} On behalf of all authors, the corresponding author states that there is no conflict of interest.		
				
				\bibliographystyle{unsrt}

			\end{document}